%% file: main.tex
\documentclass[a4paper, 11pt, dvipsnames, svgnames]{article}

\usepackage[utf8]{inputenc}
\usepackage[top=1.0in, bottom=1.0in, left=1.0in, right=1.0in]{geometry}
\usepackage[all]{xy}
\usepackage{amsfonts}
\usepackage{amsmath}
\usepackage{amssymb}
\usepackage{amsthm}
\usepackage{calrsfs}
\usepackage{enumerate}
\usepackage{enumitem}
\usepackage{xparse, etoolbox}
\usepackage{ifthen}
\usepackage{mathtools, nccmath, textcomp}
\usepackage{relsize}
\usepackage{rsfso}
\usepackage{stmaryrd}
\usepackage{sectsty}
\usepackage[nottoc]{tocbibind}
\usepackage{tensor}
\usepackage{tikz}
\usepackage{tikz-cd}
\usepackage{titling}
\usepackage[sf, bf]{titlesec}	
\usepackage[titles]{tocloft}
\usepackage{xcolor}
\usepackage[backend=biber, style=alphabetic, sorting=nyt, doi=false, url=false, maxbibnames=99, maxalphanames=99]{biblatex}
\usepackage[bbgreekl]{mathbbol}

\renewbibmacro{in:}{}
\renewbibmacro*{issue+date}{%
  \ifboolexpr{test {\iffieldundef{year}}
              and test {\iffieldundef{issue}}}
    {}
    {\printtext[parens]{%
       \printfield{issue}%
       \setunit*{\addspace}%
       \usebibmacro{date}}}%
  \newunit}

\usepackage[T1]{fontenc}
\usepackage{lmodern}

\definecolor{bondiblue}{rgb}{0.1, 0.58, 0.71}
\usepackage{hyperref}
\hypersetup{
	colorlinks=true,
	linkcolor=WildStrawberry,
	citecolor=bondiblue,
	urlcolor=blue,
}

\usepackage{imakeidx}
\makeindex

\usepackage{fancyhdr}

\newcommand{\addperiod}[1]{#1.}
\titleformat{\section}[block]{\scshape\Large\filcenter}{\thesection.}{1em}{}
\titleformat{\subsection}[runin]{\normalfont\large\bfseries}{\thesubsection.}{1em}{\addperiod}
\titleformat{\subsubsection}[runin]{\normalfont\bfseries}{\thesubsubsection.}{1em}{\addperiod}

\numberwithin{equation}{section}

\DeclareSymbolFontAlphabet{\mathbbl}{bbold}
\newcommand{\Prism}{{\mathlarger{\mathbbl{\Delta}}}}

\makeatletter
\def\keywords{\xdef\@thefnmark{}\@footnotetext}
\makeatother

\makeatletter
\def\keywords{\xdef\@thefnmark{}\@footnotetext}
\makeatother

\theoremstyle{plain}
\newtheorem{thm}{Theorem}[section]
\newtheorem{cor}[thm]{Corollary}
\newtheorem{lem}[thm]{Lemma}
\newtheorem{prop}[thm]{Proposition}

\theoremstyle{definition}
\newtheorem{defi}[thm]{Definition}
\newtheorem{rem}[thm]{Remark}

\AtEndEnvironment{const}{\qed}

\theoremstyle{remark}

\newenvironment{enumarabicup}
{\begin{enumerate}[font=\upshape, labelindent=\parindent, label=(\arabic*)]}
{\end{enumerate}}

\DeclareMathOperator*{\colim}{\text{\scalebox{0.95}{$\textup{colim}$}}}

\DeclareMathAlphabet{\pazocal}{OMS}{zplm}{m}{n}
\DeclareMathAlphabet{\dutchcal}{U}{dutchcal}{m}{n}

\newcommand{\isomorphic}{\xrightarrow{\hspace{0.5mm} \sim \hspace{0.5mm}}}
\newcommand{\lisomorphic}{\xleftarrow{\hspace{0.5mm} \sim \hspace{0.5mm}}}

\newcommand{\pazm}{\pazocal{M}}

\newcommand{\pazo}{\pazocal{O}}

\newcommand{\pazs}{\pazocal{S}}

\def\CC{\mathbb{C}}

\def\NN{\mathbb{N}}
\def\QQ{\mathbb{Q}}

\def\ZZ{\mathbb{Z}}

\newcommand{\mbfa}{\mathbf{A}}
\newcommand{\mbfb}{\mathbf{B}}
\newcommand{\mbfd}{\mathbf{D}}

\newcommand{\mbfn}{\mathbf{N}}

\newcommand{\mbft}{\mathbf{T}}
\newcommand{\mbfv}{\mathbf{V}}

\newcommand{\smbfk}{\mathbf{k}}

\newcommand{\fraka}{\mathfrak{a}}

\newcommand{\frakm}{\mathfrak{m}}

\newcommand{\algebra}{\textrm{-algebra}}

\newcommand{\Ainf}{\mbfa_{\inf}}
\newcommand{\Binf}{\mbfb_{\inf}}

\newcommand{\Acrys}{\mbfa_{\textup{cris}}}
\newcommand{\OAcrys}{\pazo\mbfa_{\textup{cris}}}
\newcommand{\Atilde}{\tilde{\mbfa}}
\newcommand{\Btilde}{\tilde{\mbfb}}
\newcommand{\Dtilde}{\tilde{\mbfd}}

\newcommand{\AF}{\mbfa_F}

\newcommand{\DF}{\mbfd_F}
\newcommand{\NF}{\mbfn_F}

\newcommand{\ALpi}{\mbfa_{L,\varpi}}
\newcommand{\OALpi}{\pazo\mbfa_{L,\varpi}}

\newcommand{\OARpi}{\pazo\mbfa_{R,\varpi}}

\newcommand{\OBdR}{\pazo\mbfb_{\textup{dR}}}
\newcommand{\Bcrys}{\mbfb_{\textup{cris}}}
\newcommand{\OBcrys}{\pazo\mbfb_{\textup{cris}}}
\newcommand{\Adag}{\mbfa^{\dagger}}
\newcommand{\Bdag}{\mbfb^{\dagger}}
\newcommand{\Adagr}{\mbfa^{\dagger, r}}
\newcommand{\Bdagr}{\mbfb^{\dagger, r}}
\newcommand{\ALdag}{\mbfa_L^{\dagger}}
\newcommand{\BLdag}{\mbfb_L^{\dagger}}
\newcommand{\ALdagr}{\mbfa_L^{\dagger, r}}
\newcommand{\BLdagr}{\mbfb_L^{\dagger, r}}
\newcommand{\DLdag}{\mbfd_L^{\dagger}}
\newcommand{\DrigLdag}{\mbfd_{\textup{rig}, L}^{\dagger}}
\newcommand{\dldag}{D_L^{\dagger}}
\newcommand{\ALbrevedag}{\mbfa_{\breve{L}}^{\dagger}}
\newcommand{\BLbrevedag}{\mbfb_{\breve{L}}^{\dagger}}
\newcommand{\DLbrevedag}{\mbfd_{\breve{L}}^{\dagger}}
\newcommand{\DrigLbrevedag}{\mbfd_{\textup{rig}, \breve{L}}^{\dagger}}
\newcommand{\Atildedag}{\tilde{\mbfa}^{\dagger}}
\newcommand{\Btildedag}{\tilde{\mbfb}^{\dagger}}
\newcommand{\Atildedagr}{\tilde{\mbfa}^{\dagger, r}}
\newcommand{\Btildedagr}{\tilde{\mbfb}^{\dagger, r}}
\newcommand{\BLtildedag}{\tilde{\mbfb}_L^{\dagger}}
\newcommand{\BLtildedagr}{\tilde{\mbfb}_L^{\dagger, r}}
\newcommand{\BrigL}{\mbfb_{\textup{rig}, L}}
\newcommand{\BrigLbreve}{\mbfb_{\textup{rig}, \breve{L}}}
\newcommand{\BLbrevenhat}{\widehat{\mbfb}_{\breve{L}, n}}
\newcommand{\BLbrevekhat}{\widehat{\mbfb}_{\breve{L}, k}}
\newcommand{\BrigLdag}{\mbfb_{\textup{rig}, L}^{\dagger}}
\newcommand{\BrigLdagr}{\mbfb_{\textup{rig}, L}^{\dagger, r}}
\newcommand{\BrigLbrevedag}{\mbfb_{\textup{rig}, \breve{L}}^{\dagger}}
\newcommand{\BrigLtilde}{\tilde{\mbfb}_{\textup{rig}, L}}
\newcommand{\Brigtilde}{\tilde{\mbfb}_{\textup{rig}}}
\newcommand{\BrigLtildedag}{\tilde{\mbfb}_{\textup{rig}, L}^{\dagger}}
\newcommand{\Brigtildedag}{\tilde{\mbfb}_{\textup{rig}}^{\dagger}}
\newcommand{\BrigLtildedagr}{\tilde{\mbfb}_{\textup{rig}, L}^{\dagger, r}}
\newcommand{\Brigtildedagr}{\tilde{\mbfb}_{\textup{rig}}^{\dagger, r}}
\newcommand{\action}{\textrm{-action}}

\newcommand{\crys}{\textup{cris}}

\newcommand{\Dcrys}{\mbfd_{\textup{cris}}}
\newcommand{\ODcrys}{\pazo\mbfd_{\textup{cris}}}
\newcommand{\ODcrysL}{\pazo\mbfd_{\textup{cris}, L}}
\newcommand{\OVcrysL}{\pazo\mbfv_{\textup{cris}, L}}

\newcommand{\DcrysLbreve}{\mbfd_{\textup{cris}, \breve{L}}}

\newcommand{\dlog}{\hspace{0.2mm}d\textup{\hspace{0.3mm}log\hspace{0.3mm}}}

\newcommand{\equivariant}{\textrm{-equivariant}}
\newcommand{\etale}{\textup{\'et}}
\newcommand{\Eff}{\textup{Eff-}}
\newcommand{\Fil}{\textup{Fil}}

\newcommand{\Fr}{\textup{Frac}}

\newcommand{\Gal}{\textup{Gal}}
\newcommand{\GammaLbreve}{\Gamma_{\Lbreve}}

\newcommand{\Hom}{\textup{Hom}}

\newcommand{\jpluss}{j_{*}^+}
\newcommand{\Fbar}{\overline{F}}
\newcommand{\Abar}{\overline{A}}

\newcommand{\Kbar}{\overline{K}}

\newcommand{\kert}{\textup{Ker }}

\newcommand{\Lframe}{L^{\square}}
\newcommand{\OLframe}{O_{L^{\square}}}
\newcommand{\Lbreve}{\breve{L}}
\newcommand{\Lbreveinfty}{\breve{L}_{\infty}}
\newcommand{\Lbrevebar}{\overline{\breve{L}}}

\newcommand{\Lie}{\textup{Lie }}
\newcommand{\linear}{\textup{-linear}}

\newcommand{\lattice}{\textrm{-lattice}}
\newcommand{\module}{\textrm{-module}}
\newcommand{\modules}{\textrm{-modules}}
\newcommand{\mubar}{\overline{\mu}}

\newcommand{\MF}{\textup{MF}}
\newcommand{\ML}{\pazm_L}
\newcommand{\MLn}{\pazm_{L, n}}
\newcommand{\MLk}{\pazm_{L, k}}
\newcommand{\MLbreve}{\pazm_{\Lbreve}}
\newcommand{\Mdag}{M^{\dagger}}
\newcommand{\Mrigdag}{M_{\rig}^{\dagger}}
\newcommand{\Mrig}{M_{\rig}}

\newcommand{\Mod}{\textup{-Mod}}

\newcommand{\Finfty}{F_{\infty}}
\newcommand{\OFinfty}{O_{F_{\infty}}}
\newcommand{\Linfty}{L_{\infty}}
\newcommand{\OLinfty}{O_{L_{\infty}}}
\newcommand{\Lbar}{\overline{L}}
\newcommand{\OLbar}{O_{\overline{L}}}
\newcommand{\OLbreve}{O_{\breve{L}}}
\newcommand{\AL}{\mbfa_L}
\newcommand{\BL}{\mbfb_L}
\newcommand{\DL}{\mbfd_L}

\newcommand{\NL}{\mbfn_L}

\newcommand{\ALframe}{\mbfa_{L^{\square}}}
\newcommand{\ALbreve}{\mbfa_{\breve{L}}}
\newcommand{\BLbreve}{\mbfb_{\breve{L}}}
\newcommand{\DLbreve}{\mbfd_{\breve{L}}}
\newcommand{\NLbreve}{\mbfn_{\breve{L}}}
\newcommand{\NLbreveinfty}{\mbfn_{\breve{L}, \infty}}
\newcommand{\NrigL}{\mbfn_{\textup{rig}, L}}
\newcommand{\NrigLbreve}{\mbfn_{\textup{rig}, \breve{L}}}

\newcommand{\ONmPD}{\pazo N_m^{\PD}}
\newcommand{\SnPD}{S_n^{\PD}}
\newcommand{\OSnPD}{\pazo S_n^{\PD}}
\newcommand{\OSmPD}{\pazo S_m^{\PD}}
\newcommand{\SnhatPD}{\widehat{S}_n^{\PD}}

\newcommand{\padic}{p\textrm{-adic}}
\newcommand{\pqheight}{[p]_q\textrm{-height}}

\newcommand{\PD}{\textup{PD}}

\newcommand{\pinverse}{\big[\tfrac{1}{p}\big]}

\newcommand{\rank}{\textup{rk}}

\newcommand{\Rep}{\textup{Rep}}

\newcommand{\Rframe}{R^{\square}}

\newcommand{\representation}{\textrm{-representation}}
\newcommand{\rig}{\hspace{0.2mm}\textup{rig}}

\newcommand{\Spec}{\textup{Spec}\hspace{0.5mm}}
\newcommand{\Spf}{\textup{Spf }}

\newcommand{\submodule}{\textrm{-submodule}}

\newcommand{\TL}{\mbft_L}

\newcommand{\VL}{\mbfv_L}

\newcommand{\textmod}{\textup{ mod }}

\newcommand{\xitilde}{\tilde{\xi}}

\newcommand{\wa}{^{\textup{wa}}}

\title{\vspace{-8mm}\textsc{Crystalline representations and Wach modules in the imperfect residue field case}}
\author{\textsc{Abhinandan}}

\newcommand{\Addresses}{{
  \footnotesize

  \rule{2cm}{0.4pt}\vspace{2mm}

  \textsc{Abhinandan}\par\nopagebreak
  \textsc{IMJ-PRG, Sorbonne Universit\'e, 4 Place Jussieu, Paris 75252, France}\par\nopagebreak\vspace{-0.7mm}
  \textit{E-mail}: \footnotesize{abhinandan@imj-prg.fr}, \textit{Web}: \footnotesize{https://abhinandan.perso.math.cnrs.fr/}
}}

\date{}

\begin{document}

\pagenumbering{arabic}

\sloppy


\maketitle
{
	\vspace{-8mm}
	\textsc{Abstract.} For an absolutely unramified extension $L/\QQ_p$ with imperfect residue field, we define and study Wach modules in the setting of $(\varphi, \Gamma)\modules$ for $L$.
	Our main result establishes a direct equivalence between the category of lattices inside crystalline representations of the absolute Galois group of $L$ and the category of integral Wach modules for $L$.
	Moreover, we provide a direct relation between a rational Wach module equipped with the Nygaard filtration and the filtered $\varphi\module$ of its associated crystalline representation.
}

\input{imperfect_residue_field.tex}


\phantomsection
\printbibliography[heading=bibintoc, title={References}]

\Addresses

\end{document}

%% file: imperfect_residue_field.tex
\section{Introduction}

In classical $\padic$ Hodge theory, Fontaine introduced and developed the idea of studying a $\padic$ representation of the absolute Galois group of $\QQ_p$ (and its extensions) via semilinear algebraic objects attached to the representation.
More concretely, for an extension $F/\QQ_p$ with perfect residue field and absolute Galois group $G_F$, in \cite{fontaine-phigamma}, Fontaine showed that the category of $\ZZ_p\textrm{-representations}$ of $G_F$ is equivalent to the category of \'etale $(\varphi, \Gamma_F)\modules$, where $\Gamma_F$ is an open subgroup of $\ZZ_p^{\times}$ (see \S \ref{intro_subsec:arithmetic_case}).
On the other hand, to understand $\padic$ representations coming from geometry, Fontaine defined several classes of representations such as \textit{crystalline}, semistable, etc.\ in \cite{fontaine-annals}.
Putting the two point of views together, Fontaine asked the following natural question: is it possible to describe crystalline representations of $G_F$ in terms of $(\varphi, \Gamma_F)\modules$?
For an unramified extension $F/\QQ_p$, Fontaine studied this question in \cite{fontaine-phigamma}, and introduced the notion of finite crystalline-height representations (\textit{repr\'esentations de cr-hauteur finie}) of $G_F$, which was further developed by Wach \cite{wach-free, wach-torsion}, Colmez \cite{colmez-hauteur} and Berger \cite{berger-limites}.
More precisely, \cite{berger-limites} showed that the category of $G_F\textrm{-stable}$ $\ZZ_p\textrm{-lattices}$ of $\padic$ crystalline representations is equivalent to the category of \textit{Wach modules}, where a Wach module is a certain integral lattice inside the \'etale $(\varphi, \Gamma_F)\module$ associated to the representation (see \S \ref{intro_subsec:arithmetic_case}).

The two point of views of Fontaine admit natural generalisations to a relative base, i.e.\ formally étale algebras over a formal torus.
In particular, relative \'etale $(\varphi, \Gamma)\modules$ were studied by Andreatta \cite{andreatta-phigamma} and relative $\padic$ crystalline representations were studied by Faltings \cite{faltings-crystalline} and Brinon \cite{brinon-relatif}.
In \cite{abhinandan-relative-wach-i}, we introduced and studied the notion of relative Wach modules for an absolutely unramified (at $p$) relative base.
However, compared to the classical case, the results of \cite{abhinandan-relative-wach-i} are restrictive, i.e.\ we only show that relative Wach modules give rise to lattices inside relative crystalline representations; the converse is the following difficult \textit{open question: can one functorially associate a relative Wach module to a $\ZZ_p\textrm{-lattice}$ inside a relative crystalline representation?}

In this article, we resolve the \textit{open question} for the imperfect residue field case (see Theorem \ref{intro_thm:crystalline_wach_equivalence_imperfect}), and we use the result thus obtained, in a subsequent work \cite{abhinandan-relative-wach-ii}, to resolve the \textit{open question} in the relative case.
More concretely, for a complete discrete valuation field $L/\QQ_p$ with imperfect residue field, \cite{andreatta-phigamma} developed the theory of \'etale $(\varphi, \Gamma_L)\modules$, where $\Gamma_L$ is an open subgroup of $\ZZ_p(1)^d \rtimes \ZZ_p^{\times}$ with $d$ being the transcendence degree of $L/\QQ_p$, and \cite{brinon-imparfait} developed the theory of $\padic$ crystalline representations of $G_L$, the absolute Galois group of $L$.
However, for absolutely unramified $L/\QQ_p$, the theory of Wach modules for $L$ was missing from the picture.
So, in this article, we define Wach modules for $L$ and prove our \textit{first main result}:
\begin{thm}[{Corollary \ref{cor:crystalline_wach_equivalence_imperfect}}]\label{intro_thm:crystalline_wach_equivalence_imperfect}
	The category of $G_L\textrm{-stable}$ $\ZZ_p\textrm{-lattices}$ inside $\padic$ crystalline representations of $G_L$ is equivalent to the category of Wach modules for $L$.
\end{thm}

As mentioned above, the difficult part of Theorem \ref{intro_thm:crystalline_wach_equivalence_imperfect} is to functorially associate a Wach module to any $G_L\textrm{-stable}$ $\ZZ_p\lattice$ $T$ inside a $\padic$ crystalline representation of $G_L$.
To resolve this, let us note that using the classical theory of \cite{berger-limites} in the perfect residue field case, one can associate to $T$ a $\varphi\module$ $N$ over the base ring of Wach modules for $L$.
However, equipping $N$ with a natural action of $\Gamma_L$ is highly non-trivial, where the difficulty arises because $\Gamma_L$ is quite large compared to $\Gamma_F$ from the classical case.
The heart of this article constitutes a direct construction of the natural action of $\Gamma_L$ on $N$ (see \S \ref{intro_subsubsec:proof_sketch} for details).
Let us remark that the analogous theory of Breuil--Kisin modules in the imperfect residue field case was studied by Brinon and Trihan \cite{brinon-trihan}.
However, the theory of loc.\ cit.\ is different from the theory of Wach modules, in particular, the construction of the action of $\Gamma_L$ does not feature in \cite{brinon-trihan}.

Besides being natural generalisations of classical results to the relative case, the usefulness of relative Wach modules stems from its applications in the computation of $\padic$ vanishing cycles using syntomic complexes.
Indeed, to generalise the computation of $\padic$ vanishing cycles by Colmez and Nizio{\l} \cite{colmez-niziol} to the case of crystalline coefficients, in \cite{abhinandan-syntomic}, crucial inputs were the results on relative Wach modules from \cite{abhinandan-relative-wach-i}.
However, as mentioned above, the results of \cite{abhinandan-relative-wach-i}, and therefore, of \cite{abhinandan-syntomic} only work for a restrictive class of crystalline coefficients.
In order to generalise the results of \cite{colmez-niziol} to all crystalline coefficients, we need the more general result on relative Wach modules from \cite[Theorem 1.5]{abhinandan-relative-wach-ii}, for which Theorem \ref{intro_thm:crystalline_wach_equivalence_imperfect} is a crucial input.
Furthermore, in op.\ cit.\ we provide an interesting application of Theorem \ref{intro_thm:crystalline_wach_equivalence_imperfect}, in particular, we give a new criteria for checking the crystallinity of relative $\padic$ representations (see \cite[Theorem 1.7 \& Corollary 1.8]{abhinandan-relative-wach-ii}).

An additional motivation for considering Wach modules is to construct a \textit{deformation of the functor $\Dcrys$} from classical $\padic$ Hodge theory (see \cite[\S B.2.3]{fontaine-phigamma}).
This construction was carried out in the Fontaine-Laffaille range by Wach \cite[Theoreme 3]{wach-torsion}, and more generally, by Berger \cite[Th\'eor\`eme III.4.4]{berger-limites}.
In this article, our \textit{second main result} provides a generalisation of loc.\ cit.\ to the imperfect residue field case (see Theorem \ref{intro_thm:qdeformation_dcrys}).
Let us remark that the general idea of deformations of crystalline and de Rham cohomologies has led to exciting new developments in integral $\padic$ Hodge theory via the introduction and development of prismatic cohomology \cite{scholze-q-deformations, bhatt-morrow-scholze-1, bhatt-morrow-scholze-2, bhatt-scholze-prisms}.

Finally, note that recent developments in the theory of prismatic $F\textrm{-crystals}$ \cite{bhatt-scholze-crystals, du-liu-moon-shimizu, guo-reinecke} provide a new approach to the classification of lattices inside crystalline representations.
While the prismatic point of view is an exciting development, in the current paper, we employ techniques from the theory of $(\varphi, \Gamma)\modules$ to obtain our results.
This is due to the fact that, in our approach, the construction of Wach modules for $L$ and the proof of Theorem \ref{intro_thm:crystalline_wach_equivalence_imperfect} and Theorem \ref{intro_thm:qdeformation_dcrys}, are explicit and direct, which could be advantageous for ``arithmetic'' applications.
In \S \ref{subsubsec:relate_other_works}, we will provide more details on relations of our results in this article to other works.
In the rest of this section, we will describe the results mentioned above in more detail.
We begin by recalling the main classical result.

\subsection{The classical case}\label{intro_subsec:arithmetic_case}

Let $p$ be a fixed prime number and let $\kappa$ denote a perfect field of characteristic $p$; set $O_F := W(\kappa)$ to be the ring of $p\textrm{-typical}$ Witt vectors with coefficients in $\kappa$ and $F := \Fr(O_F)$.
Let $\overline{F}$ denote a fixed algebraic closure of $F$, let $\CC_p := \widehat{\overline{F}}$ denote the $\padic$ completion, and $G_F := \Gal(\overline{F}/F)$ the absolute Galois group of $F$.
Moreover, let $\Finfty := \cup_n F(\mu_{p^n})$ with $\Gamma_F := \Gal(F_{\infty}/F) \isomorphic \ZZ_p^{\times}$ and $H_F := \Gal(\overline{F}/\Finfty)$.
Furthermore, let $\Finfty^{\flat}$ denote the tilt of $\Finfty$ (see \S \ref{subsec:setup_nota}) and fix $\varepsilon := (1, \zeta_p, \zeta_{p^2}, \ldots)$ in $O_{\Finfty}^{\flat}$, and $\mu := [\varepsilon]-1$ and $[p]_q := \varphi(\mu)/\mu$ in $\Ainf(\OFinfty) := W(O_{\Finfty}^{\flat})$, the ring of $p\textrm{-typical}$ Witt vectors with coefficients in $O_{\Finfty}^{\flat}$.

In \cite{fontaine-phigamma}, Fontaine estalished a categorical equivalence between $\ZZ_p\textrm{-representation}$s of $G_F$ and étale $(\varphi, \Gamma_F)\textrm{-modules}$ over a certain period ring $\AF := O_F\llbracket \mu \rrbracket[1/\mu]^{\wedge} \subset W(\Finfty^{\flat})$, where ${}^{\wedge}$ denotes the $\padic$ completion, and $\AF$ is stable under the $(\varphi, \Gamma_F)\textrm{-action}$ on $W(\Finfty^{\flat})$.
For a fixed finite free $\ZZ_p\textrm{-representation}$ $T$ of $G_F$, the associated finite free \'etale $(\varphi, \Gamma_F)\module$ over $\AF$ is given as $\DF(T) := (\mbfa \otimes_{\ZZ_p} T)^{H_F}$, where $\mbfa \subset W(\CC_p^{\flat})$ is the maximal unramified extension of $\AF$ inside $W(\CC_p^{\flat})$.
In loc.\ cit., Fontaine conjectured that if $V := T[1/p]$ is crystalline then there exists a lattice inside $\DF(V) := \DF(T)[1/p]$ over which the action of $\Gamma_F$ admits a simpler form.
Denote by $\AF^+ := O_F\llbracket \mu \rrbracket \subset \AF$, which is stable under the $(\varphi, \Gamma_F)\action$, and note the following:
\begin{defi}\label{intro_defi:wach_mods}
	Let $a, b \in \ZZ$ with $b \geqslant a$.
	A \textit{Wach module} over $\AF^+$ with weights in the interval $[a, b]$ is a finite free $\AF^+\textrm{-module}$ $N$ equipped with a continuous and semilinear action of $\Gamma_F$ such that,
	\begin{enumarabicup}
		\item The action of $\Gamma_F$ on $N/\mu N$ is trivial.

		\item There is a Frobenius-semilinear operator $\varphi: N[1/\mu] \rightarrow N[1/\varphi(\mu)]$, commuting with the action of $\Gamma_F$, and such that $\varphi(\mu^b N) \subset \mu^b N$, the map $(1 \otimes \varphi) : \varphi^{\ast}(\mu^b N) := \AF^+ \otimes_{\varphi, \AF^+} \mu^b N \rightarrow \mu^b N$ is injective and its cokernel is killed by $[p]_q^{b-a}$.
	\end{enumarabicup}
\end{defi}

Denote the category of Wach modules over $\AF^+$ as $(\varphi, \Gamma_F)\Mod_{\AF^+}^{[p]_q}$, with morphisms between objects being $\AF^+\linear$, $\Gamma_F\equivariant$ and $\varphi\equivariant$ (after inverting $\mu$) morphisms.
Let $\Rep_{\ZZ_p}^{\crys}(G_F)$ denote the category of $\ZZ_p\textrm{-lattices}$ inside $\padic$ crystalline representations of $G_F$.
To any $T$ in $\Rep_{\ZZ_p}^{\crys}(G_F)$, using \cite{wach-free} and \cite{colmez-hauteur}, Berger functorially attaches a Wach module $\NF(T)$ over $\AF^+$ in \cite{berger-limites}.
The main result in the arithmetic case is as follows (see \cite{berger-limites}):
\begin{thm}\label{intro_thm:wach_crys_arith}
	The Wach module functor induces an equivalence of $\otimes\textrm{-categories}$:
	\begin{align*}
		\Rep_{\ZZ_p}^{\crys}(G_F) &\isomorphic (\varphi, \Gamma_F)\Mod_{\AF^+}^{[p]_q}\\
		T &\longmapsto \NF(T),
	\end{align*}
	with a quasi-inverse $\otimes\textrm{-functor}$ given as $N \mapsto \big(W(\CC_p^{\flat}) \otimes_{\AF^+} N\big)^{\varphi=1}$.
\end{thm}

\subsection{The imperfect residue field case}\label{intro_subsec:imperfect_case}

Let $d \in \NN$ and let $X_1, X_2, \ldots, X_d$ be indeterminates and let $\OLframe := (O_F[X_1^{\pm 1}, \ldots, X_d^{\pm d}]_{(p)})^{\wedge}$, where ${}^{\wedge}$ denotes the $\padic$ completion.
It is a complete discrete valuation ring with uniformiser $p$, imperfect residue field $\kappa(X_1, \ldots, X_d)$ and fraction field $\Lframe := \OLframe[1/p]$.
Let $O_L$ denote a finite \'etale extension of $\OLframe$ such that it is a complete discrete valuation ring with uniformiser $p$, imperfect residue field a finite \'etale extension of $\kappa(X_1, \ldots, X_d)$ and fraction field $L := O_L[1/p]$.
Let $G_L$ denote the absolute Galois group of $L$ for a fixed algebraic closure $\Lbar/L$; let $\Gamma_L \isomorphic \ZZ_p(1)^d \rtimes \ZZ_p^{\times}$ denote the Galois group of $\Linfty$ over $L$ where $\Linfty$ is the fraction field of $\OLinfty$ obtained by adjoining to $O_L$ all $p\textrm{-power}$ roots of unity and all $p\textrm{-power}$ roots of $X_i$ for all $1 \leqslant i \leqslant d$ (see \S \ref{sec:period_rings_padic_reps}).
In this setting, we have the theory of crystalline representations of $G_L$ \cite{brinon-imparfait} and \'etale $(\varphi, \Gamma)\textrm{-modules}$ \cite{andreatta-phigamma}.
However, the theory of Wach modules for $L$, i.e.\ a description of the $\padic$ crystalline representations $G_L$ in terms of $(\varphi, \Gamma_L)\modules$, was missing from the picture.
The main goal of this article is to complete this picture, which we discuss next.

\subsubsection{Wach modules}

For $1 \leqslant i \leqslant d$, let us set $X_i^{\flat} := (X_i, X_i^{1/p}, \ldots)$ in $\OLinfty^{\flat}$ and take $[X_i^{\flat}]$ in $\Ainf(\OLinfty) = W(\OLinfty^{\flat})$ to be the Teichm\"uller representative of $X_i^{\flat}$.
Let $\AL^+$ denote the unique finite \'etale extension (along the finite \'etale map $\OLframe \rightarrow O_L$) of the $(p, \mu)\textrm{-adic}$ completion of the localisation $O_F\llbracket \mu \rrbracket\big[[X_1^{\flat}]^{\pm 1}, \ldots, [X_d^{\flat}]^{\pm 1}\big]_{(p, \mu)}$.
The ring $\AL^+$ is equipped with a Frobenius endomorphism $\varphi$ and a continuous action of $\Gamma_L$ (see \S \ref{subsec:setup_nota} and \S \ref{subsec:period_rings_imperfect}).

\begin{defi}\label{intro_defi:wach_mod_imperfect}
	Let $a, b \in \ZZ$ with $b \geqslant a$.
	A \textit{Wach module} over $\AL^+$ with weights in the interval $[a, b]$ is a finite free $\AL^+\textrm{-module}$ $N$ equipped with a continuous and semilinear action of $\Gamma_L$ satisfying the following assumptions:
	\begin{enumarabicup}
		\item The action of $\Gamma_L$ on $N/\mu N$ is trivial.

		\item There is a Frobenius-semilinear operator $\varphi: N[1/\mu] \rightarrow N[1/\varphi(\mu)]$, commuting with the action of $\Gamma_L$, and such that $\varphi(\mu^b N) \subset \mu^b N$, the map $(1 \otimes \varphi) : \varphi^{\ast}(\mu^b N) = \AL^+ \otimes_{\varphi, \AL^+} \mu^b N \rightarrow \mu^b N$ is injective and its cokernel is killed by $[p]_q^{b-a}$.
	\end{enumarabicup}
	Say that $N$ is \textit{effective} if one can take $b = 0$ and $a \leqslant 0$.
	Denote the category of Wach modules over $\AL^+$ as $(\varphi, \Gamma)\Mod_{\AL^+}^{[p]_q}$, with morphisms between objects being $\AL^+\linear$, $\Gamma_L\equivariant$ and $\varphi\equivariant$ (after inverting $\mu$) morphisms.
\end{defi}

Set $\AL := \AL^+[1/\mu]^{\wedge}$ as the $\padic$ completion, equipped with a Frobenius endomorphism $\varphi$ and a continuous action of $\Gamma_L$.
Let $T$ be a finite free $\ZZ_p\module$ equipped with a continuous action of $G_L$, and note that one can functorially attach to $T$ a finite free \'etale $(\varphi, \Gamma_L)\module$ $\DL(T)$ over $\AL$ of rank $=\rank_{\ZZ_p} T$, equipped with a Frobenius-semilinear operator $\varphi$ and a semilinear and continuous action of $\Gamma_{L}$.
In fact, the preceding functor induces an equivalence between finite free $\ZZ_p\textrm{-representations}$ of $G_L$ and finite free \'etale $(\varphi, \Gamma_L)\modules$ over $\AL$ (see \S \ref{subsec:phigamma_mod_imperfect}).

\begin{rem}
	The category of Wach modules over $\AL^+$ can be realized as a full subcategory of \'etale $(\varphi, \Gamma)\modules$ over $\AL$ (see Proposition \ref{prop:wach_etale_ff_imperfect}).
\end{rem}

\subsubsection{Main results}

Let $\Rep_{\ZZ_p}^{\crys}(G_L)$ denote the category of $\ZZ_p\textrm{-lattices}$ inside $\padic$ crystalline representations of $G_L$.
The main result of this article, i.e.\ Theorem \ref{intro_thm:crystalline_wach_equivalence_imperfect}, can be stated more precisely as follows:
\begin{thm}[Corollary \ref{cor:crystalline_wach_equivalence_imperfect}]\label{intro_thm:crystalline_wach_equivalence}
	The Wach module functor induces an equivalence of $\otimes\textrm{-categories}$
	\begin{align*}
		\Rep_{\ZZ_p}^{\crys}(G_L) &\isomorphic (\varphi, \Gamma)\Mod_{\AL^+}^{[p]_q}\\
		T &\longmapsto \NL(T),
	\end{align*}
	with a quasi-inverse given as $N \mapsto \TL(N) := \big(W\big(\CC_L^{\flat}\big) \otimes_{\AL^+} N\big)^{\varphi=1}$, where $\CC_L := \widehat{\Lbar}$.
\end{thm}

Our strategy for the proof of Theorem \ref{intro_thm:crystalline_wach_equivalence} will be described in \S \ref{intro_subsubsec:proof_sketch}.

\begin{rem}
	In Theorem \ref{intro_thm:crystalline_wach_equivalence}, we do not expect the functor $\NL$ to be exact (see \cite[Example 7.1]{chang-diamond} for an example in the arithmetic case).
	However, after inverting $p$, the Wach module functor induces an exact equivalence of $\otimes\textrm{-categories}$ $\Rep_{\QQ_p}^{\crys}(G_L) \isomorphic (\varphi, \Gamma)\Mod_{\BL^+}^{[p]_q}$, via $V \mapsto \NL(V)$, with an exact quasi-inverse $\otimes\textrm{-functor}$ given as $M \mapsto \VL(M) := \big(W(\CC_L^{\flat}) \otimes_{\AL^+} M\big)^{\varphi=1}$ (see Corollary \ref{cor:crystalline_wach_rat_equivalence_imperfect}).
\end{rem}

As indicated earlier, the proof of Theorem \ref{intro_thm:crystalline_wach_equivalence} is based on techniques in the theory of $(\varphi, \Gamma)\modules$.
One of the advantages of using this approach is that it enables us to establish several comparison results between objects appearing in the $\padic$ Hodge theory over $L$ (see Proposition \ref{prop:oalpd_comparison}, Proposition \ref{prop:dcrys_brigtilde_comp}, Corollary \ref{cor:nrig_dag_comp} and Corollary \ref{cor:qdeformation_dcrys}).
In order to keep the introduction light, we only mention one of the comparison results here and refer the reader to the main body of this article for the rest.

Let $N$ be a Wach module over $\AL^+$.
We equip $N$ with a Nygaard filtration defined as $\Fil^k N := \{x \in N \textrm{ such that } \varphi(x) \in [p]_q^k N\}$.
Then, we note that $(N/\mu N)[1/p]$ is a $\varphi\module$ over $L$, since $[p]_q = p \textmod \mu \AL^+$, and $N/\mu N$ is equipped with a filtration $\Fil^k(N/\mu N)$ given as the image of $\Fil^k N$ under the surjection $N \twoheadrightarrow N/\mu N$.
We equip $(N/\mu N)[1/p]$ with the induced filtration, in particular, it is a filtered $\varphi\module$ over $L$.
Moreover, let $V := \TL(N)[1/p]$ denote the associated crystalline representation of $G_L$ from Theorem \ref{intro_thm:crystalline_wach_equivalence}.
Then, we can functorially associate to $V$ a filtered $(\varphi, \partial)\module$ over $L$ denoted $\ODcrysL(V)$ (see \S \ref{subsec:crysrep_imperfect}), and show the following:
\begin{thm}[Corollary \ref{cor:qdeformation_dcrys}]\label{intro_thm:qdeformation_dcrys}
	Let $N$ be a Wach module over $\AL^+$ and $V = \TL(N)[1/p]$ the associated crystalline representation from Theorem \ref{intro_thm:crystalline_wach_equivalence}.
	Then, we have a natural isomorphism $(N/\mu N)[1/p] \isomorphic \ODcrysL(V)$ as filtered $\varphi\modules$ over $L$.
\end{thm}

The proof of Theorem \ref{intro_thm:qdeformation_dcrys} is obtained by utilising the computations done in the proof of Theorem \ref{thm:fh_crys_imperfect}, more specifically, using Proposition \ref{prop:oalpd_comparison}.

\begin{rem}
	The statement of Theorem \ref{intro_thm:qdeformation_dcrys} is motivated by the results \cite[\S B.2.3]{fontaine-phigamma} and \cite[Th\'eor\`eme III.4.4]{berger-limites} in the perfect residue field case, but our proof is independent of those results.
	However, note that it is also possible to deduce that the isomorphism in Theorem \ref{intro_thm:qdeformation_dcrys} is compatible with filtrations, by using \cite[Th\'eor\`eme III.4.4]{berger-limites} as an input (see \cite[Remark 5.8]{abhinandan-relative-wach-ii}).
\end{rem}

\begin{rem}\label{intro_rem:wachmod_qconnection}
	Based on the expectation put forth in \cite[Remark 4.48]{abhinandan-relative-wach-i}, it is reasonable to expect that the $L\textrm{-vector}$ space $(N/\mu N)[1/p]$ may be equipped with a connection by defining a $q\textrm{-connection}$ on $N$ using the action of the geometric part of $\Gamma_L$, i.e.\ $\Gamma_L'$ (see \S \ref{sec:period_rings_padic_reps}), and inducing a connection via $N \xrightarrow{q \mapsto 1} N/\mu N$.
	Moreover, the isomorphism $(N/\mu N)[1/p] \isomorphic \ODcrysL(V)$ in Theorem \ref{intro_thm:qdeformation_dcrys} should be further compatible with connections.
	These expectations will be verified in \cite{abhinandan-relative-wach-ii}.
\end{rem}

\subsubsection{Strategy for the proof of Theorem \ref{intro_thm:crystalline_wach_equivalence}}\label{intro_subsubsec:proof_sketch}

To prove the theorem, starting with a $\ZZ_p\lattice$ $T$ inside a $\padic$ crystalline representation of $G_L$, we first use the result in the perfect residue field case (see Theorem \ref{intro_thm:wach_crys_arith}) and its compatibility with the results of \cite{kisin-modules, kisin-ren} (see \S \ref{subsec:kisin_module}) to construct a finite free module $\NrigL(V)$ (associated to $V = T[1/p]$), over the ring of functions of the open unit disk over $L$ (denoted $\BrigL^+$), such that $\NrigL(V)$ satisfies a Frobenius finite $\pqheight$ condition.
However, proving the existence of a non-trivial action of $\Gamma_L$ on $\NrigL(V)$ is a difficult question and it does not follow from the classical theory because $\Gamma_L \isomorphic \ZZ_p(1)^d \rtimes \ZZ_p^{\times}$, whereas we have $\Gamma_F \isomorphic \ZZ_p^{\times}$ in the classical case.
To resolve this issue, our innovation is to use the Galois action on $V$ and its crystallinity to explicitly show that $\NrigL(V)$ is equipped with an action of $\Gamma_L$ (see Proposition \ref{prop:nrigl_gammastab}).
Furthermore, we show that our construction is compatible with the theory of (overconvergent) \'etale $(\varphi, \Gamma_L)\modules$ from \cite{andreatta-phigamma, andreatta-brinon}, establishing the naturality of the action of $\Gamma_L$ on $\NrigL(V)$ (see \S \ref{subsec:compatibility_phigammmod}).
Next, we set $\NL(V) := \NrigL(V) \cap \DLdag(V) \subset \DrigLdag(V)$ as a module over $\BL^+ = \AL^+[1/p]$, where $\DLdag(V)$ is the overconvergent \'etale $(\varphi, \Gamma_L)\module$ associated to $V$ and $\DrigLdag(V)$ is the $(\varphi, \Gamma_L)\module$ over the Robba ring, of slope 0 and associated to $V$ (see \S \ref{subsec:phigamma_mod_imperfect} and Definition \ref{defi:nlv}).
Finally, we set $\NL(T) := \NL(V) \cap \DL(T) \subset \DL(V)$ as an $\AL^+\module$ and show that it satisfies the axioms of Definition \ref{intro_defi:wach_mod_imperfect} (see the proof of Theorem \ref{thm:crys_fh_imperfect} in \S \ref{subsec:obtaining_wachmod}).
In the opposite direction, starting with a Wach module $N$ over $\AL^+$, we use ideas developed in \cite{abhinandan-relative-wach-i} to show that $\TL(N)[1/p]$ is crystalline (see Theorem \ref{thm:fh_crys_imperfect}).

\subsubsection{Relation to other works}\label{subsubsec:relate_other_works}

Our first main result, Theorem \ref{intro_thm:crystalline_wach_equivalence}, is a direct generalisation of Theorem \ref{intro_thm:wach_crys_arith} from \cite{wach-free, colmez-hauteur, berger-limites}.
As indicated in \S \ref{intro_subsubsec:proof_sketch}, starting with a crystalline $\ZZ_p\textrm{-representation}$ $T$ of $G_L$, the construction of a finite $\pqheight$ module $\NL(T)$ uses classical Wach modules and its compatibility with the results of \cite{kisin-modules, kisin-ren}.
However, equipping $\NL(T)$ with a natural action of $\Gamma_L$ is highly non-trivial, in particular, it does not follow from previous works and constitutes the heart of this article.
For the converse, starting with a Wach module $N$ over $\AL^+$, we use ideas from \cite{abhinandan-relative-wach-i} to show that $\TL(N)[1/p]$ is crystalline.
Moreover, as mentioned earlier, the results on Wach modules in the current paper are different from the theory of Breuil--Kisin modules in the imperfect residue field case studied in \cite{brinon-trihan}.

Now, let us note that using the unpublished results of Tsuji in \cite{tsuji-crystalline} and the use of \cite{brinon-trihan} in \cite{du-liu-moon-shimizu}, it can be seen that the current paper is a crucial input to the construction of relative Wach modules in \cite{abhinandan-relative-wach-ii}.
Moreover, recent developments in the theory of prismatic $F\textrm{-crystals}$ \cite{bhatt-scholze-crystals, du-liu-moon-shimizu, guo-reinecke}, would suggest that there is a categorical equivalence between the category of Wach modules over $\AL^+$ and the category of prismatic $F\textrm{-crystals}$ on the absolute prismatic site $(\Spf O_L)_{\Prism}$.
At this point, let us remark that unlike the case of Breuil--Kisin modules from \cite{du-liu-moon-shimizu}, obtaining the aforementioned equivalence directly is a difficult question, in particular, it is highly non-trivial to directly show that the natural functor from prismatic $F\textrm{-crystals}$ to Wach modules is essentially surjective.
This point will be explored in another work \cite{abhinandan-prismatic-wach} and the current article is independent of the results in the prismatic theory.

As indicated previously, the motivation for interpreting a Wach module as a $q\textrm{-de Rham}$ complex and as $q\textrm{-deformation}$ of crystalline cohomology, i.e.\ $\ODcrys$, comes from \cite[\S B.2.3]{fontaine-phigamma} and \cite[Th\'eor\`eme III.4.4]{berger-limites}.
Our second main result, Theorem \ref{intro_thm:qdeformation_dcrys}, is an important step towards verifying such expectations.
In addition, we note that our proof of Theorem \ref{intro_thm:qdeformation_dcrys} is entirely independent to that of loc.\ cit., thus providing an alternative proof (as well as a generalisation) of the important classical result in loc.\ cit.
Furthermore, in Proposition \ref{prop:nrigl_gammastab} and Corollary \ref{cor:nrig_dag_comp} (see Remark \ref{rem:brigldag_odcrys_dldag_comp}), we generalise some results of \cite{berger-differentielles} and \cite{berger-limites} to obtain comparison results between Wach modules, overconvergent \'etale $(\varphi, \Gamma_L)\modules$ and filtered $(\varphi, \partial)\modules$ associated to $\padic$ crystalline representations.
In particular, for a $\padic$ crystalline representation $V$ of $G_L$, we prove a comparison isomorphism between the associated $(\varphi, \Gamma_L)\module$ over the Robba ring and the scalar extension of $\ODcrysL(V)$ to the Robba ring, where we use the connection on $\ODcrysL(V)$ to equip the scalar extension with an action of $\Gamma_L$ (see \S \ref{subsec:nriglv_galois_action} and Remark \ref{rem:brigldag_odcrys_dldag_comp}).

Finally, let us remark that using the theory of Breuil--Kisin modules in the imperfect residue field case from \cite{brinon-trihan}, in \cite{gao}, Gao studied lattices inside crystalline (more generally, semistable) representations using Breuil--Kisin $G_L\modules$.
However, the objects of loc.\ cit.\ are very different from Wach modules considered in this paper.
More specifically, Breuil--Kisin $G_L\modules$ are defined using the ``Kummer tower'' and admit an action of the big Galois group $G_L$.
In contrast, Wach modules are defined using the ``cyclotomic tower'', as in the theory of \'etale $(\varphi, \Gamma)\modules$, and admit an action of $\Gamma_L$, which is much smaller than $G_L$.
Moreover, \cite{gao} only proves a full faithfulness result, whereas Theorem \ref{intro_thm:crystalline_wach_equivalence} proves a categorical equivalence which was a difficult open question.

\subsection{Setup and notations}\label{subsec:setup_nota}

We will work under the convention that $0 \in \NN$, the set of natural numbers.
Let $p$ be a fixed prime number, $\kappa$ a perfect field of characteristic $p$, $O_F := W(\kappa)$ the ring of $p\textrm{-typical}$ Witt vectors with coefficients in $\kappa$ and $F := O_F[1/p]$, the fraction field of $W$.
In particular, $F$ is an unramified extension of $\QQ_p$ with ring of integers $O_F$.
Let $\overline{F}$ be a fixed algebraic closure of $F$ so that its residue field, denoted as $\overline{\kappa}$, is an algebraic closure of $\kappa$.
Furthermore, we denote by $G_F := \Gal(\overline{F}/F)$, the absolute Galois group of $F$.

We fix $d \in \NN$ and let $X_1, X_2, \ldots, X_d$ be some indeterminates.
Set $\Rframe$ to be $\padic$ completion of $O_F[X_1^{\pm 1}, \ldots, X_d^{\pm 1}]$ . 
Let $\varphi : \Rframe \rightarrow \Rframe$ denote a morphism extending the natural Frobenius on $O_F$ by setting $\varphi(X_i) = X_i^p$, for all $1 \leqslant i \leqslant d$.
The endomorphism $\varphi$ of $\Rframe$ is flat by \cite[Lemma 7.1.5]{brinon-relatif} and faithfully flat since $\varphi(\frakm) \subset \frakm$ for any maximal ideal $\frakm \subset \Rframe$.
Moreover, it is finite of degree $p^d$ using Nakayama Lemma and the fact that $\varphi$ modulo $p$ is evidently of degree $p^d$.
Let $\OLframe := (\Rframe_{(p)})^{\wedge}$, where ${}^{\wedge}$ denotes the $\padic$ completion.
It is a complete discrete valuation ring with uniformiser $p$, imperfect residue field $\kappa(X_1, \ldots, X_d)$ and fraction field $\Lframe := \OLframe[1/p]$.
The Frobenius on $\Rframe$ extends to a unique finite and faithfuly flat of degree $p^d$ Frobenius endomorphism $\varphi : \OLframe \rightarrow \OLframe$, lifting the absolute Frobenius on $\OLframe/p\OLframe$.

Let $O_L$ denote a finite \'etale extension of $\OLframe$ such that it is a domain.
Then $O_L$ is a complete discrete valuation ring with uniformiser $p$, imperfect residue field a finite \'etale extension of $\kappa(X_1, \ldots, X_d)$ and fraction field $L := O_L[1/p]$.
Fix an algebraic closure $\Lbar/L$ and let $G_L := \Gal(\Lbar/L)$ denote the absolute Galois group.
The Frobenius on $\OLframe$ extends to a unique finite and faithfuly flat of degree $p^d$ Frobenius endomorphism $\varphi : O_L \rightarrow O_L$ lifting the absolute Frobenius on $O_L/pO_L$ (see \cite[Proposition 2.1]{colmez-niziol}).
For $k \in \NN$, let $\Omega^k_{O_L}$ denote the $\padic$ completion of module of $k\textrm{-differentials}$ of $O_L$ relative to $\ZZ$.
Then, we have that $\Omega^1_{O_L} = \oplus_{i=1}^d O_L \dlog X_i$ and $\Omega^k_{O_L} = \wedge_{O_L}^k \Omega^1_{O_L}$.

Next, let $K$ be one of $\Finfty$, $\Linfty$, $\Fbar$ or $\Lbar$, where $\Finfty := F(\mu_{p^{\infty}})$ and $\Linfty := \cup_{i=1}^d L(\mu_{p^{\infty}}, X_i^{1/p^{\infty}})$, and set $O_K$ as the ring of integers of $K$.
Then, the tilt of $O_K$ is defined as $O_K^{\flat} := \lim_{\varphi} O_K/p$, and the tilt of $K$ is defined as $K^{\flat} := \Fr(O_K^{\flat})$ (see \cite[Chapitre V, \S 1.4]{fontaine-pdivisibles}).
Finally, let $A$ be a $\ZZ_p\algebra$ equipped with a Frobenius endomorphism $\varphi$ lifting the absolute Frobenius on $A/pA$, then for any $A\module$ $M$ we write $\varphi^*(M) := A \otimes_{\varphi, A} M$.

\subsection{Outline of the paper}

This article consists of three main sections.
In \S \ref{sec:period_rings_padic_reps} we collect relevant results on $\padic$ Hodge theory in the imperfect residue field case.
In \S \ref{subsec:period_rings_imperfect} we define several period rings, in particular, we recall crystalline period rings, $(\varphi, \Gamma)\module$ theory rings, overconvergent rings and Robba rings and prove several important technical results to be used in our main proofs in \S \ref{sec:crystalline_finite_height}.
In \S \ref{subsec:phigamma_mod_imperfect} we quickly recall the relation between $\padic$ representations and $(\varphi, \Gamma)\module$ theory over the period rings described in the previous section.
In \S \ref{subsec:crysrep_imperfect} we focus on crystalline representations and prove some results relating Galois action on a crystalline representation to its associated filtered $(\varphi, \partial)\module$.
The goal of \S \ref{sec:wach_modules} is to define Wach modules in the imperfect residue field case and study the associated $\ZZ_p\textrm{-representations}$ of $G_L$.
In \S \ref{subsec:wach_mod_props} we give the definition of Wach modules and relate it to \'etale $(\varphi, \Gamma)\modules$ (see Proposition \ref{prop:wach_etale_ff_imperfect}).
Then given a Wach module, we functorially attach to it a $\ZZ_p\representation$ of $G_L$ and in \S \ref{subsec:finite_pqheight_reps} we show that these are related to finite $\pqheight$ representations studied in \cite{abhinandan-relative-wach-i}.
Finally, in \S \ref{subsec:wachmod_crystalline} we show that the $\ZZ_p\representation$ of $G_L$, associated to a Wach module, is a lattice inside a $\padic$ crystalline representation of $G_L$ (see Theorem \ref{thm:fh_crys_imperfect}) and prove the filtered isomorphism claimed in Theorem \ref{intro_thm:qdeformation_dcrys}.
In \S \ref{sec:crystalline_finite_height} we prove our main result, i.e.\ Theorem \ref{intro_thm:crystalline_wach_equivalence}.
In \S \ref{subsec:classical_wachmod} we collect important properties of classical Wach modules, i.e.\  the perfect residue field case.
In \S \ref{subsec:kisin_module} we use ideas from \cite{kisin-modules, kisin-ren} to construct a finite $\pqheight$ module on the open unit disk over $L$.
On the module thus obtained, we use results of \S \ref{subsec:crysrep_imperfect} to construct an action of $\Gamma_L$ and study its properties in \S \ref{subsec:nriglv_galois_action}.
Then in \S \ref{subsec:compatibility_phigammmod} we check that our construction is compatible with the theory of \'etale $(\varphi, \Gamma_L)\modules$.
Finally, in \S \ref{subsec:obtaining_wachmod} we construct the promised Wach module and prove Theorem \ref{intro_thm:crystalline_wach_equivalence}.

\vspace{2mm}

\noindent \textbf{Acknowledgements.} 
I would like to sincerely thank Takeshi Tsuji for discussing many ideas during the course of this project, reading a previous version of the article carefully and suggesting several improvements.
I would also like to thank Nicola Mazzari and Alex Youcis for helpful discussions.
Finally, I would like to thank the referee for thoroughly reading the article and offering valuable comments and suggestions for improvements.
This research is supported by JSPS KAKENHI grant numbers 22F22711 and 22KF0094.

\section{Period rings and \texorpdfstring{$\padic$}{-} representations}\label{sec:period_rings_padic_reps}

We will use the setup and notations from \S \ref{subsec:setup_nota}.
Recall that $O_L$ is a finite \'etale algebra over $\OLframe$.
Set $\Linfty := \cup_{i=1}^d L(\mu_{p^{\infty}}, X_i^{1/p^{\infty}})$ and for $1 \leqslant i \leqslant d$, we fix $X_i^{\flat} := (X_i, X_i^{1/p}, X_i^{1/p^2}, \ldots)$ in $\OLinfty^{\flat}$.
Then, we have the following Galois groups (see \cite[\S 1.1]{hyodo} for details):
\begin{align*}
	G_L &:= \Gal(\Lbar/L), \hspace{1mm} H_L := \Gal(\Lbar/\Linfty), \hspace{1mm} \Gamma_L := G_L/H_L = \Gal(\Linfty/L) \isomorphic \ZZ_p(1)^d \rtimes \ZZ_p^{\times}, \\
	\Gamma'_L &:= \Gal(L_{\infty}/L(\mu_{p^{\infty}})) \isomorphic \ZZ_p(1)^d, \hspace{1mm} \Gal(L(\mu_{p^{\infty}})/L) = \Gamma_L/\Gamma'_L \isomorphic \ZZ_p^{\times}.
\end{align*}
Let $\OLbreve := (\cup_{i=1}^d O_L[X_i^{1/p^{\infty}}])^{\wedge}$, where ${}^{\wedge}$ denotes the $\padic$ completion.
The $O_L\algebra$ $\OLbreve$ is a complete discrete valuation ring with perfect residue field, uniformiser $p$ and fraction field $\Lbreve := \OLbreve[1/p]$.
The Witt vector Frobenius on $\OLbreve$ is given by the Frobenius on $O_L$ described in \S \ref{subsec:setup_nota} and setting $\varphi(X_i^{1/p^n}) = X_i^{1/p^{n-1}}$, for all $1 \leqslant i \leqslant d$ and $n \in \NN$.
Let $\Lbreveinfty := \Lbreve(\mu_{p^{\infty}})$ and let $\Lbrevebar \supset \Lbar$ denote a fixed algebraic closure of $\Lbreve$.
Then, we have the following Galois groups:
\begin{align*}
	G_{\Lbreve} &:= \Gal(\Lbrevebar/\Lbreve) \isomorphic \Gal(\Lbar/\cup_{i=1}^d L(X_i^{1/p^{\infty}})), \hspace{1mm} H_{\Lbreve} := \Gal(\overline{\Lbreve}/\Lbreveinfty) = \Gal(\Lbar/\Linfty), \\
	\GammaLbreve &:= G_{\Lbreve}/H_{\Lbreve} = \Gal(\Lbreveinfty/\Lbreve) \isomorphic \Gal(\Linfty/\cup_{i=1}^d L(X_i^{1/p^{\infty}})) \isomorphic \Gal(L(\mu_{p^{\infty}})/L) \isomorphic \ZZ_p^{\times}.
\end{align*}
From the description above note that $G_{\Lbreve}$ can be identified with a subgroup of $G_L$, $H_{\Lbreve} \isomorphic H_L$ and $\GammaLbreve$ can be identified with a quotient of $\Gamma_L$.

\subsection{Period rings}\label{subsec:period_rings_imperfect}

In this subsection, we will quickly recall and fix notations for the period rings to be used in the rest of this article.
For details on constructions of period rings, please refer to \cite{andreatta-phigamma}, \cite{brinon-imparfait} and \cite{ohkubo-monodromy}.

As we will recall many period rings in this subsection, let us first briefly mention the usefulness of some of those rings in the constructions carried out for our main results (for precise definitions, please refer to \S \ref{subsubsec:crystalline_rings} -- \S \ref{subsubsec:periodrings_Lbreve}).

\begin{rem}\label{rem:period_ring_usage_summary}
	The period rings introduced in \S \ref{subsubsec:crystalline_rings}, for example, $\Ainf(\OLbar)$, $\Acrys(\OLbar)$, $\Bcrys(\OLinfty)$, etc.\ will be used to define and study properties of crystalline representations of $G_L$ (see \S \ref{subsec:crysrep_imperfect}), to show that Wach modules in the imperfect residue field case are crystalline (see \S \ref{subsec:wachmod_crystalline}), and to study the action of $\Gamma_L$ on various scalar extensions of a Wach module associated to a crystalline representation (see \S \ref{subsec:nriglv_galois_action}).
	Note that Wach modules are certain $(\varphi, \Gamma_L)\modules$ and the rings introduced in \S \ref{subsubsec:phigammamod_rings_imperfect} provide the basic setup for defining these objects and studying their properties.
	In particular, we remark that an (integral) Wach module $\NL(T)$, associated to a crystalline $\ZZ_p\textrm{-representation}$ $T$ of $G_L$, lives over the ring $\AL^+$, and the \'etale $(\varphi, \Gamma_L)\module$ associated to $T$ lives over $\AL$ (see \S \ref{subsec:phigamma_mod_imperfect} and \S \ref{subsec:wach_mod_props}).
	Next, the overconvergent period rings from \S \ref{subsubsec:overconvergent_rings} will be used to define overconvergent \'etale $(\varphi, \Gamma)\modules$ over $\ALdag$ (see \S \ref{subsec:phigamma_mod_imperfect}), which will be a crucial input for the construction of the Wach module associated to a crystalline representation of $G_L$ (see \S \ref{subsec:obtaining_wachmod}), and will be used to check that our constructions are compatible with the theory of \'etale $(\varphi, \Gamma_L)\modules$ (see \S \ref{subsec:compatibility_phigammmod}).
	Furthermore, the analytic rings of \S \ref{subsubsec:analytic_rings} will be the most important technical input for our constructions.
	For example, as a first step in our construction of $\NL(T)$, we construct an intermediate $\varphi\module$ $\NrigL(V)$ (where $V=T[1/p]$), over the ring $\BrigL^+$ using some ideas of Kisin (see \S \ref{subsec:kisin_module}).
	Additionally, to equip $\NrigL(V)$ with an action of $\Gamma_L$, we use the rings $\BrigL^+$, $\BrigLtilde^+$, $\Bcrys(\OLinfty)$, etc.; this is the main technical innovation of this article (see Proposition \ref{prop:nrigl_gammastab}).
	Moreover, the rings such as $\BrigLdag$ and $\BrigLtildedag$ are used to study the compatibility of $\NrigL(V)$ with the theory of $(\varphi, \Gamma_L)\modules$ of Andreatta (see \S \ref{subsec:compatibility_phigammmod}).
	Finally, the corresponding period rings over $\Lbreve$ in \S \ref{subsubsec:periodrings_Lbreve} are helpful in recollecting the results on classical Wach modules which are crucial inputs to our constructions (see \S \ref{subsec:classical_wachmod} and \S \ref{subsec:kisin_module}).
\end{rem}

\subsubsection{Crystalline period rings}\label{subsubsec:crystalline_rings}

Let $\Ainf(\OLinfty) := W(\OLinfty^{\flat})$ and $\Ainf(\OLbar) := W(\OLbar^{\flat})$ admitting the Frobenius on Witt vectors and continuous $G_L\action$ (for the weak topology).
We fix $\mubar := \varepsilon-1$, where $\varepsilon := (1, \zeta_p, \zeta_{p^2}, \ldots)$ is in $\OFinfty^{\flat}$ with $\zeta_{p^n}$ being a primitive $p^n\textrm{-th}$ root of unity, for each $n \geqslant 1$.
Set $\mu := [\varepsilon] - 1$ and $\xi := \mu/\varphi^{-1}(\mu)$ in $\Ainf(\OFinfty)$.
For any $g$ in $G_L$, we have that $g(1+\mu) = (1+\mu)^{\chi(g)}$, where $\chi$ is the $\padic$ cyclotomic character.
Moreover, we have a $G_L\equivariant$ surjection $\theta : \Ainf(\OLbar) \rightarrow O_{\CC_L}$, where $\CC_L := \widehat{\overline{L}}$ and $O_{\CC_L}$ is its ring of integers; note that $\kert \theta = \xi \Ainf(\OLbar)$.
The map $\theta$ further induces a $\Gamma_L\equivariant$ surjection $\theta : \Ainf(\OLinfty) \rightarrow O_{\widehat{L}_{\infty}}$.

Recall that, for $1 \leqslant i \leqslant d$, we fixed $X_i^{\flat} = (X_i, X_i^{1/p}, X_i^{1/p^2}, \ldots)$ in $\OLinfty^{\flat}$ and we take $\{\gamma_0, \gamma_1, \ldots, \gamma_d\}$ to be topological generators of $\Gamma_L$ such that $\{\gamma_1, \ldots, \gamma_d\}$ are topological generators of $\Gamma'_L$ and $\gamma_0$ is a topological generator of $\Gamma_L/\Gamma'_L$ and $\gamma_j(X_i^{\flat}) = \varepsilon X_i^{\flat}$, if $i=j$, and $X_i^{\flat}$, otherwise.
Let us also fix Teichm\"uller lifts $[X_i^{\flat}]$ in $\Ainf(\OLinfty)$.
We set $\Acrys(\OLinfty) := \Ainf(\OLinfty)\langle \xi^k/k!, k \in \NN \rangle$.
Let $t := \log(1+\mu)$ which converges in $\Acrys(\OFinfty)$ and set $\Bcrys^+(\OLinfty) := \Acrys(\OLinfty)[1/p]$ and $\Bcrys(\OLinfty) := \Bcrys^+(\OLinfty)[1/t]$.
For any $g$ in $G_L$, we have that $g(t) = \chi(g) t$.
Furthermore, one can define period rings $\OAcrys(\OLinfty)$, $\OBcrys^+(\OLinfty)$ and $\OBcrys(\OLinfty)$.
These rings are equipped with a Frobenius endomorphism $\varphi$ and a continuous $\Gamma_L\action$, and the former two rings $\OAcrys(\OLinfty)$ and $\OBcrys^+(\OLinfty)$ are further equipped with an appropriate extension of the map $\theta$.
Rings with a subscript ``cris'' are equipped with a decreasing filtration and rings with a prefix ``$\pazo$'' are further equipped with an integrable connection satisfying Griffiths transversality with respect to the filtration (see \cite[\S 2.2]{abhinandan-relative-wach-i} for definitions over $R$ with similar notations).
One can define variations of these rings over $\Lbar$ which are further equipped with $G_L\action$.
Moreover, from \cite[Lemma 4.32]{morrow-tsuji} note that $\Acrys(\OLinfty) = \Acrys(\OLbar)^{H_L}$ and $\Bcrys^+(\OLinfty) = \Bcrys^+(\OLbar)^{H_L}$.

We have two $O_L\algebra$ structures on $\OAcrys(\OLinfty)$: a canonical structure coming from the definition of $\OAcrys(\OLinfty)$; a non-canonical $(\varphi, \Gamma_{\Lbreve})\equivariant$ structure $O_L \rightarrow \OAcrys(\OLinfty)$ given by the map $x \mapsto \sum_{\smbfk \in \NN^d} \prod_{i=1}^d \partial_i^{k_i}(x) \prod_{i=1}^d ([X_i^{\flat}]-X_i)^{[k_i]}$, where $\partial_i := \tfrac{\partial}{\partial X_i}$ is a differential operator defined over $O_L$, for $1 \leqslant i \leqslant d$.
In particular, under the preceding map, we have that $X_i \mapsto [X_i^{\flat}]$.

\subsubsection{Rings of \texorpdfstring{$(\varphi, \Gamma)\modules$}{-}}\label{subsubsec:phigammamod_rings_imperfect}

For detailed explanations of objects defined in this subsubsection, see \cite{andreatta-phigamma}.
Recall that $\OLframe$ is a complete discrete valuation ring with a uniformiser $p$ and an imperfect residue field, and $O_L$ is a finite \'etale $\OLframe\algebra$.
Let us set $\ALframe^+$ to be the $(p, \mu)\textrm{-adic}$ completion of the localisation $O_F\llbracket \mu \rrbracket\big[[X_1^{\flat}]^{\pm 1}, \ldots, [X_d^{\flat}]^{\pm 1}\big]_{(p, \mu)}$.
We have a natural embedding $\ALframe^+ \hookrightarrow \Ainf(\OLinfty)$ and $\ALframe^+$ is stable under the Witt vector Frobenius and $\Gamma_L\action$ on $\Ainf(\OLinfty)$; we equip $\ALframe^+$ with induced structures.
Moreover, we have an injective homomorphism of rings $\iota : \OLframe \rightarrow \ALframe^+$, via the map $X_i \mapsto [X_i^{\flat}]$, and it extends to an isomorphism of rings $\OLframe\llbracket \mu \rrbracket \isomorphic \ALframe^+$.
Equip $\OLframe\llbracket \mu \rrbracket$ with a finite and faithfully flat of degree $p^{d+1}$ Frobenius endomorphism using the Frobenius on $\OLframe$ and setting $\varphi(\mu) = (1+\mu)^p-1$.
Then, the injective homomorphism $\iota$ and the isomorphism $\OLframe\llbracket \mu \rrbracket \isomorphic \ALframe^+$ are Frobenius-equivariant.

Let $\AL^+$ denote the $(p, \mu)\textrm{-adic}$ completion of the unique extension of the embedding $\ALframe^+ \rightarrow \Ainf(\OLinfty)$ along the finite \'etale map $\OLframe \rightarrow O_L$ (see \cite[Proposition 2.1]{colmez-niziol}).
We have a natural embedding of $O_F\textrm{-algebras}$ $\AL^+ \hookrightarrow \Ainf(\OLinfty)$ and $\AL^+$ is stable under the induced Frobenius and $\Gamma_L\action$.
Moreover, the injective homomorphism of rings $\iota : \OLframe \rightarrow \ALframe^+ \subset \AL^+$ and the isomorphism $\OLframe\llbracket \mu \rrbracket \isomorphic \ALframe^+ \subset \AL^+$, respectively, extend to a unique injective homomorphism of rings $\iota: O_L \rightarrow \AL^+$ and an isomorphism $O_L\llbracket \mu \rrbracket \isomorphic \AL^+$.
Equip $O_L\llbracket \mu \rrbracket$ with a finite and faithfully flat of degree $p^{d+1}$ Frobenius endomorphism using the Frobenius on $O_L$ and setting $\varphi(\mu) = (1+\mu)^p-1$.
Then, the injective homomorphism $\iota$ and the isomorphism $O_L\llbracket \mu \rrbracket \isomorphic \AL^+$ are Frobenius-equivariant.
In particular, the Frobenius $\varphi : \AL^+ \rightarrow \AL^+$ is finite and faithfully flat of degree $p^{d+1}$.
Let $u_{\alpha} := (1+\mu)^{\alpha_0} [X_1^{\flat}]^{\alpha_1} \cdots [X_d^{\flat}]^{\alpha_d}$, where $\alpha := (\alpha_0, \alpha_1, \ldots, \alpha_d)$ is a $d\textrm{-tuple}$ with $\alpha_i$ in $\{0, 1, \ldots, p-1\}$, for $0 \leqslant i \leqslant d$.
Then, we have that $\varphi^*(\AL^+) := \AL^+ \otimes_{\varphi, \AL^+} \AL^+ \isomorphic \oplus_{\alpha} \varphi(\AL^+) u_{\alpha}$.

Recall that $\CC_L = \widehat{\Lbar}$ and set $\Atilde := W\big(\CC_L^{\flat}\big)$ and $\Btilde := \Atilde[1/p]$ admitting the Frobenius on Witt vectors and continuous $G_L\action$ (for the weak topology).
Set $\AL := \AL^+[1/\mu]^{\wedge}$, where ${}^{\wedge}$ denotes the $\padic$ completion; equip $\AL^+$ with the induced Frobenius endomorphism and continuous $\Gamma_L\action$.
Note that $\AL$ is a complete discrete valuation ring with maximal ideal $p\AL$, residue field $(O_L/p)((\mu))$ and fraction field $\BL := \AL[1/p]$.
Similar to above, $\varphi : \AL \rightarrow \AL$ is finite and faithfully flat of degree $p^{d+1}$ and we have that $\varphi^*(\AL) := \AL \otimes_{\varphi, \AL} \AL \isomorphic \oplus_{\alpha} \varphi(\AL) u_{\alpha} = (\oplus_{\alpha} \varphi(\AL^+) u_{\alpha}) \otimes_{\varphi(\AL^+)} \varphi(\AL) \lisomorphic \AL^+ \otimes_{\varphi, \AL^+} \AL$.
Furthermore, we have a natural Frobenius and $\Gamma_L\equivariant$ embedding $\AL \subset \Atilde^{H_L}$.
Let $\mbfa$ denote the $\padic$ completion of the maximal unramified extension of $\AL$ inside $\Atilde$ and set $\mbfb := \mbfa[1/p] \subset \Btilde$, i.e.\ $\mbfa$ is the ring of integers of $\mbfb$.
The rings $\mbfa$ and $\mbfb$ are stable under induced Frobenius and $G_L\action$, and we have $\AL = \mbfa^{H_L}$ and $\BL = \mbfb^{H_L}$ stable under the induced Frobenius and residual $\Gamma_L\action$.

\subsubsection{Overconvergent rings}\label{subsubsec:overconvergent_rings}

We begin by definining the ring of overconvergent coefficients stable under Frobenius and $G_L\action$ (see \cite{cherbonnier-colmez} and \cite{andreatta-brinon}).
Denote the natural valuation on $\OLbar^{\flat}$ by $\upsilon^{\flat}$ extending the valuation on $O_{\Fbar}^{\flat}$.
Let $r \in \QQ_{> 0}$ and set,
\begin{equation*}
	\Atildedagr := \big\{\textstyle \sum_{k \in \NN} p^k [x_k] \in \Atilde \textrm{ such that } \upsilon^{\flat}(x_k) + \tfrac{pr}{p-1}k \rightarrow +\infty \hspace{1mm} \textrm{as} \hspace{1mm} k \rightarrow +\infty \big\}.
\end{equation*}
The $G_L\action$ and Frobenius $\varphi$ on $\Atilde$ induce commuting actions of $G_L$ and $\varphi$ on $\Atildedagr$ such that $\varphi(\Atildedagr) = \tilde{\mbfa}^{\dagger, pr}$.
Define the ring of \textit{overconvergent coefficients} as $\Atildedag := \cup_{r \in \QQ_{>0}} \Atildedagr \subset \Atilde$ equipped with the induced Frobenius and $G_L\action$.
Moreover, inside $\Atilde$ we take $\ALdagr := \AL \cap \Atildedagr$ and $\Adagr := \mbfa \cap \Atildedagr$.
Define $\ALdag := \AL \cap \Atildedag = \cup_{r \in \QQ_{>0}} \ALdagr$ and $\Adag := \mbfa \cap \Atildedag = \cup_{r \in \QQ_{>0}} \Adagr$, equipped with the induced Frobenius endomorphism and $G_L\action$ from respective actions on $\Atilde$; we have $\ALdag = (\Adag)^{H_L}$.
Upon inverting $p$ in the definitions above one obtains $\QQ_p\textrm{-algebras}$ inside $\Btilde$, i.e.\ set $\Btildedagr := \Atildedagr[1/p]$, $\Btildedag := \Atildedag[1/p]$, $\Bdagr := \Adagr[1/p]$, $\Bdag := \Adag[1/p]$, equipped with the induced Frobenius and $G_L\action$.
Moreover, set $\BLtildedagr := (\Btildedagr)^{H_L}$, $\BLtildedag := (\Btildedag)^{H_L}$, $\BLdagr := (\Bdagr)^{H_L} = \ALdagr[1/p]$ and $\BLdag := (\Bdag)^{H_L} = \ALdag[1/p]$, equipped with the induced Frobenius and residual $\Gamma_L\action$.

\subsubsection{Analytic rings}\label{subsubsec:analytic_rings}

In this subsection, we will define the Robba ring over $L$ following \cite[\S 2]{kedlaya-slope} and \cite[\S 1]{ohkubo-differential}.
However, we will use the notations of \cite[\S 2]{berger-differentielles} in the perfect residue field case (see \cite[\S 1.10]{ohkubo-differential} for compatibility between different notations).
Define
\begin{equation*}
	\Brigtildedag := \cup_{r \geqslant 0} \cap_{s \geqslant r} \big(\Ainf(\OLbar)\langle \tfrac{p}{[\mubar]^r}, \tfrac{[\mubar]^s}{p} \rangle \pinverse\big).
\end{equation*}
The ring $\Brigtildedag$ can also be defined as $\cup_{r \in \QQ_{>0}} \Brigtildedagr$, where $\Brigtildedagr$ denotes the Fr\'echet completion of $\Btildedagr = \Atildedagr[1/p]$ for a certain family of valuations (see \cite[\S 2]{kedlaya-slope} and \cite[\S 1.6]{ohkubo-differential}).
The Frobenius and $G_L\action$ on $\Btildedagr$, respectively, induce Frobenius and $G_L\action$ on $\Brigtildedagr$, which extend to respective actions on $\Brigtildedag$.
In particular, we have a Frobenius and $G_L\equivariant$ inclusion $\Btildedag \subset \Brigtildedag$ (see \cite[\S 1.6 \& \S 1.10]{ohkubo-differential}).
Set 
\begin{equation*}
	\Brigtilde^+ := \cap_{n \in \NN} \varphi^n(\Bcrys^+(\OLbar))
\end{equation*}
equipped with an induced Frobenius endomorphism and $G_L\action$ from the respective actions on $\Bcrys^+(\OLbar)$.
The descriptions of rings in \cite[Lemme 2.5, Exemple 2.8 \& \S 2.3]{berger-differentielles} directly extend to our situation as the aforementioned results do not depend on structure of the residue field of base ring $O_L$.
Therefore, from loc.\ cit.\ it follows that $\Brigtilde^+ \subset \Brigtildedag$ compatible with Frobenius and $G_L\action$.
Moreover, we set $\BrigLtildedagr :=  (\Brigtildedagr)^{H_L}$, $\BrigLtildedag :=  (\Brigtildedag)^{H_L}$ and $\BrigLtilde^+ := (\Brigtilde^+)^{H_L} \subset \BrigLtildedag$, equipped with the induced Frobenius endomorphism and residual $\Gamma_L\action$.

\begin{rem}\label{rem:brigtilde_lbreve}
	Note that the definition $\Brigtildedag$ and $\Brigtilde^+$ as rings does not depend on $L$, in particular, one may define these rings using $\Ainf(O_{\overline{\Lbreve}})$ and equip them with a Frobenius endomorphism compatible with the Frobenius endomorphism defined above.
\end{rem}

\begin{lem}\label{lem:brigtilde_frobfixed}
	We have $(\Brigtildedag)^{\varphi=1} = (\Brigtilde^+)^{\varphi=1} = \QQ_p$.
\end{lem}
\begin{proof}
	Using Remark \ref{rem:brigtilde_lbreve}, note that the Frobenius invariant elements can be computed using the corresponding results in the perfect residue field case.
	In particular, we have that $(\Brigtildedag)^{\varphi=1} = (\Brigtilde^+)^{\varphi=1} = \QQ_p$, where the first equality follows from \cite[Proposition I.4.1]{berger-limites} and the second equality follows from \cite[Proposition 9.15]{colmez-banach}.
\end{proof}

Recall that from \S \ref{subsubsec:phigammamod_rings_imperfect} we have a Frobenius-equivariant injective homomorphism of rings $\iota : O_L \rightarrow \AL^+$.
Then, from \cite[\S 1.6]{ohkubo-differential} the ring $\ALdagr$ admits the following description:
\begin{equation*}
	\ALdagr \isomorphic \big\{\textstyle \sum_{k \in \ZZ} \iota(a_k) \mu^k \textrm{ such that } a_k \in O_L \textrm{ and for any } p^{-1/r} \leqslant \rho < 1, \lim_{k \rightarrow -\infty} |a_k|\rho^k = 0\big\}.
\end{equation*}
Moreover, we have that $\BLdagr = \ALdagr[1/p]$ and we set
\begin{equation*}
	\BrigLdagr := \big\{\textstyle \sum_{k \in \ZZ} \iota(a_k) \mu^k \textrm{ such that } a_k \in L \textrm{ and for any } p^{-1/r} \leqslant \rho < 1, \lim_{k \rightarrow \pm\infty} |a_k|\rho^k = 0\big\}.
\end{equation*}
The ring $\BrigLdagr$ can also be defined as Fr\'echet completion of $\BLdagr$ for a family of valuations induced by the inclusion $\BLdagr \subset \Btildedagr$ (see \cite[\S 2]{kedlaya-slope} and \cite[\S 1.6]{ohkubo-differential}).
Define the \textit{Robba ring} over $L$ as $\BrigLdag := \cup_{r \geqslant 0} \BrigLdagr$.
The Frobenius and $G_L\action$ on $\BLdagr$ induce respective Frobenius and $G_L\action$ on $\BrigLdagr$, which extend to respective actions on $\BrigLdag$ (also see \cite[\S 4.3]{ohkubo-differential} where Ohkubo constructs the differential action of $\Lie \Gamma_L$; one may also obtain the action of $\Gamma_L$ by exponentiating the action of $\Lie \Gamma_L$).
From the preceding discussion, we have a Frobenius and $\Gamma_L\equivariant$ injection $\BLdag \subset \BrigLdag$ and the former ring $\BLdag$ is also known as the \textit{bounded Robba ring}.
Furthermore, note that $\BLdagr \subset \tilde{\mbfb}_L^{\dagger, r} = (\Btildedagr)^{H_L} \subset \BrigLtildedagr$, where the last term can also be described as the Fr\'echet completion of the middle term for a family of valuations induced by the inclusion $\tilde{\mbfb}_L^{\dagger, r} \subset \Btildedagr$ (see \cite[\S 2]{kedlaya-slope} and \cite[\S 1.6]{ohkubo-differential}).

To summarise, for $r \in \QQ_{>0}$, we have the following commutative diagram with injective arrows:
\begin{center}
	\begin{tikzcd}
		\Bdagr \arrow[r] & \Btildedagr \arrow[rd]\\
		\BLdagr \arrow[u] \arrow[r] \arrow[d] & \BLtildedagr \arrow[u] \arrow[r] \arrow[d] & \Brigtildedagr\\
		\BrigLdagr \arrow[r] & \BrigLtildedagr \arrow[ru],
	\end{tikzcd}
\end{center}
where in the second row, the two rings on the left are obtained from the rings in first row by taking $H_L\textrm{-invariants}$ and the rightmost ring in the second row is obtained as the Fr\'echet completion of the rightmost ring in the first row.
The bottom row is obtained as the Fr\'echet completion of the two rings on the left in the second row.
These inclusions are compatible with the respective Frobenii and $\Gamma_L\textrm{-actions}$ and these compatibilities are preserved after passing to the respective Fr\'echet completions.
In particular, we have a Frobenius and $\Gamma_L\equivariant$ embedding $\BrigLdag \subset \BrigLtildedag$.

\begin{defi}
	Define $\BrigL^+ := \BrigLdag \cap \BrigLtilde^+ \subset \BrigLtildedag$, equipped with the induced Frobenius endomorphism and $\Gamma_L\action$.
\end{defi}

\begin{lem}\label{lem:brigl+_explicit}
	The ring $\BrigL^+$ can be identified with the ring of convergent power series over the open unit disk in one variable over $L$, i.e.
	\begin{equation*}
		\BrigL^+ \isomorphic \big\{\textstyle \sum_{k \in \NN} \iota(a_k) \mu^k \textrm{ such that } a_k \in L \textrm{ and for any } 0 \leqslant \rho < 1, \lim_{k \rightarrow +\infty} |a_k|\rho^k = 0\big\},
	\end{equation*}
\end{lem}
\begin{proof}
	Let $x$ be any element of $\BrigL^+ \subset \BrigLdag$.
	Using the explicit description of $\BrigLdagr$ and $\BLdagr$ for $r \in \QQ_{>0}$, we can write $x = x^+ + x^-$, with $x^+$ convergent on the open unit disk over $L$ and $x^-$ in $\BLdagr$, for some $r \in \QQ_{>0}$, in particular, we have that $x^+$ is in $\Brigtilde^+$.
	Moreover, using Remark \ref{rem:brigtilde_lbreve} and \cite[Lemma 2.18, Corollaire 2.28]{berger-differentielles}, we have an exact sequence $0 \rightarrow \Binf(\OLbar) \rightarrow \Btildedagr \oplus \Brigtilde^+ \rightarrow \Brigtildedagr \rightarrow 0$, where $\Binf(\OLbar) = \Ainf(\OLbar)[1/p]$.
	So, $x$ is in $\BrigL^+ \subset \Brigtilde^+$ if and only if $x^- = x-x^{+}$ is in $\Binf(\OLbar) \cap \BLdagr = \Binf(\OLinfty) \cap \BLdagr = \BL^+$, where we have used that $\Ainf(\OLbar)^{H_L} = \Ainf(\OLinfty)$ (see \cite[Proposition 7.2]{andreatta-phigamma}).
	Hence, $x$ converges on the open unit disk over $L$.
	The other inclusion is obvious, allowing us to conclude.
\end{proof}

\begin{rem}\label{rem:brigl_topology}
	The topology on $\BrigL^+$ can be described as follows:
	let $D(L, \rho)$ denote the closed disk of radius $0 < \rho < 1$ over $L$ and let $\pazo(D(L, \rho))$ denote the ring of analytic functions, i.e.\ power series converging on the closed disk $D(L, \rho)$.
	Then, $\pazo(D(L, \rho))$ is equipped with a topology induced by the supremum norm $\|f\|_{\rho} := \sup_{x \in D(L, \rho)}|f(x)| < +\infty$.
	We have that $\BrigL^+ = \lim_{\rho} \pazo(D(L, \rho)) \subset L\llbracket \mu \rrbracket$ and we equip it with the topology induced by the Fr\'echet limit of the topology on $\pazo(D(L, \rho))$ induced by the supremum norm, i.e.\ the topology on $\BrigL^+$ can be described by uniform convergence on $D(L, \rho)$ for $\rho \rightarrow 1^{-}$.
\end{rem}

\begin{lem}\label{lem:bl_brigl_ff}
	The natural map $\BL^+ \rightarrow \BrigL^+$ is faithfully flat.
\end{lem}
\begin{proof}
	Note that $\BL^+$ is a principal ideal domain and $\BrigL^+$ is a domain, so the map in the claim is flat.
	To show that it is faithfully flat, it is enough to show that for any maximal ideal $\frakm \subset \BL^+$, we have that $\frakm \BrigL^+ \neq \BrigL^+$.
	Note that if $\frakm \subset \BL^+$ is a maximal ideal, then $\frakm = (f)$, where $f$ is an irreducible distinguished polynomial in the sense of \cite[Chapter 5, \S 2]{lang}.
	Since any $f$ as above admits a zero over the open unit disk, therefore, it follows that $f$ is not a unit in $\BrigL^+$.
	Hence, $\frakm \BrigL^+ \neq \BrigL^+$.
\end{proof}

\begin{rem}\label{lem:brigl+_frob_ff}
	From \S \ref{subsec:setup_nota} recall that $\varphi : L \rightarrow L$ is finite of degree $p^d$ and we also have that $\varphi(\mu) = (1+\mu)^p-1$.
	Therefore, from the explicit description of $\BrigL^+$ in Lemma \ref{lem:brigl+_explicit}, it follows that the Frobenius endomorphism $\varphi : \BrigL^+ \rightarrow \BrigL^+$ is finite and faithfully flat of degree $p^{d+1}$.
\end{rem}

\subsubsection{Period rings for \texorpdfstring{$\Lbreve$}{-}}\label{subsubsec:periodrings_Lbreve}

Definitions above may be adopted almost verbatim to define corresponding period rings for $\Lbreve$, in particular, one recovers definitions of period rings in \cite{fontaine-phigamma}, \cite{cherbonnier-colmez} and \cite{berger-differentielles}, i.e.\ we have period rings  $\ALbreve^+$, $\ALbreve$, $\ALbrevedag$, $\BrigLbreve^+$ and $\BrigLbrevedag$ equipped with a Frobenius endomorphism $\varphi$ and $\GammaLbreve\action$.
Note that we have a natural identification $\OLbreve\llbracket \mu \rrbracket \isomorphic \ALbreve^+$, where the right hand side is equipped with a finite and faithfully flat of degree $p$ Frobenius endomorphism using the natural Frobenius on $\OLbreve$ and setting $\varphi(\mu) = (1+\mu)^p-1$ and a $\GammaLbreve\action$ given as $g(\mu) = (1+\mu)^{\chi(g)}-1$, for any $g$ in $\GammaLbreve$.
Moreover, the preceding isomorphism naturally extends to a Frobenius and $\GammaLbreve\equivariant$ isomorphism $\ALbreve \isomorphic \OLbreve\llbracket \mu \rrbracket[1/\mu]^{\wedge}$, where ${}^{\wedge}$ denotes the $\padic$ completion.

Similar to above, we further equip $O_L\llbracket \mu \rrbracket$ with an $O_L\linear$ action of $\GammaLbreve$, by setting $g(\mu) = (1+\mu)^{\chi(g)}-1$, for any $g$ in $\GammaLbreve$.
Then the isomorphism $O_L\llbracket \mu \rrbracket \isomorphic \AL^+$ from \ref{subsubsec:phigammamod_rings_imperfect}, is Frobenius and $\GammaLbreve\equivariant$.
Now, recall that the Frobenius-equivariant embedding $O_L \rightarrow \OLbreve$ is faithfully flat and it naturally extends to a Frobenius and $\GammaLbreve\equivariant$ faithfully flat embedding $O_L\llbracket \mu \rrbracket \rightarrow \OLbreve\llbracket \mu \rrbracket$.
So, using the preceding embedding and the Frobenius and $\GammaLbreve\equivariant$ isomorphisms: the inverse of $O_L\llbracket \mu \rrbracket \isomorphic \AL^+$ and the isomorphism $\OLbreve\llbracket \mu \rrbracket \isomorphic \ALbreve^+$, we obtain a Frobenius and $\GammaLbreve\equivariant$ faithfully flat embedding $\AL^+ \rightarrow \ALbreve^+$, sending $[X_i^{\flat}] \mapsto X_i$.
This further extends to a Frobenius and $\GammaLbreve\equivariant$ faithfully flat embedding $\AL \rightarrow \ALbreve$.

We can equip $\Ainf(\OLinfty)$ with a non-canonical $O_L\algebra$ structure by first defining an injection $\OLframe \rightarrow \Ainf(\OLinfty)$, via the map $X_i \mapsto [X_i^{\flat}]$, and then extending it uniquely along the finite \'etale map $\OLframe \rightarrow O_L$, to an injection $O_L \rightarrow \Ainf(\OLinfty)$ (see \cite[Proposition 2.1]{colmez-niziol}).
Note that the preceding maps are Frobenius-equivariant but not $\Gamma_L\equivariant$.
The $O_L\algebra$ structure on $\Ainf(\OLinfty)$ naturally extends to a Frobenius-equivariant $\OLbreve\algebra$ structure by sending $X_i^{1/p^n} \mapsto [(X_i^{1/p^n})^{\flat}]$, for all $1 \leqslant i \leqslant d$ and $n \in \NN$.
We can further extend this to a Frobenius and $\GammaLbreve\equivariant$ embedding $\ALbreve^+ = \OLbreve\llbracket \mu \rrbracket \rightarrow \Ainf(\OLinfty)$.

Using the embeddings described above and following the definitions of various period rings discussed so far, we obtain a commutative diagram with injective arrows where the top horizontal arrows are Frobenius and $\Gamma_L\equivariant$ and the rest are Frobenius and $\GammaLbreve\equivariant$:
\begin{center}
	\begin{tikzcd}
		\BrigLtilde^+ & \BrigL^+ \arrow[l] \arrow[d] \arrow[r] & \BrigLdag \arrow[d] \arrow[r] & \BrigLtildedag\\
		& \BrigLbreve^+ \arrow[lu] \arrow[r] & \BrigLbrevedag \arrow[ru].
	\end{tikzcd}
\end{center}

\begin{rem}\label{rem:briglbreve+}
	Similar to Lemma \ref{lem:brigl+_explicit} we have that,
	\begin{equation*}
		\BrigLbreve^+ \isomorphic \big\{\textstyle \sum_{k \in \NN} a_k \mu^k \textrm{ such that } a_k \in \Lbreve \textrm{ and for any } 0 \leqslant \rho < 1, \lim_{k \rightarrow +\infty} |a_k|\rho^k = 0\big\}.
	\end{equation*}
	The ring $\BrigLbreve^+$ is equipped with a Fr\'echet topology similar to Remark \ref{rem:brigl_topology}.
	Moreover, since $\varphi : \Lbreve \isomorphic \Lbreve$ and $\varphi(\mu) = (1+\mu)^p-1$, the Frobenius endomorphism on $\BrigLbreve^+$ is finite and faithfully flat of degree $p$.
\end{rem}

\begin{lem}\label{lem:brigl_lbreve_flat}
	The rings $\BrigL^+$ and $\BrigLbreve^+$ are B\'ezout domains and $\BrigL^+ \rightarrow \BrigLbreve^+$ is flat.
\end{lem}
\begin{proof}
	The first claim follows from \cite[Proposition 4.12]{berger-differentielles}.
	Note that loc.\ cit.\  assumes the residue field of the discrete valuation base field ($L$ and $\Lbreve$ in our case) to be perfect, however, the proof of loc.\ cit.\ only depends on \cite{lazard} and \cite{helmer} which are independent of this assumption.
	For the second claim, note that we can write $\BrigLbreve^+ = \colim_{i \in I} M_i$, where $I$ is the directed index set of finitely generated $\BrigL^+\textrm{-submodules}$ of $\BrigLbreve^+$.
	Since $\BrigLbreve^+$ is a domain, $M_i$ is torsion-free for each $i \in I$.
	Now, recall that finitely generated torsion-free modules over a B\'ezout domain are finite projective (see \cite[Chapter VII, Proposition 4.1]{cartan-eilenberg} noting that B\'ezout domains are a special case of Pr\"ufer domains), and therefore finite free by \cite[Proposition 2.5]{kedlaya-monodromy}.
	Moreover, directed colimit of finite free modules over a ring is flat (see \cite[\href{https://stacks.math.columbia.edu/tag/058G}{Tag 058G}]{stacks-project}).
	Hence, it follows that $\BrigL^+ \rightarrow \BrigLbreve^+$ is flat.
\end{proof}

\begin{lem}\label{lem:brigl_lbreve_intersect}
	The element $t/\mu = (\log(1+\mu))/\mu = \prod_{n \in \NN} (\varphi^n([p]_q)/p)$ converges in $\BrigL^+ \subset \BrigLbreve^+$.
	Moreover, $(t/\mu) \BrigLbreve^+ \cap \BrigL^+ = (t/\mu) \BrigL^+$.
\end{lem}
\begin{proof}
	The first claim follows from \cite[Exemple I.3.3]{berger-limites} and \cite[Remarque 4.12]{lazard}.
	For the second claim let $x = \sum_{k \in \NN} x_k \mu^k$ in $\BrigL^+$, with $x_k \in L$, and let $y = \sum_{k \in \NN} y_k \mu^k$ in $\BrigLbreve^+$, with $y_k \in \Lbreve$, such that $ty/\mu = x$.
	Write $t/\mu = \sum_{k \in \NN} a_k \mu^k$, with $a_k \in \QQ_p$.
	Then, we have that $(\sum_{k \in \NN} a_k \mu^k) (\sum_{k \in \NN} y_k \mu^k) = \sum_{k \in \NN} x_k \mu^k$.
	We will show that $y_k$ is in $L$, for all $k \in \NN$, using induction.
	Indeed, note that $a_0 y_0 = x_0$ in $L$, so $y_0 = x_0/a_0$ is in $L$.
	Let $n \in \NN$ and assume that $y_k$ is in $L$, for every $k \leqslant n$.
	Then, we have that $\sum_{k=0}^{n+1} a_k y_{n+1-k} = x_{n+1}$ in $L$ and by the induction assumption we get that $y_{n+1} = (x_{n+1} - \sum_{k=0}^{n} a_k y_{n+1-k})/a_0$ is in $L$.
	Hence, we conclude that $y$ is in $\BrigL^+$, implying that $(t/\mu) \BrigLbreve^+ \cap \BrigL^+ = (t/\mu) \BrigL^+$.
\end{proof}

\begin{lem}\label{lem:brigl_tilde_t}
	We have $(t/\mu) \BrigLtilde^+ \cap \BrigLbreve^+ = (t/\mu) \BrigLbreve^+$, therefore $(t/\mu) \BrigLtilde^+ \cap \BrigL^+ = (t/\mu) \BrigL^+$ from Lemma \ref{lem:brigl_lbreve_intersect}.
\end{lem}
\begin{proof}
	Let us first note that for each $n \in \NN_{\geqslant 1}$ we have the following diagram:
	\begin{center}
		\begin{tikzcd}[column sep=large]
			0 \arrow[r] & \BrigLbreve^+ \arrow[r,"{\varphi^{n-1}([p]_q)}"] \arrow[d] & \BrigLbreve^+ \arrow[r,"\mu \mapsto \zeta_{p^n}-1"] \arrow[d] & \Lbreve(\zeta_{p^n}) \arrow[d, "\phi_{\Lbreve}"] \arrow[r] & 0\\
			0 \arrow[r] & \Brigtilde^+ \arrow[r,"{\varphi^{n-1}([p]_q)}"] & \Brigtilde^+ \arrow[r,"\theta \circ \varphi^{-n}"] & \CC_L \arrow[r] & 0,
		\end{tikzcd}
	\end{center}
	where left and middle vertical arrows are natural inclusions, the right vertical arrow is $\phi_{\Lbreve} : \Lbreve(\zeta_{p^n}) \isomorphic \Lbreve(\zeta_{p^n}) \subset \CC_L$, given as $\sum_{k=0}^{e-1} a_k \zeta_{p^n}^k \mapsto \sum_{k=0}^{e-1} \varphi_{\Lbreve}^{-n}(a_k) \zeta_{p^n}^k$, with $e = [\Lbreve(\zeta_{p^n}) : \Lbreve]$ and $\varphi_{\Lbreve} : \Lbreve \isomorphic \Lbreve$ and $\theta : \Brigtilde^+ \subset \Bcrys^+(\OLbar) \rightarrow \CC_L$ from \S \ref{subsubsec:crystalline_rings}.
	The top row is obviously exact and the bottom row is exact by \cite[Proposition 2.11, Proposition 2.12 \& Remarque 2.14]{berger-differentielles}.
	All vertical maps are injective and hence we obtain that $\varphi^n([p]_q) \Brigtilde^+ \cap \BrigLbreve^+ = \varphi^n([p]_q) \BrigLbreve^+$, for all $n \in \NN$, in particular, $\varphi^n([p]_q) \BrigLtilde^+ \cap \BrigLbreve^+ = \varphi^n([p]_q) \BrigLbreve^+$.
	Now, let $x$ be in $(t/\mu) \BrigLtilde^+ \cap \BrigLbreve^+$ and write $x = ty/\mu$, for some $y$ in $\BrigLtilde^+$.
	We will show that $y$ is in $\BrigLbreve^+$ by showing that it converges over each closed disk $D(\Lbreve, \rho)$, for $0 < \rho < 1$.
	Fix some $0 < \rho < 1$ and from Lemma \ref{lem:brigl_lbreve_intersect}, we write $t/\mu = \prod_{n \in \NN} (\varphi^n([p]_q)/p) = \upsilon \prod_{n=0}^{m} (\varphi^n([p]_q)/p)$, for a unit $\upsilon \in \pazo(D(\Lbreve, \rho))^{\times}$ and $m \in \NN$ depending on $\rho$.
	Then, we have that $x = ([p]_q/p) y_1$, where $y_1 := \upsilon \prod_{n=1}^{m} (\varphi^n([p]_q)/p) y$ is in $\BrigLtilde^+ \cap (p/[p]_q) \BrigLbreve^+ = \BrigLbreve^+$.
	Repeating the preceding argument for $1 \leqslant k \leqslant m$, we obtain elements $y_{k+1} := \upsilon \prod_{n=k+1}^{m} (\varphi^n([p]_q)/p) y$ in $\BrigLtilde^+ \cap \varphi^n(p/[p]_q) \BrigLbreve^+ = \BrigLbreve^+$.
	In particular, we see that $y = \upsilon^{-1} y_{m+1}$ is in $\pazo(D(\Lbreve, \rho))$.
	Since, $\BrigLbreve^+ = \lim_{\rho} \pazo(D(\Lbreve, \rho))$, therefore, we conclude that $y$ must be in $\BrigLbreve^+$.
\end{proof}

\subsubsection{\texorpdfstring{$\varphi$}{-}-modules over certain period rings}

Let $\varphi\Mod_{\BrigLdag}$ denote the category of finite free modules over $\BrigLdag$ equipped with an isomorphism $1 \otimes \varphi : \varphi^*M \isomorphic M$ and morphisms between objects are $\BrigLdag\linear$ maps compatible with $1 \otimes \varphi$ on both sides; denote by $\varphi\Mod_{\BrigLdag}^0$ the full subcategory of objects that are pure of slope 0 in the sense of \cite[\S 6.3]{kedlaya-monodromy}.
Similarly, one can define the category $\varphi\Mod_{\BLdag}$ and denote by $\varphi\Mod_{\BLdag}^0$ the full subcategory of objects that are pure of slope 0 (as $\varphi\modules$ over a discretely valued field).

Let $\Eff\varphi\Mod_{\AL^+}^{[p]_q}$ denote the category of effective and finite $\pqheight$ $\AL^+\modules$, i.e.\ an object is this category is a finite free $\AL^+\module$ $N$ equipped with a Frobenius-semilinear endomorphism $\varphi : N \rightarrow N$ such that the map $1 \otimes \varphi : \varphi^*(N) \rightarrow N$ is injective and its cokernel is killed by a finite power of $[p]_q$; denote by $\Eff\varphi\Mod_{\AL^+}^{[p]_q} \otimes \QQ_p$ the associated isogeny category.
Similarly, define $\Eff\varphi\Mod_{\BrigL^+}^{[p]_q}$ to be the category of effective and finite $\pqheight$ $\BrigL^+\modules$ and $\Eff\varphi\Mod_{\BrigL^+}^{[p]_q, 0}$ to be the full subcategory of objects that are pure of slope 0, i.e.\ $M$ such $\BrigLdag \otimes_{\BrigL^+} M$ is pure of slope 0.

\begin{lem}\phantomsection\label{lem:finiteheightslope0_equiv}
	\begin{enumarabicup}
	\item There is a natural equivalence of categories $\varphi\Mod_{\BLdag}^0 \isomorphic \varphi\Mod_{\BrigLdag}^0$, induced by the functor $M \mapsto M \otimes_{\BLdag} \BrigLdag$.

	\item There is an exact equivalence of $\otimes\textrm{-categories}$ $\Eff\varphi\Mod_{\AL^+}^{[p]_q} \otimes \QQ_p \isomorphic \Eff\varphi\Mod_{\BrigL^+}^{[p]_q, 0}$, induced by the functor $N \mapsto N \otimes_{\AL^+} \BrigL^+$.
	\end{enumarabicup}
\end{lem}
\begin{proof}
	The claim in (1) follows from \cite[Theorem 6.3.3]{kedlaya-slope}.
	The equivalence of $\otimes\textrm{-categories}$ in (2) follows from (1), \cite[Lemma 1.3.13]{kisin-modules} and \cite[Proposition 6.5]{kedlaya-monodromy}, and the exactness follows since $\BL^+ \rightarrow \BrigL^+$ is faithfully flat by Lemma \ref{lem:bl_brigl_ff}.
	Note that in \cite{kisin-modules}, Kisin assumes the residue field of the discrete valuation base field ($L$ in our case) to be perfect.
	However, the proof of \cite[Lemma 1.3.13]{kisin-modules} depends only on \cite[Proposition 6.5]{kedlaya-monodromy} and \cite[Theorem 6.3.3]{kedlaya-slope} which are independent of the structure of the residue field.
	In particular, the proof of \cite[Lemma 1.3.13]{kisin-modules} applies almost verbatim to our case.
	We recall the quasi-inverse functor from loc.\ cit.\ that will be useful later (see \S \ref{subsec:obtaining_wachmod}).

	Let $\Mrig^+$ be a finite height effective $\BrigL^+\module$ pure of slope 0, then $\Mrigdag := \BrigLdag \otimes_{\BrigL^+} \Mrig^+$ is pure of slope 0 and (1) implies that there exists a finite free $\BLdag\module$ $\Mdag$ pure of slope 0 such that $\BrigLdag \otimes_{\BLdag} \Mdag \isomorphic \Mrigdag \lisomorphic \BrigLdag \otimes_{\BrigL^+} \Mrig^+$.
	Choose a $\BLdag\textrm{-basis}$ of $\Mdag$ and a $\BrigL^+\textrm{-basis}$ of $\Mrig^+$.
	The composite of the isomorphisms above is given by a matrix with values in $\BrigLdag$.
	By \cite[Proposition 6.5]{kedlaya-monodromy}, after modifying the chosen bases, we may assume the matrix to be identity, in particular, $\Mdag$ and $\Mrig^+$ are spanned by a common basis.
	Let $M$ denote the $\BL^+\textrm{-span}$ of this basis.
	Since $\BL^+ = \BrigL^+ \cap \BLdag \subset \BrigLdag$, we obtain that $M = \Mrig^+ \cap \Mdag \subset \Mrigdag$, and $\BrigL^+ \otimes_{\BL^+} M \isomorphic \Mrig^+$ and $\BLdag \otimes_{\BL^+} M \isomorphic \Mdag$.
	Moreover, $\Mdag$ is pure of slope 0, so there exists an $\ALdag\textrm{-lattice}$ $M_0^{\dagger} \subset \Mdag$.
	Let $M_0' := M \cap M_0^{\dagger} \subset \Mdag$ and set $M_0 := (\ALdag \otimes_{\AL^+} M_0') \cap M_0'[1/p] \subset \Mdag$.
	Using \cite[Lemma 1.3.13]{kisin-modules} and the discussion above, $M_0 \subset M$ is a finite free $\varphi\textrm{-stable}$ $\AL^+\textrm{-submodule}$ such that cokernel of the injective map $1 \otimes \varphi : \varphi^*(M_0) \rightarrow M_0$ is killed by some finite power of $[p]_q$.
\end{proof}

\begin{rem}\label{rem:closedsubmod_free}
	Let $M$ be a finite free $\BrigL^+\module$ and $N \subset M$ a $\BrigL^+\textrm{-submodule}$.
	Then, $N$ is finite free if and only if it is finitely generated if and only if it is a closed submodule of $M$.
	Equivalences in the preceding statement essentially follow from \cite[Lemma 1.1.5]{kisin-modules}.
	Note that Kisin assumes the residue field of the discrete valuation base field ($L$ in our case) to be perfect.
	However, the proof of loc.\ cit.\ depends on the results of \cite[\S 7-\S 8]{lazard}, \cite[Lemma 2.4]{kedlaya-monodromy} and \cite[Proposition 4.12 \& Lemme 4.13]{berger-differentielles}, where the proof of the latter depends on \cite{lazard} and \cite{helmer}.
	Relevant results of \cite{lazard}, \cite{kedlaya-monodromy} and \cite{helmer} are independent of the structure of the residue field of $L$.
	Hence, we get the claim by using the proof of \cite[Lemma 1.1.5]{kisin-modules} almost verbatim.
\end{rem}

Next, we note some useful facts about $\varphi\modules$ over $\AL^+$.

\begin{lem}\label{lem:finitefree_2dimregularlocal}
	Let $O_K := O_F$, $O_L$ or $\OLbreve$ and let $A := O_K \llbracket \mu \rrbracket$ equipped with a Frobenius endorphism extending the natural Frobenius on $O_K$ by setting $\varphi(\mu) = (1+\mu)^p-1$.
	Let $N$ be a finitely generated $A\module$ equipped with a Frobenius-semilinear endomorphism such that $1 \otimes \varphi : \varphi^*(N)[1/[p]_q] \isomorphic N[1/[p]_q]$.
	Then, $N[1/p]$ is finite free over $A[1/p]$ .
\end{lem}
\begin{proof}
	The proof is essentially the same as \cite[Proposition 4.3]{bhatt-morrow-scholze-1}.
	Let $J$ denote the smallest non-zero Fitting ideal of $N$ over $A$.
	Set $K := O_K[1/p]$ and $\Abar = A/J$.
	From loc.\ cit.\ the claim can be reduced to checking that $\Abar[1/p] = 0$.
	Note that the Frobenius endomorphism on $A$ and finite height condition on $N$ are different from loc.\ cit.
	Therefore, we need some modifications in the arguments of loc. cit.; we point out the differences in terms of their notations.
	Fix an algebraic closure $\Kbar$ of $K$ and consider the finite set $Z := \Spec(\Abar[1/p])(\Kbar)$ of $\Kbar\textrm{-valued}$ points of $\Abar[1/p]$.
	Let $Z' := \{x \in \frakm \textrm{ such that } (1+x)^p-1 \in Z\}$, where $\frakm \subset O_{\Kbar}$ is the maximal ideal.
	Then, from the equality $(A/J)[1/[p]_q] = (A/\varphi(J))[1/[p]_q]$, we get that $Z \cap U = Z' \cap U$, where $U := \frakm - \{\zeta_p-1, \ldots, \zeta_p^{p-1}-1\}$.
	Now, all the arguments from loc.\ cit.\ can be easily adapted to show that there exists some $r \in \NN$ such that we have an isomorphism $K[\mu]/(\mu^r) \isomorphic K[\mu]/(\varphi(\mu)^r)$.
	But, then we obtain that $(\varphi(\mu)/\mu)^r$ is a unit in $K[\mu]$, whereas $\varphi(\mu)/\mu \in K[\mu]$ is an irreducible polynomial.
	Hence, we must have that $r = 0$ and thus $(A/J)[1/p] = 0$, allowing us to conclude.
\end{proof}

\begin{rem}\label{rem:intersect_finitefree}
	Let $N$ be a finitely generated torsion-free $\AL^+\module$.
	Then $D = \AL \otimes_{\AL^+} N$ is a finite free $\AL\module$ and $N \subset D$ an $\AL^+\textrm{-submodule}$.
	Moreover, the $\AL^+\module$ $N' = N[1/p] \cap D$ is finite free.
	The claim essentially follows from \cite[Proposition B.1.2.4]{fontaine-phigamma}.
	Note that Fontaine assumes the residue field of the discrete valuation base field ($L$ in our case) to be perfect.
	However, the proof of \cite[Proposition B.1.2.4]{fontaine-phigamma} only depends on \cite[Chapter 5, Theorem 3.1]{lang} which is independent of the structure of the residue field of $L$.
	Therefore, one can adpat Fontaine's proof verbatim to show that $N'$ is finite free.
\end{rem}

Let $N$ be a finite free $\ALbreve^+\module$.
Say that $N$ is \textit{effective} and of \textit{finite $\pqheight$} if $N$ is equipped with a Frobenius-semilinear endomorphism $\varphi$ such that the natural map $1 \otimes \varphi : \varphi^*(N) \rightarrow N$ is injective and its cokernel is killed by some finite power of $[p]_q$.

Let $D_{\Lbreve}$ be a finite free \'etale $\varphi\module$ over $\ALbreve$.
Let $\pazs(D_{\Lbreve})$ denote the set of all finitely generated $\ALbreve^+\textrm{-submodules}$ $M \subset D_{\Lbreve}$ such that $M$ is stable under the induced $\varphi$ from $D_{\Lbreve}$, and the cokernel of the injective map $1 \otimes \varphi : \varphi^*(M) \rightarrow M$ is killed by some finite power of $[p]_q$.
In \cite[\S B.1.5.5]{fontaine-phigamma}, Fontaine functorially attached to $D_{\Lbreve}$ an $\ALbreve^+\textrm{-submodule}$ $\jpluss(D_{\Lbreve}) := \cup_{M \in \pazs(D_{\Lbreve})} M \subset D_{\Lbreve}$ (Fontaine uses the notation $j^q_*$ to denote the functor $\jpluss$; we change notations to avoid the obvious confusion).

\begin{lem}\label{lem:fontaine_jqs}
	The $\ALbreve^+\module$ $\jpluss(D_{\Lbreve})$ is free of rank $\leqslant \rank_{\ALbreve} D_{\Lbreve}$.	
	Moreover, if $N$ is an effective $\ALbreve^+\module$ of finite $\pqheight$, then the cokernel of the injective map $N \rightarrow \jpluss(\ALbreve \otimes_{\ALbreve^+} N)$ is killed by some finite power of $\mu$.
\end{lem}
\begin{proof}
	The first claim is shown in \cite[\S B.1.5.5]{fontaine-phigamma}.
	For the second claim note that $N$ is finite free over $\ALbreve^+$ and of finite $\pqheight$, therefore, by the equivalence shown in \cite[Proposition B.1.3.3]{fontaine-phigamma} we get that $N$ is $p\textrm{-\'etale}$ in the sense of \cite[\S B.1.3.1]{fontaine-phigamma}.
	In particular, we get that $D_{\Lbreve} = \ALbreve \otimes_{\ALbreve^+} N$ is an \'etale $\varphi\module$ and $N \in \pazs(D_{\Lbreve})$.
	Now, from \cite[Proposition B.1.5.6]{fontaine-phigamma}, it follows that the cokernel of the injective map $N \rightarrow \jpluss(D_{\Lbreve})$ is killed by some finite power of $\mu$.
\end{proof}

\subsection{\texorpdfstring{$\padic$}{-} representations and \texorpdfstring{$(\varphi, \Gamma)\modules$}{-}}\label{subsec:phigamma_mod_imperfect}

Let $T$ be a finite free $\ZZ_p\representation$ of $G_L$.
From the theory of $(\varphi, \Gamma_L)\modules$ (see \cite{fontaine-phigamma} and \cite{andreatta-phigamma}), one can functorially associate to $T$ a finite free \'etale $(\varphi, \Gamma_L)\module$ $\DL(T) := (\mbfa \otimes_{\ZZ_p} T)^{H_L}$ over $\AL$ of rank $= \rank_{\ZZ_p} T$, i.e.\ $\DL(T)$ is equipped with a continuous and semilinear action of $\Gamma_L$ and a Forbenius-semilinear endomorphism $\varphi$ commuting with $\Gamma$ and such that the natural map $1 \otimes \varphi : \varphi^*(\DL(T)) \rightarrow \DL(T)$ is an isomorphism.
Moreover, we have  $\Dtilde_L(T) := (\Atilde \otimes_{\ZZ_p} T)^{H_L} \isomorphic \Atilde^{H_L} \otimes_{\AL} \DL(T)$.
Furthermore, by the theory of overconvergence of $\padic$ and $\ZZ_p\textrm{-representations}$ (see \cite{cherbonnier-colmez} and \cite{andreatta-brinon}), one can functorially associate to $T$ a finite free \'etale $(\varphi, \Gamma_L)\module$ $\DLdag(T) := (\Adag \otimes_{\ZZ_p} T)^{H_L}$ over $\ALdag$ of rank $= \rank_{\ZZ_p} T$ and such that $\AL \otimes_{\ALdag} \DLdag(T) \isomorphic \DL(T)$.
We have natural isomorphisms
\begin{equation}\label{eq:phigamma_comp_iso_imperfect}
	\mbfa \otimes_{\AL} \DL(T) \isomorphic \mbfa \otimes_{\ZZ_p} T, \hspace{4mm} \Adag \otimes_{\ALdag} \DLdag(T) \isomorphic \Adag \otimes_{\ZZ_p} T,
\end{equation}
compatible with $(\varphi, \Gamma_L)\textrm{-actions}$.
More generally, the constructions described above are functorial and induce equivalence of categories
\begin{equation}\label{eq:rep_phigamma_imperfect}
	\Rep_{\ZZ_p}(G_L) \isomorphic (\varphi, \Gamma_L)\Mod_{\AL}^{\etale} \lisomorphic (\varphi, \Gamma_L)\Mod_{\ALdag}^{\etale}.
\end{equation}
Similar statements are also true for $\padic$ representations of $G_L$.
For a $\padic$ representation $V$ of $G_L$, set $\DrigLdag(V) := \BrigLdag \otimes_{\BLdag} \DLdag(V)$, which is the unique finite free $(\varphi, \Gamma_L)\module$ over $\BrigLdag$ of rank $=\dim_{\QQ_p} V$ and pure of slope 0 functorially attached to $V$ (see \cite{berger-differentielles}, \cite{kedlaya-slope} and \cite{ohkubo-differential}).
Moreover, the preceding functor induces an equivalence of categories between $\padic$ representations of $G_L$ and finite free $(\varphi, \Gamma_L)\modules$ over $\BrigLdag$ which are pure of slope 0 (see \cite[Lemma 4.5.7]{ohkubo-differential}), and we have a natural $(\varphi, G_L)\equivariant$ isomorphism,
\begin{equation}\label{eq:brigbatdag_iso}
	\Brigtildedag \otimes_{\BrigLdag} \DrigLdag(V) \isomorphic \Brigtildedag \otimes_{\QQ_p} V.
\end{equation}

\begin{rem}
	We have variations of the results mentioned above for $\padic$ (resp.\ $\ZZ_p\textrm{-representations}$) of $G_{\Lbreve}$ as well (see \cite{fontaine-phigamma}, \cite{cherbonnier-colmez} and \cite{berger-differentielles} for details).
\end{rem}

Finally, let $V$ be a $\padic$ representation of $G_L$ and $T \subset V$ a $G_L\textrm{-stable}$ $\ZZ_p\textrm{-lattice}$.
Since $G_{\Lbreve}$ is a subgroup of $G_L$, therefore, by restriction $V$ is a $\padic$ representation of $G_{\Lbreve}$ and $T \subset V$ a $G_{\Lbreve}\textrm{-stable}$ $\ZZ_p\textrm{-lattice}$.
Furthermore, we have a $\GammaLbreve\textrm{-equivariant}$ embedding $\AL \subset \ALbreve$ (via the map $[X_i^{\flat}] \mapsto X_i$) and thus we have isomorphisms of \'etale $(\varphi, \GammaLbreve)\modules$ $\DLbreve(T) \isomorphic \ALbreve \otimes_{\AL} \DL(T)$ and $\Dtilde_{\Lbreve}(T) := (\Atilde \otimes_{\ZZ_p} T)^{H_{\Lbreve}} \isomorphic \Atilde^{H_{\Lbreve}} \otimes_{\ALbreve} \DLbreve(T)$.
Similar statements are also true for $V$.

\subsection{Crystalline representations}\label{subsec:crysrep_imperfect}

Let $\Rep_{\QQ_p}^{\crys}(G_L)$ denote the category of $\padic$ crystalline representations of $G_L$ (see \cite[\S 3.3]{brinon-imparfait}) and let $\MF_L\wa(\varphi, \partial)$ denote the category of weakly admissible filtered $(\varphi, \partial)\modules$ over $L$ (see \cite[D\'efinition 4.21]{brinon-imparfait}).
Then, the following functor induces an exact equivalence of $\otimes\textrm{-categories}$:
\begin{align}\label{eq:odcrysl_func}
	\begin{split}
		\Rep_{\QQ_p}^{\crys}(G_L) &\isomorphic \MF_L\wa(\varphi, \partial)\\
		V &\longmapsto \ODcrysL(V) := (\OBcrys(\OLbar) \otimes_{\QQ_p} V)^{G_L},
	\end{split}
\end{align}
with an exact quasi-inverse given as $D \mapsto \OVcrysL(D) := (\Fil^0(\OBcrys(\OLbar) \otimes_L D))^{\partial=0, \varphi=1}$ (see \cite[Corollaire 4.37]{brinon-imparfait}).
In particular, if $V$ is a $\padic$ crystalline representation of $G_L$, then $\ODcrysL(V)$ is a rank $=\dim_{\QQ_p} V$, weakly admissible filtered $(\varphi, \partial)\module$ over $L$.
Moreover, as a representation of $G_{\Lbreve}$ one can attach to $V$ a rank $= \dim_{\QQ_p} V$, filtered $\varphi\module$ over $\Lbreve$, denoted as $\DcrysLbreve(V)$.
Now, note that since $V$ is crystalline for $G_L$, therefore, we have a $(\varphi, G_L)\equivariant$ isomorphism $\OBcrys(\OLbar) \otimes_L \ODcrysL(V) \isomorphic \OBcrys(\OLbar) \otimes_{\QQ_p} V$.
By base changing it along the $(\varphi, G_{\Lbreve})\equivariant$ surjection $\OBcrys(\OLbar) \twoheadrightarrow \Bcrys(\OLbar) = \Bcrys(O_{\Lbrevebar})$, sending $X_i \mapsto [X_i^{\flat}]$ for $1 \leqslant i \leqslant d$, we obtain the following $(\varphi, G_{\Lbreve})\equivariant$ isomorphism (also see the proof of \cite[Proposition 4.14]{brinon-trihan}):
\begin{equation*}
	\Bcrys(O_{\Lbrevebar}) \otimes_L \ODcrysL(V) \isomorphic \Bcrys(O_{\Lbrevebar}) \otimes_{\QQ_p} V,
\end{equation*}
where in the left hand term we consider the $(\varphi, G_{\Lbreve})\equivariant$ composition $L \rightarrow \OBcrys(\OLbar) \twoheadrightarrow \Bcrys(O_{\Lbrevebar})$, where the first map is the non-canonical $L\algebra$ structure on $\OBcrys(\OLbar)$ (see \S \ref{subsubsec:crystalline_rings}).
By taking $G_{\Lbreve}\textrm{-invariants}$ in the preceding isomorphism, we obtain an isomorphism of filtered $\varphi\modules$ over $\Lbreve$:
\begin{equation}\label{eq:dcrys_llbreve_comp}
	\Lbreve \otimes_L \ODcrysL(V) \isomorphic \DcrysLbreve(V).
\end{equation}
The representation $V$ is said to be positive if all its Hodge-Tate weights are $\leqslant 0$, and in this case we have that $\ODcrysL(V) = (\OBcrys^+(\OLbar) \otimes_{\QQ_p} V)^{G_L}$.

\begin{lem}\label{lem:faltings_isomorphism_dcrys}
	There exists a natural $\Bcrys(\OLbar)\textrm{-linear}$ and Frobenius-equivariant isomorphism,
	\begin{equation*}
		\Bcrys(\OLbar) \otimes_{\QQ_p} V \isomorphic (\OBcrys(\OLbar) \otimes_L \ODcrysL(V))^{\partial = 0} \isomorphic \Bcrys(\OLbar) \otimes_L \ODcrysL(V),
	\end{equation*}
	induced by the surjective map $\OBcrys(\OLbar) \twoheadrightarrow \Bcrys(\OLbar)$, sending $X_i \mapsto [X_i^{\flat}]$ for $1 \leqslant i \leqslant d$.
\end{lem}
\begin{proof}
	Let $J\OBcrys(\OLbar) := \padic \textrm{ closure of the ideal } ([X_1^{\flat}]-X_1, \ldots, [X_d^{\flat}]-X_d) \subset \OBcrys(\OLbar)$.
	Then, we have a projection,
	\begin{equation}\label{eq:bcryslin_projection}
		\OBcrys(\OLbar) \otimes_L \ODcrys(V) \longrightarrow \Bcrys(\OLbar) \otimes_L \ODcrys(V),
	\end{equation}
	via the map $X_i \mapsto [X_i^{\flat}]$, with the kernel given as $J\OBcrys(\OLbar) \otimes_L \ODcrys(V)$.
	Moreover, recall that we have a connection $\partial : \ODcrysL(V) \rightarrow \ODcrysL(V) \otimes_{O_L} \Omega^1_{O_L/O_F}$, given as $\partial(x) = \sum_{i=1}^d \partial_i(x) dX_i$, for differential operators $\partial_i$ on $\ODcrysL(V)$. 
	Then, using the non-canonical $L\algebra$ structure on $\OBcrys(\OLbar)$ (see \S \ref{subsubsec:crystalline_rings}), we can give an $L\textrm{-linear}$ map,
	\begin{align}\label{eq:dcrys_section}
		\begin{split}
			\ODcrysL(V) &\longrightarrow \OBcrys(\OLbar) \otimes_L \ODcrysL(V)\\
			x &\longmapsto \textstyle \sum_{\smbfk \in \NN^d} \prod_{i=1}^d \partial_i^{k_i}(x) \prod_{i=1}^d ([X_i^{\flat}] - X_i)^{[k_i]},
		\end{split}
	\end{align}
	where we write $\prod_{i=1}^d \partial_i^{k_i}(x) = \partial_1^{k_1} \circ \cdots \circ \partial_d^{k_d}(x)$ for notational convenience.
	As the connection $\partial$ on $\ODcrysL(V)$ is $p\textrm{-adically}$ quasi-nilpotent, therefore, $\prod_{i=1}^d \partial_i^{k_i}(x)$ goes $p\textrm{-adically}$ to $0$ as $\sum_{i=1}^d k_i \rightarrow +\infty$, and thus, in \eqref{eq:dcrys_section}, the formula on the right converges in the target, for its natural topology (see Remark \ref{rem:odcrys_scalarext_top}).
	Moreover, note that the map in \eqref{eq:dcrys_section} extends $\Bcrys(\OLbar)\textrm{-linearly}$ to a map,
	\begin{align}\label{eq:bcryslin_section}
		\begin{split}
			\Bcrys(\OLbar) \otimes_L \ODcrysL(V) &\longrightarrow \OBcrys(\OLbar) \otimes_L \ODcrysL(V),\\
			a \otimes x &\longmapsto a \otimes \textstyle \sum_{\smbfk \in \NN^d} \prod_{i=1}^d \partial_i^{k_i}(x) \prod_{i=1}^d ([X_i^{\flat}] - X_i)^{[k_i]},
		\end{split}
	\end{align}
	and it provides a section to the projection in \eqref{eq:bcryslin_projection}.
	In particular, we obtain a $\Bcrys(\OLbar)\textrm{-linear}$ direct sum decomposition,
	\begin{equation*}
		\OBcrys(\OLbar) \otimes_L \ODcrysL(V) = (J\OBcrys(\OLbar) \otimes_L \ODcrysL(V)) \oplus (\Bcrys(\OLbar) \otimes_L \ODcrysL(V)).
	\end{equation*}
	Note that the image of the section in \eqref{eq:bcryslin_section} lies in $(\OBcrys(\OLbar) \otimes_L \ODcrysL(V))^{\partial=0}$.
	Moreover, since $V$ is crystalline, we have that $\OBcrys(\OLbar) \otimes_L \ODcrysL(V) \isomorphic \OBcrys(\OLbar) \otimes_L V$, and it is easy to see that $(J \OBcrys(\OLbar) \otimes_L \ODcrysL(V))^{\partial = 0} = 0$.
	Therefore, from the direct sum decomposition it follows that we have $(\OBcrys(\OLbar) \otimes_L \ODcrysL(V))^{\partial = 0} \isomorphic \Bcrys(\OLbar) \otimes_L \ODcrysL(V)$.
	Note that the maps in \eqref{eq:bcryslin_projection} and \eqref{eq:bcryslin_section} are evidently compatible with Frobenius, therefore, the isomorphism in the claim is compatible with Frobenius.
	This allows us to conclude.
\end{proof}

\begin{rem}\label{rem:odcrys_scalarext_top}
	Note that the $\varphi\module$ $\ODcrysL(V)$ is a finite dimensional $L\textrm{-vector space}$.
	Therefore, $\OBcrys(\OLbar) \otimes_L \ODcrysL(V)$ is a finite free module over $\OBcrys(\OLbar)$ equipped with the natural product topology induced from the $\padic$ topology on $\OBcrys(\OLbar)$.
	Similarly, in Lemma \ref{lem:odcrys_gammal_action}, we will see that $\BrigL^+ \otimes_L \ODcrysL(V)$ is a finite free module over $\BrigL^+$ and equipped with the natural product topology induced from the Fr\'echet topology on $\BrigL^+$ (see Remark \ref{rem:brigl_topology}).
\end{rem}

\begin{rem}\label{rem:glaction_odcrys}
	Using the $\Bcrys(\OLbar)\textrm{-linear}$ map in \eqref{eq:bcryslin_section}, we equip $\Bcrys(\OLbar) \otimes_L \ODcrysL(V)$ with a continuous $G_L\action$ by transport of structure.
	In particular, for any $g$ in $G_L$, its action on $a \otimes x$ in $\Bcrys(\OLbar) \otimes_L \ODcrysL(V)$ is given as $g(a \otimes x) = g(a) \otimes \sum_{\smbfk \in \NN^d} \prod_{i=1}^d \partial_i^{k_i}(x) \prod_{i=1}^d (g([X_i^{\flat}]) - [X_i^{\flat}])^{[k_i]}$.
\end{rem}

\begin{rem}\label{rem:b+dcrys_hlinv}
	Using the description in Remark \ref{rem:glaction_odcrys}, we have that $\Bcrys^+(\OLbar) \otimes_L \ODcrysL(V) \subset \Bcrys(\OLbar) \otimes_L \ODcrysL(V)$ is stable under the action of $G_L$ as well.
	Moreover, note that the action of $H_L$, induced from the action of $G_L$ described in Remark \ref{rem:glaction_odcrys}, is trivial on $\ODcrysL(V) \subset \Bcrys(\OLbar) \otimes_L \ODcrysL(V)$, and we have that $\Bcrys^+(\OLbar)^{H_L} = \Bcrys^+(\OLinfty)$ by \cite[Lemma 4.32]{morrow-tsuji}.
	Therefore, we get that,
	\begin{equation}\label{eq:b+dcrys_hlinv}
		(\Bcrys^+(\OLbar) \otimes_L \ODcrysL(V))^{H_L} = \Bcrys^+(\OLinfty) \otimes_L \ODcrysL(V).
	\end{equation}
	We equip $\Bcrys^+(\OLinfty) \otimes_L \ODcrysL(V)$ with the residual $\Gamma_L\action$.
\end{rem}

\begin{lem}\label{lem:odcrys_gammal_action}
	For $x \in \ODcrysL(V)$ and $g \in \Gamma_L$, the series $\sum_{\smbfk \in \NN^d} \prod_{i=1}^d \partial_i^{k_i}(x) \prod_{i=1}^d (g([X_i^{\flat}]) - [X_i^{\flat}])^{[k_i]}$ converges in $\BrigL^+ \otimes_L \ODcrysL(V)$.
	In particular, $\BrigL^+ \otimes_L \ODcrysL(V) \subset \Bcrys^+(\OLinfty) \otimes_L \ODcrysL(V)$ is stable under $\Gamma_L\action$.
\end{lem}
\begin{proof}
	Let $\{\gamma_0, \gamma_1, \ldots, \gamma_d\}$ be topological generators of $\Gamma_L$ as in \S \ref{subsubsec:crystalline_rings}, in particular, $\gamma_j([X_i]^{\flat}) = (1+\mu)[X_i^{\flat}]$, if $i=j$, and $0$, otherwise.
	As $\BrigL^+ \otimes_L \ODcrysL(V) \subset \Bcrys^+(\OLinfty) \otimes_L \ODcrysL(V)$ is closed for the $\padic$ topology and the action of $\Gamma_L$ on the latter is continuous (see Remark \ref{rem:glaction_odcrys}), therefore, it is enough to show the claim for the chosen topological generators of $\Gamma_L$.
	For $\gamma_j$, we can simplify the sum in the claim and rewrite it as $\sum_{\smbfk \in \NN^d} \mu^{[k_j]} [X_j^{\flat}] \prod_{i=1}^d \partial_i^{k_i}(x)$.
	Recall that the connection $\partial$ on $\ODcrysL(V)$ is $p\textrm{-adically}$ quasi-nilpotent, i.e.\ there exists an $O_L\textrm{-lattice}$ $M \subset \ODcrys(V)$ stable under $\partial$, i.e.\ $\partial : M \rightarrow M \otimes \Omega^1_{O_L}$ such that $\partial$ is nilpotent modulo $p$.
	Let $\{e_1, \ldots, e_h\}$ denote an $O_L\textrm{-basis}$ of $M$.
	Then, we may check on the chosen basis that we have $\varphi(M) \subset p^{-r} M$, for some fixed $r \in \NN$.
	Moreover, recall that we have $L \otimes_{\varphi, L} \ODcrysL(V) \isomorphic \ODcrysL(V)$, so we may write $x = \sum_{j=1}^h a_j \varphi(e_j)$, for some $a_j \in L$.
	Since $\partial_i(\varphi(e_j)) = p \varphi(\partial_i(e_j))$, for all $1 \leqslant i \leqslant d$ and $1 \leqslant j \leqslant h$, therefore, we get that,
	\begin{equation*}
		\textstyle \sum_{\smbfk \in \NN^d} \mu^{[k_j]} [X_j^{\flat}] \prod_{i=1}^d \partial_i^{k_i}(\varphi(e_i)) = p^{-dr}\sum_{\smbfk \in \NN^d} \mu^{[k_j]} [X_i^{\flat}] \prod_{i=1}^d p^{k_i}p^r\varphi(\partial_i^{k_i}(e_i)),
	\end{equation*}
	converges $p\textrm{-adically}$, and thus converges for the natural topology on $\BrigL^+ \otimes_L \ODcrysL(V)$ (see Remark \ref{rem:odcrys_scalarext_top}).
	Therefore, by using the Leibniz rule, we are reduced to showing that the summation $\sum_{\smbfk \in \NN^d} \mu^{[k_j]} [X_j^{\flat}] \prod_{i=1}^d \partial_i^{k_i}(a)$ converges in $\BrigL^+$, for any $a$ in $L$.
	This follows easily since we have $\partial_i^k(X_i^n)/k! = 0$, for $n < k$, $\partial_i^k(X_i^n)/k! = \binom{n}{k} X_i^{n-k}$, for $n \geqslant k$, and $\partial_i^k(X_i^{-n})/k! = (-1)^k\binom{n+k-1}{k} X_i^{-(n+k)}$, for $n \in \NN$.
	Hence, the lemma is proved.
\end{proof}

\begin{lem}\label{lem:brigodcris_modmu}
	The action of $\Gamma_L$ on $\BrigL^+ \otimes_L \ODcrysL(V)$ is trivial modulo $\mu$.
\end{lem}
\begin{proof}
	Note that $g(\mu) = (1+\mu)^{\chi(g)}-1$, for any $g$ in $\Gamma_L$ and $\chi$ the $\padic$ cyclotomic character.
	Using Lemma \ref{lem:odcrys_gammal_action} and Remark \ref{rem:glaction_odcrys}, for $a \otimes x$ in $\BrigL^+[\mu/t] \otimes_L \ODcrysL(V)$ and $g \in \Gamma_L$, we have that $g(a \otimes x) = g(a) \otimes \sum_{\smbfk \in \NN^d} \prod_{i=1}^x \partial_i^{k_i}(x) \prod_{i=1}^d (g([X_i^{\flat}])-[X_i^{\flat}])^{[k_i]}$.
	Note that 
	\begin{equation}\label{eq:brigodcris_modmu}
		(g-1)(a \otimes x) = ((g-1)a) \otimes x + g(a) \otimes ((g-1)x),
	\end{equation}
	where $g(x)$ is given by the series in the statement of Lemma \ref{lem:odcrys_gammal_action}.
	So, we have that $(g-1)x = \sum_{\smbfk \in \NN_+^d} \prod_{i=1}^x \partial_i^{k_i}(x) \prod_{i=1}^d ((g-1)[X_i^{\flat}])^{[k_i]}$, where $\NN_+^d = \NN^d \setminus \{(0, 0, \ldots, 0)\}$.
	Using the explicit description of $\BrigL^+$ in Lemma \ref{lem:brigl+_explicit}, note that $(g-1)\BrigL^+ \subset \mu \BrigL^+$, and from the proof of Lemma \ref{lem:odcrys_gammal_action} note that $(g-1)[X_i^{\flat}]$ is in $\mu \BL^+$.
	Therefore, an argument similar to the proof of Lemma \ref{lem:odcrys_gammal_action} shows that $(g-1)x$ converges in $\mu \BrigL^+ \otimes_L \ODcrysL(V)$.
	So, from \eqref{eq:brigodcris_modmu} it follows that $(g-1)(a \otimes x)$ is in $\mu \BrigL^+ \otimes_L \ODcrysL(V)$.
	This allows us to conclude.
\end{proof}

\section{Wach modules}\label{sec:wach_modules}

In this section, we will define and study Wach modules in the imperfect residue field case and finite $\pqheight$ representations of $G_L$ and relate them to crystalline representations.
Our definition is a direct and natural generalisation of Wach modules in the perfect residue field case (see \cite[D\'efinition III.4.1]{berger-limites}).

\subsection{Wach modules over \texorpdfstring{$\AL^+$}{-}}\label{subsec:wach_mod_props}

In the period ring $\Ainf(\OFinfty)$, let us fix $q := [\varepsilon]$, $\mu := q-1 = [\varepsilon]-1$ and $[p]_q := \tilde{\xi} := \varphi(\mu)/\mu$.

\begin{defi}\label{defi:wach_mod_imperfect}
	Let $a, b \in \ZZ$ with $b \geqslant a$.
	A \textit{Wach module} over $\AL^+$ with weights in the interval $[a, b]$ is a finite free $\AL^+\module$ $N$ equipped with a continuous and semilinear action of $\Gamma_{L}$ and satisfying the following assumptions:
	\begin{enumarabicup}
	\item The action of $\Gamma_L$ on $N/\mu N$ is trivial.

	\item There is a Frobenius-semilinear operator $\varphi: N[1/\mu] \rightarrow N[1/\varphi(\mu)]$, commuting with the action of $\Gamma_L$, and such that $\varphi(\mu^b N) \subset \mu^b N$ and the cokernel of the induced injective map $(1 \otimes \varphi) : \varphi^{\ast}(\mu^b N) \rightarrow \mu^b N$ is killed by $[p]_q^{b-a}$.
	\end{enumarabicup}
	Define the $[p]_q\textit{-height}$ of $N$ to be the largest value of $-a$, for $a \in \ZZ$ as above.
	Say that $N$ is \textit{effective} if one can take $b = 0$ and $a \leqslant 0$.
	A Wach module over $\BL^+$ is a finitely generated module $M$ equipped with a Frobenius-semilinear operator $\varphi: M[1/\mu] \rightarrow M[1/\varphi(\mu)]$, commuting with the action of $\Gamma_L$, and such that there exists a $\varphi\textrm{-stable}$ (after inverting $\mu$) and $\Gamma_L\textrm{-stable}$ $\AL^+\submodule$ $N \subset M$, with $N$ a Wach module over $\AL^+$ (equipped with the induced $(\varphi, \Gamma_L)\action$) and $N[1/p] = M$.
\end{defi}
Denote the category of Wach modules over $\AL^+$ as $(\varphi, \Gamma)\Mod_{\AL^+}^{[p]_q}$ with morphisms between objects being $\AL^+\linear$, $\Gamma_L\equivariant$ and $\varphi\equivariant$ morphisms (after inverting $\mu$).

\begin{defi}\label{defi:nygaard_fil}
	Let $N$ be a Wach module over $\AL^+$.
	Define a decreasing filtration on $N$ called the \textit{Nygaard filtration}, for $k \in \ZZ$, as
	\begin{equation*}
		\Fil^k N := \{ x \in N \textrm{ such that } \varphi(x) \in [p]_q^k N\}.
	\end{equation*}
	From the definition, it is clear that $N$ is effective if and only if $\Fil^0 N = N$.
	Similarly, we can define a Nygaard filtration on $M := N[1/p]$ and it satisfies $\Fil^k M = (\Fil^k N)[1/p]$.
\end{defi}

Extending scalars along $\AL^+ \rightarrow \AL$ induces a functor $(\varphi, \Gamma)\Mod_{\AL^+}^{[p]_q} \rightarrow (\varphi, \Gamma)\Mod_{\AL}^{\etale}$, and we make the following claim:
\begin{prop}\label{prop:wach_etale_ff_imperfect}
	The following natural functor is fully faithful:
	\begin{align*}
		(\varphi, \Gamma)\Mod_{\AL^+}^{[p]_q} &\longrightarrow (\varphi, \Gamma)\Mod_{\AL}^{\etale}\\
		N &\longmapsto \AL \otimes_{\AL^+} N.
	\end{align*}
\end{prop}
\begin{proof}
	We need to show that for Wach modules $N$ and $N'$, we have a natural bijection,
	\begin{equation}\label{eqref:homset_bijection_imperfect}
		\Hom_{(\varphi, \Gamma)\Mod_{\AL^+}^{[p]_q}}(N, N') \isomorphic \Hom_{(\varphi, \Gamma)\Mod_{\AL}^{\etale}}(\AL \otimes_{\AL^+} N, \AL \otimes_{\AL^+} N').
	\end{equation}
	Note that $\AL^+ \rightarrow \AL = \AL^+[1/\mu]^{\wedge}$ is injective, in particular, the map in \eqref{eqref:homset_bijection_imperfect} is injective.
	To check that \eqref{eqref:homset_bijection_imperfect} is surjective, let $D_L := \AL \otimes_{\AL^+} N$, $D'_L := \AL \otimes_{\AL^+} N'$, and take an $\AL\linear$ and $(\varphi, \Gamma_L)\equivariant$ map $f : D_L \rightarrow D'_L$.
	Then, by base changing $f$ along the embedding $\AL \rightarrow \ALbreve$ (see \S \ref{subsubsec:periodrings_Lbreve}), we obtain an $\ALbreve\linear$ and $(\varphi, \GammaLbreve)\equivariant$ map $f_{\Lbreve} : D_{\Lbreve} \rightarrow D'_{\Lbreve}$.
	Using the definition and notation preceding Lemma \ref{lem:fontaine_jqs}, we further obtain an $\ALbreve^+\linear$ and $(\varphi, \GammaLbreve)\equivariant$ map $f_{\Lbreve} : \jpluss(D_{\Lbreve}) \rightarrow \jpluss(D'_{\Lbreve})$, where we abuse notations by writing $f_{\Lbreve}$ instead of $\jpluss(f_{\Lbreve})$.
	From Lemma \ref{lem:fontaine_jqs}, note that for some $s \in \NN$ and $N_{\Lbreve} := \ALbreve^+ \otimes_{\AL^+} N$, we have an inclusion $\mu^s N_{\Lbreve} \subset \jpluss(D_{\Lbreve})$ and the cokernel is killed by some finite power of $\mu$.
	Hence, $N_{\Lbreve}[1/\mu] \isomorphic \jpluss(D_{\Lbreve})[1/\mu]$.
	Similarly, one can also show that $N'_{\Lbreve}[1/\mu] \isomorphic \jpluss(D'_{\Lbreve})[1/\mu]$.

	Now, from the map $f_{\Lbreve} : \jpluss(D_{\Lbreve}) \rightarrow \jpluss(D'_{\Lbreve})$, we obtain an induced $\GammaLbreve\equivariant$ map $f_{\Lbreve} : N_{\Lbreve}[1/\mu] = \jpluss(D'_{\Lbreve})[1/\mu] \rightarrow \jpluss(D'_{\Lbreve})[1/\mu] = N'_{\Lbreve}[1/\mu]$, and from Lemma \ref{lem:wach_muinverse_ff} we get that $f_{\Lbreve}(N_{\Lbreve}) \subset N'_{\Lbreve}$.
	It is easy to see that $N = N_{\Lbreve} \cap D_L \subset D_{\Lbreve}$ and $N' = N'_{\Lbreve} \cap D'_L \subset D'_{\Lbreve}$.
	So, we conclude that $f(N) = f_{\Lbreve}(N_{\Lbreve}) \cap f(D_L) \subset N'_{\Lbreve} \cap D'_L = N'$.
	This proves the surjectivity of \eqref{eqref:homset_bijection_imperfect}.
\end{proof}

\begin{lem}\label{lem:wach_muinverse_ff}
	Let $N$ and $N'$ be Wach modules over $\ALbreve^+$ and let $f : N[1/\mu] \rightarrow N'[1/\mu]$ be an $\ALbreve^+\linear$ and $\GammaLbreve\equivariant$ map.
	Then $f(N) \subset N'$.
\end{lem}
\begin{proof}
	The proof is similar to the proof of \cite[Lemma 5.32]{abhinandan-relative-wach-i}.
	Assume that $f(N) \subset \mu^{-k} N'$, for some $k \in \NN$, and consider the reduction of $f$ modulo $\mu$, which is again $\GammaLbreve\equivariant$.
	By definition we have that $\GammaLbreve$ acts trivially over $N/\mu N$, whereas $\mu^{-k}N'/\mu^{-k+1}N' \isomorphic N'/\mu N' (-k)$, i.e.\ the action of $\GammaLbreve$ on $\mu^{-k}N'/\mu^{-k+1}N'$ is given by $\chi^{-k}$, where $\chi$ is the $\padic$ cyclotomic character, in particular, $(\mu^{-k}N'/\mu^{-k+1}N')^{\GammaLbreve} = 0$.
	Since $f$ is $\GammaLbreve\equivariant$, therfore, we must have that $k = 0$, i.e.\ $f(N) \subset N'$.
\end{proof}

Analogous to above, one can define categories $(\varphi, \Gamma)\Mod_{\BL^+}^{[p]_q}$ and $(\varphi, \Gamma)\Mod_{\BL}^{\etale}$ and a functor from the former to latter by extending scalars along $\BL^+ \rightarrow \BL$.
Then passing to associated isogeny categories in Proposition \ref{prop:wach_etale_ff_imperfect}, we obtain the following:
\begin{cor}\label{cor:wach_etale_ff_imperfect}
	The natural functor $(\varphi, \Gamma)\Mod_{\BL^+}^{[p]_q} \rightarrow (\varphi, \Gamma)\Mod_{\BL}^{\etale}$ is fully faithful.
\end{cor}

Composing the functor in Proposition \ref{prop:wach_etale_ff_imperfect} with the equivalence in \eqref{eq:rep_phigamma_imperfect}, we obtain a fully faithful functor,
\begin{align}\label{eq:wach_reps_imperfect}
	\begin{split}
		\TL : (\varphi, \Gamma)\Mod_{\AL^+}^{[p]_q} &\longrightarrow \Rep_{\ZZ_p}(G_L)\\
		N &\longmapsto \big(\mbfa \otimes_{\AL^+} N\big)^{\varphi = 1} \isomorphic \big(W(\CC_L^{\flat}) \otimes_{\AL^+} N\big)^{\varphi = 1}.
	\end{split}
\end{align}

\begin{lem}\label{lem:wachmod_comp_imperfect}
	Let $N$ be Wach module of $\pqheight$ $s$ and let $T := \TL(N)$.
	Then, we have a $G_L\equivariant$ isomorphism,
	\begin{equation}\label{eq:wachmod_comp_imperfect}
		\mbfa^+[1/\mu] \otimes_{\AL^+} N \isomorphic \mbfa^+[1/\mu] \otimes_{\ZZ_p} T.
	\end{equation}
	Moreover, if $N$ is effective, then we have $G_L\equivariant$ inclusions $\mu^s (\mbfa^+ \otimes_{\ZZ_p} T) \subset \mbfa^+ \otimes_{\AL^+} N \subset \mbfa^+ \otimes_{\ZZ_p} T$.
\end{lem}
\begin{proof}
	For $r \in \NN$ large enough, the Wach module $\mu^r N (-r)$ is always effective and we have that $\TL(\mu^rN(-r)) = \TL(N)(-r)$ (the twist $(-r)$ denotes a Tate twist on which $\Gamma_L$ acts via $\chi^{-r}$, where $\chi$ is the $\padic$ cyclotomic character).
	Therefore, it is enough to show both the claims for effective Wach modules.
	So assume that $N$ is effective.
	Since $N$ is finite free over $\AL^+$, therefore, by using Definition \ref{defi:wach_mod_imperfect} (2) and the tensor product Frobenius, we obtain a Frobenius-semilinear isomorphism $\varphi : \Ainf(\OLbar)[1/\xi] \otimes_{\AL^+} N \isomorphic \Ainf(\OLbar)[1/\xitilde] \otimes_{\AL^+} N$.
	Then, from \cite[Proposition 6.15]{morrow-tsuji} we get the following $G_L\equivariant$ inclusions:
	\begin{equation*}
		\mu^s (\Ainf(\OLbar) \otimes_{\ZZ_p} T) \subset \Ainf(\OLbar) \otimes_{\AL^+} N \subset \Ainf(\OLbar) \otimes_{\ZZ_p} T \subset \Atilde \otimes_{\AL^+} N.
	\end{equation*}
	Moreover, from \eqref{eq:phigamma_comp_iso_imperfect}, we have that $\mbfa \otimes_{\AL^+} N \isomorphic \mbfa \otimes_{\ZZ_p} T$.
	Therefore, by taking the intersection of $\mbfa _{\otimes} T$ with the chain of inclusions above, inside $\Atilde \otimes_{\AL^+} N \isomorphic \Atilde \otimes_{\ZZ_p} T$, we obtain the following $G_L\equivariant$ inclusions:
	\begin{equation*}
		\mu^s (\Ainf(\OLbar) \cap \mbfa) \otimes_{\ZZ_p} T \subset (\Ainf(\OLbar) \cap \mbfa) \otimes_{\AL^+} N \subset (\Ainf(\OLbar) \cap \mbfa) \otimes_{\ZZ_p} T.
	\end{equation*}
	Since $\mbfa^+ = \Ainf(\OLbar) \cap \mbfa$, therefore, we get that the natural map in \eqref{eq:wachmod_comp_imperfect} is bijective and $\mu^s (\mbfa^+ \otimes_{\ZZ_p} T) \subset \mbfa^+ \otimes_{\AL^+} N \subset \mbfa^+ \otimes_{\ZZ_p} T$ (for $N$ effective), as desired.
\end{proof}

\subsection{Finite \texorpdfstring{$[p]_q\textrm{-height}$}{-} representations}\label{subsec:finite_pqheight_reps}

In this subsection, we will generalise the definition of finite $\pqheight$ representations from \cite[Definition 4.9]{abhinandan-relative-wach-i} in the imperfect residue field case.
Let $T$ be a finite free $\ZZ_p\representation$ of $G_L$, $V := T[1/p]$ and set $\DL^+(T) := (\mbfa^+ \otimes_{\ZZ_p} T)^{H_L}$ to be the $(\varphi, \Gamma_L)\module$ over $\AL^+$ associated to $T$ and let $\DL^+(V) := \DL^+(T)[1/p]$ be the $(\varphi, \Gamma_L)\module$ over $\BL^+$ associated to $V$.

\begin{defi}\label{defi:finite_pqheight_imperfect}
	A \textit{finite $\pqheight$} $\ZZ_p\representation$ of $G_L$ is a finite free $\ZZ_p\module$ $T$ admitting a linear and continuous action of $G_L$ such that there exists a finite free $\AL^+\submodule$ $\NL(T) \subset \DL(T)$ satisfying the following:
	\begin{enumarabicup}
	\item $\NL(T)$ is a Wach module in the sense of Definition \ref{defi:wach_mod_imperfect}.

	\item We have a natural $(\varphi, \Gamma_L)\equivariant$ isomorphism $\AL \otimes_{\AL^+} \NL(T) \isomorphic \DL(T)$.
	\end{enumarabicup}
	Set the $\pqheight$ of $T$ to be the $\pqheight$ of $\NL(T)$. Say $T$ is \textit{positive} if $\NL(T)$ is effective.

	A finite $\pqheight$ $\padic$ representation of $G_L$ is a finite dimensional $\QQ_p\textrm{-vector space}$ admitting a linear and continuous action of $G_L$ such that there exists a $G_L\textrm{-stable}$ $\ZZ_p\textrm{-lattice}$ $T \subset V$ with $T$ of finite $\pqheight$.
	We set $\NL(V) = \NL(T)[1/p] \subset \DL(T)[1/p] = \DL(V)$ as a $\BL^+\submodule$ of $\DL(V)$, and satisfying properties analogous to (1) and (2) above.
	Set the $\pqheight$ of $V$ to be the $\pqheight$ of $T$.
	Say $V$ is \textrm{positive} if $\NL(V)$ is effective.
\end{defi}

\begin{rem}
	For $T$ a finite $\pqheight$ $\ZZ_p\representation$ of $G_L$ and $r \in \NN$.
	We set $\NL(T(r)) := \mu^{-r} \NL(T)(r)$, in particular, property of being finite $\pqheight$ is invariant under Tate twists.
\end{rem}

\begin{lem}\label{lem:finiteheight_props_imperfect}
	Let $T$ be a finite $\pqheight$ $\ZZ_p\representation$ of $G_L$.
	\begin{enumarabicup}
	\item If $T$ is positive, then $\mu^s \DL^+(T) \subset \NL(T) \subset \DL^+(T)$.

	\item The $\AL^+\module$ $\NL(T)$ is unique, i.e.\ if there exists an $\AL^+\submodule$ $N \subset \DL(T)$ satisfying the conditions (1) and (2) in Definition \ref{defi:finite_pqheight_imperfect}, then we must have that $N = \NL(T) \subset \DL(T)$.
	\end{enumarabicup}
\end{lem}
\begin{proof}
	Note that $\AL \otimes_{\AL^+} \NL(T) \isomorphic \DL(T)$ and this scalar extension is fully faithful by Proposition \ref{prop:wach_etale_ff_imperfect}, therefore, we obtain that $\TL(\NL(T)) \isomorphic T$ as representations of $G_L$ (here $\TL$ is the functor defined in \eqref{eq:wach_reps_imperfect}).
	This also implies that Lemma \ref{lem:wachmod_comp_imperfect} holds for $\NL(T)$, so by taking $H_L\textrm{-invariants}$ of the chain of inclusions in the final statement of Lemma \ref{lem:wachmod_comp_imperfect}, we obtain that $\mu^s \DL^+(T) \subset \NL(T) \subset \DL^+(T)$ which proves (1).
	The claim in (2) follows from Proposition \ref{prop:wach_etale_ff_imperfect}, or using an argument similar to \cite[Proposition 4.13]{abhinandan-relative-wach-i}.
\end{proof}

\begin{rem}\label{rem:finiteheight_props_imperfect}
	Let $V$ be a finite $\pqheight$ $\padic$ representation of $G_L$ and $T \subset V$ a finite $\pqheight$ $G_L\textrm{-stable}$ $\ZZ_p\textrm{-lattice}$.
	Then, we have that $\NL(V) = \NL(T)[1/p]$ and from Lemma \ref{lem:finiteheight_props_imperfect} we get that if $V$ is positive then $\mu^s \DL^+(V) \subset \NL(V) \subset \DL^+(V)$.
	Moreover, from Corollary \ref{cor:wach_etale_ff_imperfect} (or \cite[Proposition 4.13]{abhinandan-relative-wach-i}) it follows that $\NL(V)$ is unique, i.e.\ if there exists a $\BL^+$ module $M \subset \DL(V)$ satisfying the conditions analogous to (1) and (2) in Definition \ref{defi:finite_pqheight_imperfect}, then we must have that $M = \NL(V) \subset \DL(V)$.
	In particular, it follows that $\NL(V)$ is independent of the choice of the lattice $T \subset V$.
	Alternatively, note that since we have $\NL(V(r)) = \mu^{-r} \NL(V)(r)$, without loss of generality we may assume that $V$ is positive and $T' \subset V$ another finite $\pqheight$ $G_L\textrm{-stable}$ $\ZZ_p\textrm{-lattice}$.
	Then, we have that $\mu^s \DL^+(V) \subset \NL(T')[1/p] \subset \DL^+(V)$, and using the argument in the proof of \cite[Proposition 4.13]{abhinandan-relative-wach-i} almost verbatim gives $\NL(V) = \NL(T)[1/p] \isomorphic \NL(T')[1/p]$ compatible with the respective $(\varphi, \Gamma_L)\textrm{-actions}$.
\end{rem}

\begin{rem}\label{rem:finitepqheight_wach_imperfect}
	From the definition of finite $\pqheight$ representations, Lemma \ref{lem:finiteheight_props_imperfect} and the fully faithful functor in \eqref{eq:wach_reps_imperfect}, it follows that the data of a finite $\pqheight$ representation is equivalent to the data of a Wach module.
\end{rem}

\subsection{Wach modules are crystalline}\label{subsec:wachmod_crystalline}

The goal of this subsection is to prove Theorem \ref{thm:fh_crys_imperfect} and Corollary \ref{cor:qdeformation_dcrys}.
To prove our results, we need certain period rings similar to \cite[\S 4.3.1]{abhinandan-relative-wach-i}, which we denote as $\ALpi^{\PD}$ and $\OALpi^{\PD}$ below.
We define these as follows: 
let $\varpi = \zeta_p-1$ and set $\ALpi^+ := \AL^+[\varphi^{-1}(\mu)] \subset \Ainf(\OLinfty)$.
Restricting the map $\theta$ on $\Ainf(\OLinfty)$ (see \S \ref{subsubsec:crystalline_rings}) to $\ALpi^+$, we get a surjection $\theta: \ALpi^+ \twoheadrightarrow O_L[\varpi]$.
Define $\ALpi^{\PD}$ to be the $\padic$ completion of the divided power envelope of the map $\theta$ with respect to $\kert \theta$.
Moreover, consider the surjective map $\theta_L : O_L \otimes_{\ZZ} \ALpi^+ \twoheadrightarrow O_L[\varpi]$, given as $x \otimes y \mapsto x\theta(y)$.
Define $\OALpi^{\PD}$ to be the $\padic$ completion of the divided power envelope of the map $\theta_L$ with respect to $\kert \theta_L$.
Similar to \cite[\S 4.3.1]{abhinandan-relative-wach-i}, one can show that $\ALpi^{\PD} \subset \Acrys(\OLinfty)$ and $\OALpi^{\PD} \subset \OAcrys(\OLinfty)$, stable under the Frobenius and $\Gamma_L\action$ on latter.
We equip $\ALpi^{\PD}$ and $\OALpi^{\PD}$ with induced structures, in particular, a filtration (same as the filtration by divided powers of $\kert \theta$ and $\kert \theta_L$ respectively, see \cite[Remark 4.23]{abhinandan-relative-wach-i}) and a connection $\partial_A$ on $\OALpi^{\PD}$ satisfying Griffiths transversality and such that $(\OALpi^{\PD})^{\partial_A=0} = \ALpi^{\PD}$.
Furthermore, note that we have natural inclusions $\BL^+ \subset \ALpi^{\PD}[1/p] \subset \OALpi^{\PD}[1/p] \subset \OBcrys^+(\OLbar) \subset \OBdR^+(\OLbar)$, where the last term is the big de Rham period ring which is an integral domain (see \cite[Proposition 2.9 \& Remarque 2.10]{brinon-imparfait}).
So, it follows that $\ALpi^{\PD}[1/p]$ and $\OALpi^{\PD}[1/p]$ are torsion-free modules over the principal ideal domain $\BL^+$, hence flat.
A similar reasoning shows that $\Fil^k \OALpi^{\PD}[1/p]$ is a flat $\BL^+\module$, for any $k \in \NN$.

\begin{thm}\label{thm:fh_crys_imperfect}
	Let $N$ be a Wach module over $\AL^+$.
	Then, $V := \TL(N)[1/p]$ is a $\padic$ crystalline representation of $G_L$.
\end{thm}
\begin{proof}
	For $r \in \NN$ large enough, the Wach module $\mu^r N (-r)$ is always effective and we have that $\TL(\mu^rN(-r)) = \TL(N)(-r)$ (the twist $(-r)$ denotes a Tate twist on which $\Gamma_L$ acts via $\chi^{-r}$, where $\chi$ is the $\padic$ cyclotomic character).
	Therefore, it is enough to show the claim for effective Wach modules.
	So assume that $N$ is effective.
	Note that $N$ is free over $\AL^+$ and $\TL(N)$ is a finite $\pqheight$ $\ZZ_p\representation$ of $G_L$ in the sense of Definition \ref{defi:finite_pqheight_imperfect} (see Remark \ref{rem:finitepqheight_wach_imperfect}).
	So, the results of \cite[\S 4.3-\S 4.5]{abhinandan-relative-wach-i} can be adapted to the case of $O_L$ almost verbatim as all objects appearing in loc.\ cit.\ admit a natural variation for $O_L$.
	In particular, as we explain below, the proofs of \cite[Theorem 4.25, Proposition 4.28]{abhinandan-relative-wach-i} can be adapted to get that $V = \TL(N)[1/p]$ is a crystalline representation of $G_L$.

	Set $D_L := (\OALpi^{\PD} \otimes_{\AL^+} N[1/p])^{\Gamma_L} \subset \ODcrysL(V)$.
	Then, from Proposition \ref{prop:oalpd_comparison} it follows that $D_L$ is a finite $L\textrm{-vector space}$ of dimension $= \rank_{\AL^+} N$ equipped with a tensor product Frobenius and a connection induced from the connection on $\OALpi^{\PD}$ satisfying Griffiths transversality with respect to the tensor-product filtration defined as $\Fil^k D_L := \big(\sum_{i+j=k} \Fil^i \OALpi^{\PD} \otimes_{\AL^+} \Fil^j N[1/p]\big)^{\Gamma_L}$, where $N[1/p]$ is equipped with Nygaard filtration of Definition \ref{defi:nygaard_fil} (see after Lemma \ref{lem:osmpd_comp} for well-definedness of the tensor-product filtration).
	Moreover, from Proposition \ref{prop:oalpd_comparison} we have a natural isomorphism $\OALpi^{\PD} \otimes_{O_L} D_L \isomorphic \OALpi^{\PD} \otimes_{\AL^+} N[1/p]$.
	Now, consider the following diagram:
	\begin{equation}\label{eq:fh_crys_imperfect}
		\begin{tikzcd}
			\OBcrys(\OLbar) \otimes_L D_L \arrow[r, "\eqref{eq:oalpd_comparison}", "\sim"'] \arrow[d, "\eqref{eq:ml_in_dcrys}"'] & \OBcrys(\OLbar) \otimes_{\AL^+} N[1/p] \arrow[d, "\eqref{eq:wachmod_comp_imperfect}"', "\wr"]\\
			\OBcrys(\OLbar) \otimes_L \ODcrysL(V) \arrow[r] & \OBcrys(\OLbar) \otimes_{\QQ_p} V,
		\end{tikzcd}
	\end{equation}
	where the left vertical arrow is the extension of the inclusion $D_L \subset \ODcrysL(V)$, from \eqref{eq:ml_in_dcrys}, along the natural map $L \rightarrow \OBcrys(\OLbar)$, the top horizontal arrow is the extension of the isomorphism, in Proposition \ref{prop:oalpd_comparison}, along the natural map $\OALpi^{\PD}[1/p] \rightarrow \OBcrys(\OLbar)$, the right vertical arrow is the extension of the isomorphism \eqref{eq:wachmod_comp_imperfect}, in Lemma \ref{lem:wachmod_comp_imperfect}, along the natural map $\mbfa^+[1/\mu] \rightarrow \OBcrys(\OLbar)$ and the bottom horizontal arrow is the natural injective map (see \cite[Proposition 3.22]{brinon-imparfait}).
	Commutativity and compatibility of the diagram with $(\varphi, G_L)\action$ and connection follows from \eqref{eq:ml_in_dcrys}.
	Bijectivity of the top horizontal arrow and the right vertical arrow imply that the left vertical arrow and the bottom horizontal arrow are bijective as well.
	Hence, we obtain that $V$ is a crystalline representation of $G_L$.
\end{proof}

\begin{rem}\label{rem:dl_dcrys_comp}
	In diagram \eqref{eq:fh_crys_imperfect}, by taking the $G_L\textrm{-fixed part}$ of the left vertical arrow, we get that,
	\begin{equation}\label{eq:dl_dcrys_comp}
		D_L \isomorphic \ODcrysL(V),
	\end{equation}
	compatible with Frobenius and connection.
	Moreover, since the bottom horizontal arrow of the diagram \eqref{eq:fh_crys_imperfect} is compatible with filtrations (see \cite[Proposition 3.35]{brinon-imparfait}), an argument similar to the proof of \cite[Proposition 4.49]{abhinandan-relative-wach-i} shows that the isomorphism in \eqref{eq:dl_dcrys_comp} is compatible with filtrations, where we consider the Hodge filtration on $\ODcrysL(V)$.
\end{rem}

The following result was used in the proof of Theorem \ref{thm:fh_crys_imperfect}:
\begin{prop}\label{prop:oalpd_comparison}
	Let $N$ be an effective Wach module over $\AL^+$.
	Then, $D_L = \big(\OALpi^{\PD} \otimes_{\AL^+} N[1/p]\big)^{\Gamma_L}$ is a finite $L\textrm{-vector space}$ of dimension $= \rank_{\AL^+} N$ equipped with a Frobenius, a filtration and a connection satisfying Griffiths transversality with respect to the filtration.
	Moreover, we have a natural comparison isomorphism
	\begin{equation}\label{eq:oalpd_comparison}
		\OALpi^{\PD} \otimes_{O_L} D_L \isomorphic \OALpi^{\PD} \otimes_{\AL^+} N[1/p],
	\end{equation}
	compatible with the respective Frobenii, filtrations, connections and $\Gamma_L\textrm{-actions}$.
\end{prop}
\begin{proof}
	We will adapt the proof of \cite[Proposition 4.28]{abhinandan-relative-wach-i}.
	The main idea, as explained below, is to work over a new period ring $\OSnPD$ (defined below), prove an isomorphism analogous to \eqref{eq:oalpd_comparison} (see Lemma \ref{lem:osmpd_comp}), and then, extend the latter isomorphism over to $\OALpi^{\PD}$ using the Frobenius.
	So, following \cite[\S 4.4.1]{abhinandan-relative-wach-i}, for $n \in \NN$, let us define a $p\textrm{-adically}$ complete ring $\SnPD := \mbfa_L^+\langle\tfrac{\mu}{p^n}, \tfrac{\mu^2}{2!p^{2n}}, \ldots, \tfrac{\mu^k}{k!p^{kn}}, \ldots\big\rangle$.
	The $p\textrm{-adically}$ completed divided power ring $\SnPD$ is equipped with a continuous action of $\Gamma_L$ and we have a Frobenius homomorphism $\varphi : \SnPD \rightarrow S_{n-1}^{\PD}$, in particular, $\varphi^n(\SnPD) \subset S_0^{\PD} \subset \ALpi^{\PD}$, where the latter inclusion is obvious.
	The reader should note that in \cite[\S 4.4.1]{abhinandan-relative-wach-i} we consider a further completion of $\SnPD$ with respect to certain filtration by PD-ideals, denoted $\SnhatPD$ in loc.\ cit.
	However, such a completion is not strictly necessary and all proofs of loc.\ cit.\ can be carried out without it.
	In particular, many good properties of $\SnhatPD$ restrict to good properties on $\SnPD$ as well (for example, $(\varphi, \Gamma_L)\action$ above).

	Now, consider the $O_F\textrm{-linear}$ homomorphism of rings $\iota : O_L \rightarrow \SnPD$, sending $X_j \mapsto [X_j^{\flat}]$, for $1 \leqslant j \leqslant d$.
	Using $\iota$ define an $O_F\textrm{-linear}$ morphism of rings $f : O_L \otimes_{O_F} \SnPD \rightarrow \SnPD$, via $a \otimes b \mapsto \iota(a) b$.
	Let $\pazo \SnPD$ denote the $\padic$ completion of the divided power envelope of $O_L \otimes_{O_F} \SnPD$ with respect to $\kert f$.
	The divided power ring $\OSnPD$ is equipped with a continuous action of $\Gamma_L$, an integrable connection and we have a Frobenius $\varphi : \OSnPD \rightarrow \pazo S_{n-1}^{\PD}$, in particular, $\varphi^n(\OSnPD) \subset \OALpi^{\PD}$.
	Moreover, we have that $O_L = (\OSnPD)^{\Gamma_L}$, and with $V_j := \frac{X_j \otimes 1}{1 \otimes [X_j^{\flat}]}$, for $1 \leqslant j \leqslant d$, we have $p\textrm{-adically}$ closed divided power ideals
	\begin{equation*}
		J^{[i]}\OSnPD := \Big\langle \tfrac{\mu^{[k_0]}}{p^{nk_0}} \prod_{j=1}^d (1-V_j)^{[k_j]}, \hspace{1mm} \smbfk = (k_0, k_1, \ldots, k_d) \in \NN^{d+1} \hspace{1mm} \textrm{such that} \hspace{1mm} \sum_{j=0}^d k_j \geqslant i\Big\rangle.
	\end{equation*}

	Next, let us equip $\OSnPD \otimes_{\AL^+} N$ with the tensor product Frobenius and an integrable connection induced from the connection on $\OSnPD$.
	Then, $D_n := \big(\OSnPD \otimes_{\AL^+} N[1/p]\big)^{\Gamma_L}$ is an $L\textrm{-vector space}$ equipped with an integrable connection and we have a Frobenius morphism $\varphi : D_n \rightarrow D_{n-1}$.
	In particular, $\varphi^n(D_n) \subset D_L = \big(\OALpi^{\PD} \otimes_{\AL^+} N[1/p]\big)^{\Gamma_L} \subset \big(\OAcrys(\OLbar) \otimes_{\AL^+} N[1/p]\big)^{G_L}$, where the last inclusion follows since $\OALpi^{\PD} \subset \OAcrys(\OLinfty) = \OAcrys(\OLbar)^{H_L}$ (see \cite[Corollary 4.34]{morrow-tsuji}).
	Let $T = \TL(N)$ be the associated finite free $\ZZ_p\representation$ of $G_L$ and $V = T[1/p]$.
	Then, we have that,
	\begin{align}\label{eq:ml_in_dcrys}
		\begin{split}
			D_L \subset \big(\OBcrys^+(\OLbar) \otimes_{\BL^+} N[1/p]\big)^{G_L} \subset &\big(\OBcrys(\OLbar) \otimes_{\BL^+} N[1/p]\big)^{G_L}\\
			&\isomorphic \big(\OBcrys(\OLbar) \otimes_{\QQ_p} V\big)^{G_L} = \ODcrysL(V),
		\end{split}
	\end{align}
	where the isomorphism follows by taking $G_L\textrm{-fixed}$ elements of the extension along $\mbfa^+[1/\mu] \rightarrow \OBcrys(\OLbar)$ of the isomorphism in Lemma \ref{lem:wachmod_comp_imperfect}.
	Recall that $\varphi^n(D_n) \subset D_L$, or equivalently, the $L\linear$ map $1 \otimes \varphi^n : L \otimes_{\varphi^n, L} D_n \rightarrow D_L$ is injective, therefore, we get that $L \otimes_{\varphi^n, L} D_n$ is a finite dimensional $L\textrm{-vector space}$.
	Moreover, $\varphi$ is finite free over $L$, so it follows that $D_n$ is a finite dimensional $L\textrm{-vector space}$ equipped with an integrable connection.
	Furthermore, for $n \geqslant 1$ similar to the proof of \cite[Lemmas 4.32 \& 4.36]{abhinandan-relative-wach-i}, one can show that $\log \gamma_i = \sum_{k \in \NN} (-1)^k \frac{(\gamma_i-1)^{k+1}}{k+1}$ converge as a series of operators on $\OSmPD \otimes_{\AL^+} N$, where $\{\gamma_0, \gamma_1, \ldots, \gamma_d\}$ are topological generators of $\Gamma_L$ (see \S \ref{subsec:period_rings_imperfect}).

\begin{lem}\label{lem:osmpd_comp}
	Let $m \geqslant 1$ (let $m \geqslant 2$ if $p=2$), then we have a $\Gamma_L\equivariant$ isomorphism via the natural map $a \otimes b \otimes x \mapsto ab \otimes x$:
	\begin{equation}\label{eq:osmpd_dm_n}
		\OSmPD \otimes_{O_L} D_m \isomorphic \OSmPD \otimes_{\AL^+} N[1/p].
	\end{equation}
\end{lem}
\begin{proof}
	Compatibility of \eqref{eq:osmpd_dm_n} with the $\Gamma_L\action$ is obvious from the definitions, so we only need to check that it is bijective.
	We will first show that \eqref{eq:osmpd_dm_n} is injective.
	Note that we have an injective ring homomorphism $\OSmPD[1/p] \xrightarrow{\hspace{1mm} \varphi^m \hspace{1mm}} \OALpi^{\PD}[1/p] \rightarrow \OBcrys(\OLbar)$.
	Since $D_m$ is a finite dimensional $L\textrm{-vector space}$, therefore, we get that the following map is injective:
	\begin{equation}\label{eq:osmpd_obcrys_dm}
		\OSmPD \otimes_{O_L} D_m = \OSmPD[1/p] \otimes_L D_m \longrightarrow \OBcrys(\OLbar) \otimes_{\varphi^m, L} D_m.
	\end{equation}
	Now, recall that $V = T[1/p]$ and consider the following composition:
	\begin{equation}\label{eq:dm_obcrys_dcrys}
		\OBcrys(\OLbar) \otimes_{\varphi^m, L} D_m \xrightarrow{\hspace{1mm} 1 \otimes \varphi^m \hspace{1mm}} \OBcrys(\OLbar) \otimes_L D_L \longrightarrow \OBcrys(\OLbar) \otimes_L \ODcrysL(V),
	\end{equation}
	where the first map is injective because $1 \otimes \varphi^m : L \otimes_{\varphi^m, L} D_m \rightarrow D_L$ is injective, and the injectivity of the second map in \eqref{eq:dm_obcrys_dcrys} follows from \eqref{eq:ml_in_dcrys}, in particular, \eqref{eq:dm_obcrys_dcrys} is injective.
	Furthermore, similiar to \eqref{eq:osmpd_obcrys_dm}, note that $N[1/p]$ is a finite free $\BL^+\module$, so it follows that the map $\OSmPD \otimes_{\AL^+} N[1/p] = \OSmPD[1/p] \otimes_{\BL^+} N[1/p] \rightarrow \OBcrys(\OLbar) \otimes_{\varphi^m, \BL^+} N[1/p]$ is injective as well.
	Also, recall that we have an isomorphism $1 \otimes \varphi : \BL^+ \otimes_{\varphi, \BL^+} N[1/p, 1/[p]_q] \isomorphic N[1/p, 1/[p]_q]$.
	So, $\OBcrys(\OLbar) \otimes_{\varphi^m, \BL^+} N[1/p] \isomorphic \OBcrys(\OLbar) \otimes_{\BL^+} N[1/p]$, since $[p]_q$ is invertible in $\OBcrys(\OLbar)$.
	Combining the preceding two observations, we get that the following composition is injective:
	\begin{equation}\label{eq:osmpd_obcrys_n}
		\OSmPD \otimes_{\AL^+} N[1/p] \longrightarrow \OBcrys(\OLbar) \otimes_{\varphi^m, \BL^+} N[1/p] \xrightarrow[\sim]{\hspace{1mm} 1 \otimes \varphi^m \hspace{1mm}} \OBcrys(\OLbar) \otimes_{\BL^+} N[1/p].
	\end{equation}
	Now consider the following diagram
	\begin{center}
		\begin{tikzcd}
			\OSmPD \otimes_{O_L} D_m \arrow[r, "\eqref{eq:osmpd_obcrys_dm}"] \arrow[d, "\eqref{eq:osmpd_dm_n}"'] & \OBcrys(\OLbar) \otimes_{\varphi^m, L} D_m \arrow[r, "\eqref{eq:dm_obcrys_dcrys}"] & \OBcrys(\OLbar) \otimes_R \ODcrysL(V) \arrow[d]\\
			\OSmPD \otimes_{\AL^+} N[1/p] \arrow[r, "\eqref{eq:osmpd_obcrys_n}"] & \OBcrys(\OLbar) \otimes_{\BL^+} N[1/p] \arrow[r] & \OBcrys(\OLbar) \otimes_{\QQ_p} V,
		\end{tikzcd}
	\end{center}
	where the right vertical arrow is the natural injective map (see \cite[Proposition 3.22]{brinon-imparfait}) and the bottom right horizontal map is the extension of the isomorphism in Lemma \ref{lem:wachmod_comp_imperfect} along the natural map $\mbfa^+[1/\mu] \rightarrow \OBcrys(\OLbar)$.
	The diagram commutes by definition and it follows that the left vertical arrow, i.e.\ \eqref{eq:osmpd_dm_n} is injective.

	Now let us check the surjectivity of the map \eqref{eq:osmpd_dm_n}.
	Define the following operators on $\ONmPD := \OSmPD \otimes_{\AL^+} N[1/p]$,
	\begin{equation*}
		\partial_i := \left\{
			\begin{array}{ll}
				-(\log \gamma_0)/t, & \textrm{for} \hspace{1mm} i = 0,\\
				(\log \gamma_i)/(tV_i), & \textrm{for} \hspace{1mm} 1 \leqslant i \leqslant d,
			\end{array}
		\right.
	\end{equation*}
	where $V_i = \frac{X_i \otimes 1}{1 \otimes [X_i^{\flat}]}$, for $1 \leqslant i \leqslant d$ (see \cite[\S 4.4.2]{abhinandan-relative-wach-i}).
	Using the fact that for any $g$ in $\Gamma_L$ and $x$ in $\OSmPD \otimes_{\AL^+} N$, we have $(g-1)(ax) = (g-1)a \cdot x + g(a)(g-1)x$, and from the equality $\log(\gamma_i) = \lim_{n \rightarrow +\infty} (\gamma_i^{p^n}-1)/p^n$, it is easy to see that $\partial_i$ satisfy the Leibniz rule, for all $0 \leqslant i \leqslant d$.
	In particular, the operator $\partial : \ONmPD \rightarrow \ONmPD \otimes_{\OSmPD} \Omega^1_{\OSmPD/O_L}$, given as $x \mapsto \partial_0(x) dt + \sum_{i=1}^d \partial_i(x) d [X_i^{\flat}]$ defines a connection on $\ONmPD$.
	The connection $\partial$ is integrable since the operators $\partial_i$ commute with each other (see \cite[Lemma 4.38]{abhinandan-relative-wach-i}) and using the finite $\pqheight$ property of $N$ it is easy to show that $\partial$ is $p\textrm{-adically}$ quasi-nilpotent as well (see \cite[Lemma 4.39]{abhinandan-relative-wach-i}).

	For $x \in N[1/p]$, similar to the proof of \cite[Lemma 4.41 \& Lemma 4.43]{abhinandan-relative-wach-i}, it follows that the following sum converges in $D_m = (\ONmPD)^{\Gamma_L} = (\ONmPD)^{\partial=0}$:
	\begin{equation}\label{eq:horizontal_elems}
		y = \sum_{\smbfk \in \NN^{d+1}} \partial_0^{k_0} \circ \partial_1^{k_1} \circ \cdots \circ \partial_d^{k_d} (x) \tfrac{t^{[k_0]}}{p^{mk_0}} (1-V_1)^{[k_1]} \cdots (1-V_d)^{[k_d]}.
	\end{equation}
	By choosing a basis of $N$ and using the formula in \eqref{eq:horizontal_elems} on the basis elements, we can define a linear transformation $\alpha$ on the finite free $\OSmPD[1/p]\module$ $\ONmPD$.
	Now, similar to the proof of \cite[Lemma 4.43]{abhinandan-relative-wach-i} it can easily be deduced that for some large enough $N \in \NN$, we can write $p^N \det \alpha \in 1 + J^{[1]}\OSmPD$, i.e.\ $\det \alpha$ is a unit in $\OSmPD[1/p]$ and $\alpha$ defines an automorphism of $\ONmPD$.
	Finally, as the formula in \eqref{eq:horizontal_elems} converges in $D_m$, it follows that the map $\OSmPD \otimes_{O_L} D_m \rightarrow \OSmPD \otimes_{\AL^+} N[1/p]$ is surjective.
	Hence, \eqref{eq:osmpd_dm_n} is bijective.
\end{proof}

	Note that $D_L$ is an $L\textrm{-vector space}$ equipped with the tensor product Frobenius, a filtration given as $\Fil^k D_L = \big(\sum_{i+j=k} \Fil^i \OALpi^{\PD} \otimes_{\AL^+} \Fil^j N[1/p]\big)^{\Gamma_L}$, where $N[1/p]$ is equipped with the Nygaard filtration of Definition \ref{defi:nygaard_fil}.
	The preceding filtration is well-defined, i.e.\ $\Fil^k D_L$ is a sub vector space of $D_L$, for each $k \in \NN$.
	Indeed, it is enough to check that $\Fil^i \OALpi^{\PD}[1/p] \otimes_{\AL^+} \Fil^j N[1/p]$ is an $\OALpi^{\PD}[1/p]\textrm{-submodule}$ of $\OALpi^{\PD}[1/p] \otimes_{\BL^+} N[1/p]$, for each $i, j \in \NN$.
	This easily follows from the fact that the $\OALpi^{\PD}[1/p]\textrm{-linear}$ composition $\Fil^i \OALpi^{\PD}[1/p] \otimes_{\BL^+} \Fil^j N[1/p] \rightarrow \OALpi^{\PD}[1/p] \otimes_{\BL^+} \Fil^j N[1/p] \rightarrow \OALpi^{\PD}[1/p] \otimes_{\BL^+} N[1/p]$ is injective, where the first arrow is obtained by tensoring the $\BL^+\textrm{-linear}$ inclusion $\Fil^i \OALpi^{\PD}[1/p] \rightarrow \OALpi^{\PD}[1/p]$ with the $\BL^+\module$ $\Fil^j N[1/p]$ which is flat (because it is a finite torsion-free module over the principal ideal domain $\BL^+$), and the second arrow is obtained by tensoring the $\BL^+\textrm{-linear}$ inclusion $\Fil^j N[1/p] \rightarrow N[1/p]$ with the flat $\BL^+\algebra$ $\OALpi^{\PD}[1/p]$ (see the discussion at the start of \S \ref{subsec:wachmod_crystalline}).
	Moreover, $D_L$ is equipped with an integrable connection induced from the connection on $\OALpi^{\PD}$ satisfying Griffiths transversality with respect to the filtration since the same is true for the connection on $\OALpi^{\PD}$.
	Now, consider the following diagram:
	\begin{equation}\label{eq:oalpipd_commdiag}
		\begin{tikzcd}
			\OALpi^{\PD} \otimes_{O_L, \varphi^m} D_m \arrow[r, "1 \otimes \varphi^m"] \arrow[d, "\eqref{eq:osmpd_dm_n}"', "\wr"] & \OALpi^{\PD} \otimes_{O_L} D_L \arrow[d, "\eqref{eq:oalpd_comparison}"]\\
			\OALpi^{\PD} \otimes_{\AL^+, \varphi^m} N[1/p] \arrow[r, "\sim"] & \OALpi^{\PD} \otimes_{\AL^+} N[1/p],
		\end{tikzcd}
	\end{equation}
	where the left vertical arrow is the extension of the isomorphism \eqref{eq:osmpd_dm_n} in Lemma \ref{lem:osmpd_comp} along $\varphi^m : \OSmPD \rightarrow \OALpi^{\PD}$ and the bottom horizontal isomorphism follows from an argument similar to \cite[Lemma 4.46]{abhinandan-relative-wach-i}.
	By the description of the arrows it follows that the diagram in \eqref{eq:oalpipd_commdiag} is commutative and $(\varphi, \Gamma_L)\equivariant$.
	Taking $\Gamma_L\textrm{-invariants}$ for the composition of the left vertical and the bottom horizontal isomorphisms gives an $L\textrm{-linear}$ isomorphism $O_L \otimes_{O_L, \varphi^m} D_m \isomorphic D_L$.
	So it follows that the top horizontal arrow in the diagram \eqref{eq:oalpipd_commdiag} is bijective as well.
	The preceding observation together with the bijectivity of the left vertical and the bottom horizontal arrows imply that the right vertical arrow is bijective as well, in particular, the comparison in \eqref{eq:oalpd_comparison} is an isomorphism compatible with the respective Frobenii, connections and $\Gamma_L\textrm{-actions}$.
	Compatibilty of \eqref{eq:oalpd_comparison} with filtrations follows from an argument similar to \cite[Corollary 4.54]{abhinandan-relative-wach-i} (using the filtration compatible isomorphism \eqref{eq:dl_dcrys_comp} in Remark \ref{rem:dl_dcrys_comp}).
	This concludes our proof.
\end{proof}

There exists another relation between the Wach module $N$ and $\ODcrysL(V)$.
Let us equip $N$ with a Nygaard filtration as in Definition \ref{defi:nygaard_fil}.
Then we note that $(N/\mu N)[1/p]$ is a $\varphi\module$ over $L$, since $[p]_q = p \textmod \mu N$ and $N/\mu N$ is equipped with a filtration $\Fil^k(N/\mu N)$ given as the image of $\Fil^k N$ under the surjection $N \twoheadrightarrow N/\mu N$.
We equip $(N/\mu N)[1/p]$ with the induced filtration, in particular, it is a filtered $\varphi\module$ over $L$.

\begin{cor}\label{cor:qdeformation_dcrys}
	Let $N$ be a Wach module over $\AL^+$ and $V = \TL(N)[1/p]$ the associated crystalline representation from Theorem \ref{thm:fh_crys_imperfect}.
	Then, we have that $(N/\mu N)[1/p] \isomorphic \ODcrysL(V)$ as filtered $\varphi\modules$ over $L$.
\end{cor}
\begin{proof}
	For $r \in \NN$ large enough, the Wach module $\mu^r N (-r)$ is always effective and we have that $\TL(\mu^rN(-r)) = \TL(N)(-r)$ (the twist $(-r)$ denotes a Tate twist on which $\Gamma_L$ acts via $\chi^{-r}$, where $\chi$ is the $\padic$ cyclotomic character).
	Therefore, it is enough to show the claim for effective Wach modules.
	So assume that $N$ is effective and set $M := N[1/p]$ equipped with the induced Frobenius, $\Gamma_L\action$ and Nygaard filtration.
	Note that the $L\textrm{-vector space}$ $M/\mu M$ is equipped with a Frobenius-semilinear operator $\varphi$ induced from $M$ such that $1 \otimes \varphi : \varphi^*(M/\mu M) \isomorphic M/\mu M$, since $[p]_q = p \textrm{ mod } \mu$.
	The filtration $\Fil^k (M/\mu M)$ on $M/\mu M$ is the image of $\Fil^k M$ under the surjective map $M \twoheadrightarrow M/\mu M$.
	From the discussion before Theorem \ref{thm:fh_crys_imperfect}, recall that we have a period ring $\OARpi^{\PD} \subset \OAcrys(\OLinfty)$ equipped with a natural Frobenius, filtration, connection and $\Gamma_L\action$.
	Moreover, from Theorem \ref{thm:fh_crys_imperfect}, we have that $D_L = (\OALpi^{\PD} \otimes_{\AL^+} \NL(V))^{\Gamma_L}$ is equipped with a natural Frobenius, filtration and connection, such that $D_L \isomorphic \ODcrys(V)$ compatible with the respective Frobenii, filtrations and connections (see \eqref{eq:dl_dcrys_comp} in Remark \ref{rem:dl_dcrys_comp}).
	Now, consider the following diagram with exact rows:
	\begin{center}
		\begin{tikzcd}[column sep=7mm]
			0 \arrow[r] & \mu M \arrow[r] \arrow[d] & M \arrow[r] \arrow[d] & M/\mu M \arrow[r] \arrow[d] & 0\\
			0 \arrow[r] & (\Fil^1 \OALpi^{\PD}) \otimes_{\AL^+} M \arrow[r] & \OALpi^{\PD} \otimes_{\AL^+} M \arrow[r] & L(\zeta_p) \otimes_L M/\mu M \arrow[r] & 0\\
			0 \arrow[r] & (\Fil^1 \OALpi^{\PD}) \otimes_{O_L} D_L \arrow[r] \arrow[u, "\wr"] & \OALpi^{\PD} \otimes_{O_L} D_L \arrow[r] \arrow[u, "\wr", "\eqref{eq:oalpd_comparison}"'] & L(\zeta_p) \otimes_L D_L \arrow[r] \arrow[u, "\wr"] & 0.
		\end{tikzcd}
	\end{center}
	Note that $(\Fil^1 \OALpi^{\PD} \otimes_{\AL^+} M) \cap M = (\Fil^1 \OALpi^{\PD} \cap \AL^+) \otimes_{\AL^+} M = \mu M$, so the vertical maps from the first to the second row are natural inclusions and the second row is exact.
	For the vertical arrows from the third to the second row, we note that the middle vertical arrow is the isomorphism \eqref{eq:oalpd_comparison} in Proposition \ref{prop:oalpd_comparison}, from which it easily follows that the left vertical arrow is also an isomorphism, and therefore, we get that the right vertical arrow is an isomorphism as well.
	Taking the $\Gal(L(\zeta_p)/L)\textrm{-invariants}$ of the right vertical arrow gives $M/\mu M \lisomorphic D_L \isomorphic \ODcrysL(V)$, where the last isomorphism is compatible with the respective Frobenii, filtrations and connections (see \eqref{eq:dl_dcrys_comp} in Remark \ref{rem:dl_dcrys_comp}).

	Note that the isomorphism $D_L \isomorphic M/\mu M$ is compatible with the respective Frobenii and we need to check the compatibility between the respective filtrations.
	In the diagram above, the middle term of the second row is equipped with the tensor product filtration, so the image of $\Fil^k(\OALpi^{\PD} \otimes_{\AL^+} M)$ under the surjective map from the second to the third term is given as $L(\zeta_p) \otimes_L \Fil^k(M/\mu M)$.
	Similarly, the middle term of the third row is equipped with the tensor product filtration, so the image of $\Fil^k(\OALpi^{\PD} \otimes_{O_L} D_L)$ under the surjective map from the second to the third term is given as $L(\zeta_p) \otimes_L \Fil^k D_L$.
	Since the isomorphism \eqref{eq:oalpd_comparison} in Proposition \ref{prop:oalpd_comparison} is compatible with filtrations, we get that $L(\zeta_p) \otimes_L \Fil^k D_L \isomorphic L(\zeta_p) \otimes_L \Fil^k(M/\mu M)$.
	Taking $\Gal(L(\zeta_p)/L)\textrm{-invariants}$ in the preceding isomorphism gives $\Fil^k D_L \isomorphic \Fil^k(M/\mu M)$.
	This concludes our proof.
\end{proof}

\section{Crystalline implies finite height}\label{sec:crystalline_finite_height}

The goal of this section is to prove the following claim:
\begin{thm}\label{thm:crys_fh_imperfect}
	Let $T$ be a finite free $\ZZ_p\representation$ of $G_L$ such that $V := T[1/p]$ is a $\padic$ crystalline representation of $G_L$.
	Then, there exists a unique Wach module $\NL(T)$ over $\AL^+$ satisfying Definition \ref{defi:finite_pqheight_imperfect}.
	In other words, $T$ is of finite $\pqheight$.
\end{thm}

Before carrying out the proof of Theorem \ref{thm:crys_fh_imperfect}, we note the following corollaries:
let $\Rep_{\ZZ_p}^{\crys}(G_L)$ denote the category of $\ZZ_p\textrm{-lattices}$ inside $\padic$ crystalline representations of $G_L$.
Then, by combining Theorem \ref{thm:fh_crys_imperfect} and Theorem \ref{thm:crys_fh_imperfect} and \cite[Proposition 4.14]{abhinandan-relative-wach-i} (for compatibility with tensor products below), we obtain the following:
\begin{cor}\label{cor:crystalline_wach_equivalence_imperfect}
	The Wach module functor induces an equivalence of $\otimes\textrm{-categories}$,
	\begin{align*}
		\Rep_{\ZZ_p}^{\crys}(G_L) &\isomorphic (\varphi, \Gamma)\Mod_{\AL^+}^{[p]_q}\\
		T &\longmapsto \NL(T),
	\end{align*}
	with a quasi-inverse $\otimes\textrm{-functor}$ given as $N \mapsto \TL(N) := \big(W(\CC_L^{\flat}) \otimes_{\AL^+} N\big)^{\varphi=1}$.
\end{cor}

Passing to associated isogeny categories, we obtain the following:
\begin{cor}\label{cor:crystalline_wach_rat_equivalence_imperfect}
	The Wach module functor induces an exact equivalence of $\otimes\textrm{-categories}$ $\Rep_{\QQ_p}^{\crys}(G_L) \isomorphic (\varphi, \Gamma)\Mod_{\BL^+}^{[p]_q}$, via $V \mapsto \NL(V)$, with an exact quasi-inverse $\otimes\textrm{-functor}$ given as $M \mapsto \VL(M) := \big(W(\CC_L^{\flat}) \otimes_{\AL^+} M\big)^{\varphi=1}$.
\end{cor}

In the rest of this section, we will carry out the proof of Theorem \ref{thm:crys_fh_imperfect} and Corollary \ref{cor:crystalline_wach_equivalence_imperfect} by constructing $\NL(T)$ and show Corollary \ref{cor:crystalline_wach_rat_equivalence_imperfect} as a consquence.
In \S \ref{subsec:classical_wachmod}, we collect important properties of classical Wach modules, i.e.\ the perfect residue field case.
In \S \ref{subsec:kisin_module}, we use ideas from \cite{kisin-modules, kisin-ren} to show that classical Wach modules are compatible with Kisin-Ren modules, and we further show that in our setting, a finite $\pqheight$ module on the open unit disk over $\Lbreve$ descends to a finite $\pqheight$ module on the open unit disk over $L$, similar to \cite{brinon-trihan}.
On the module thus obtained, we use results of \S \ref{subsec:crysrep_imperfect} to construct an action of $\Gamma_L$ and study its properties in \S \ref{subsec:nriglv_galois_action}.
Then, in \S \ref{subsec:compatibility_phigammmod}, we check that our construction is compatible with the theory of \'etale $(\varphi, \Gamma_L)\modules$.
Finally, in \S \ref{subsec:obtaining_wachmod}, we construct the promised Wach module $\NL(T)$ and prove Theorem \ref{thm:crys_fh_imperfect} and Corollary \ref{cor:crystalline_wach_rat_equivalence_imperfect}.

For a $\padic$ representation of $G_L$, note that the property of being crystalline and of finite $\pqheight$ is invariant under twisting the representation by $\chi^r$, for $r \in \NN$.
So, from now onwards we will assume that $V$ is a $\padic$ positive crystalline representation of $G_L$, i.e. all its Hodge-Tate weights are $\leqslant 0$ and we have $T \subset V$ a $G_L\textrm{-stable}$ $\ZZ_p\textrm{-lattice}$.

\subsection{Classical Wach modules}\label{subsec:classical_wachmod}

Recall that $G_{\Lbreve}$ is a subgroup of $G_L$, so from \cite[Proposition 4.14]{brinon-trihan}, it follows that $V$ is a $\padic$ positive crystalline representation of $G_{\Lbreve}$ and $T \subset V$ a $G_{\Lbreve}\textrm{-stable}$ $\ZZ_p\textrm{-lattice}$.
Note that $\Lbreve$ is an unramified extension of $\QQ_p$ with perfect residue field, therefore, the $G_{\Lbreve}\representation$ $V$ is of finite $\pqheight$ (see \cite{colmez-hauteur} and \cite{berger-limites}).
Let the $\pqheight$ of $V$ be $s \in \NN$.
One associates to $V$ a finite free $(\varphi, \GammaLbreve)\module$ over $\BLbreve^+$ of rank $=\dim_{\QQ_p} V$, called the Wach module $\NLbreve(V)$, and to $T$ a finite free $(\varphi, \GammaLbreve)\module$ over $\ALbreve^+$ of rank $=\dim_{\QQ_p} V$, called the Wach module $\NLbreve(T)$ (see \cite{wach-free, wach-torsion, berger-limites} and \cite[\S 4.1]{abhinandan-relative-wach-i} for a recollection).
Let $\Dtilde_L^+(T) := (\Ainf(\OLbar) \otimes_{\ZZ_p} T)^{H_L}$ be the $(\varphi, \Gamma_L)\module$ over $\Ainf(\OLinfty) = \Ainf(\OLbar)^{H_L}$ (see \cite[Proposition 7.2]{andreatta-phigamma}) associated to $T$ and let $\Dtilde_L^+(V) := \Dtilde_L^+(T)[1/p]$ over $\Binf(\OLinfty) = \Binf(\OLbar)^{H_L}$ associated to $V$.

\begin{lem}[{\cite[Lemme II.1.3, Th\'eor\`eme III.3.1]{berger-limites}}]\phantomsection\label{lem:wach_props}
	With notations as above, we have the following:
	\begin{enumarabicup}
	\item $\NLbreve(T) = \NLbreve(V) \cap \DLbreve(T) \subset \DLbreve(V)$.

	\item We have that $\mu^s \Ainf(\OLbar) \otimes_{\ZZ_p} T \subset \Ainf(\OLbar) \otimes_{\ALbreve^+} \NLbreve(T) \subset \Ainf(\OLbar) \otimes_{\ZZ_p} T$, and taking $H_L\textrm{-invariants}$ gives $\mu^s \Dtilde_L^+(T) \subset \Ainf(\OLinfty) \otimes_{\ALbreve^+} \NLbreve(T) \subset \Dtilde_L^+(T)$.
		Similar claims are also true for $V$.
	\end{enumarabicup}
\end{lem}

By properties of Wach modules, we have the following functorial isomorphisms of \'etale $(\varphi, \Gamma_L)\modules$, where the second isomorphism in the first row follows from \cite[Th\'eor\`eme III.3.1]{berger-limites}:
\begin{align}\label{eq:nlbreve_phigamma_comp}
	\begin{split}
		\ALbreve \otimes_{\ALbreve^+} \NLbreve(T) &\isomorphic \DLbreve(T) \hspace{2mm} \textrm{and} \hspace{2mm} \ALbrevedag \otimes_{\ALbreve^+} \NLbreve(T) \isomorphic \DLbrevedag(T),\\
		\BLbreve \otimes_{\BLbreve^+} \NLbreve(V) &\isomorphic \DLbreve(V) \hspace{2mm} \textrm{and} \hspace{2mm} \BLbrevedag \otimes_{\BLbreve^+} \NLbreve(V) \isomorphic \DLbrevedag(V),\\
		&\BrigLbrevedag \otimes_{\BLbreve^+} \NLbreve(V) \isomorphic \DrigLbrevedag(V).
	\end{split}
\end{align}

Let us set $\NrigLbreve(V) := \BrigLbreve^+ \otimes_{\BLbreve^+} \NLbreve(V)$ equipped with the induced tensor-product Frobenius-semilinear operator $\varphi$ and $\GammaLbreve\action$.
From \cite[Proposition II.2.1]{berger-limites}, recall that we have a natural inclusion $\DcrysLbreve(V) \subset \NrigLbreve(V)$, which extends $\BrigLbreve^+\textrm{-linearly}$ to a Frobenius and $\GammaLbreve\equivariant$ inclusion $\BrigLbreve^+ \otimes_{\Lbreve} \DcrysLbreve(V) \subset \NrigLbreve(V)$ such that its cokernel is killed by $(t/\mu)^s \in \BrigLbreve^+$ (see \cite[Propositions II.3.1 \& III.2.1]{berger-limites}).
In particular, we obtain a $(\varphi, \GammaLbreve)\equivariant$ isomorphism,
\begin{equation}\label{eq:wach_crys_rigcomp}
	\BrigLbreve^+[\mu/t] \otimes_{\Lbreve} \DcrysLbreve(V) \isomorphic \BrigLbreve^+[\mu/t] \otimes_{\BLbreve^+} \NLbreve(V).
\end{equation}
Moreover, note that from loc.\ cit., we have a natural $\Lbreve\linear$ isomorphism of filtered $\varphi\modules$ $\DcrysLbreve(V) \isomorphic \NrigLbreve(V)/\mu\NrigLbreve(V) = \NLbreve(V)/\mu\NLbreve(V)$ such that the largest Hodge-Tate weight of $V$ equals $s$, i.e.\ the $\pqheight$ of $V$.
Since $t/\mu$ is a unit in $\Bcrys^+(\OLinfty)$ and $\BrigLbreve^+ \subset \BrigLtilde^+ \subset \Bcrys^+(\OLinfty)$, therefore, extension of scalars of \eqref{eq:wach_crys_rigcomp} gives a $\Bcrys^+(\OLinfty)\textrm{-linear}$ and $(\varphi, \GammaLbreve)\equivariant$ isomorphism,
\begin{equation}\label{eq:wac_crys_comp}
	\Bcrys^+(\OLinfty) \otimes_{\Lbreve} \DcrysLbreve(V) \isomorphic \Bcrys^+(\OLinfty) \otimes_{\BLbreve^+} \NLbreve(V).
\end{equation}

\begin{lem}\label{lem:wac_crys_rep_comp}
	The following diagram is commutative and $(\varphi, G_{\Lbreve})\equivariant$:
	\begin{center}
		\begin{tikzcd}
			\Bcrys(\OLbar) \otimes_{\Lbreve} \DcrysLbreve(V) \arrow[r, "\sim"] \arrow[d, "\wr"] & \Bcrys(\OLbar) \otimes_{\BLbreve^+} \NLbreve(V) \arrow[d, "\wr"]\\
			\Bcrys(\OLbar) \otimes_{\QQ_p} V \arrow[r, equal] & \Bcrys(\OLbar) \otimes_{\QQ_p} V,
		\end{tikzcd}
	\end{center}
	where the top vertical arrow is the extension of scalars of \eqref{eq:wac_crys_comp} along $\Bcrys^+(\OLinfty) \rightarrow \Bcrys(\OLbar)$, and the left vertical arrow is the natural isomorphism as $V$ is a crystalline representation of $G_{\Lbreve}$.
\end{lem}
\begin{proof}
	From Lemma \ref{lem:wach_props} (2), note that we have a $(\varphi, G_{\Lbreve})\equivariant$ isomorphism $\Ainf(\OLbar)[1/\mu] \otimes_{\ALbreve^+} \NLbreve(T) \isomorphic \Ainf(\OLbar)[1/\mu] \otimes_{\ZZ_p} T$, and extending this isomorphism along $\Ainf(\OLbar)[1/\mu] \rightarrow \Bcrys(\OLbar)$ gives the isomorphism in the right vertical arrow.
	The commutativity of the diagram follows because the top horizontal arrow is also the $\Bcrys(\OLbar)\linear$ extension of the natural inclusion $\DcrysLbreve(V) \subset \BrigLbreve^+ \otimes_{\BLbreve^+} \NLbreve(V) \subset \Bcrys(\OLbar) \otimes_{\BLbreve^+} \NLbreve(V)$ (see \cite[\S II.2]{berger-limites}).
\end{proof}

\subsection{Kisin's construction}\label{subsec:kisin_module}

Our goal is to construct a Wach module $\NL(T)$ over $\AL^+$.
To this end, we will adapt some ideas from \cite{brinon-trihan} and \cite{kisin-ren}, generalising the results of Kisin in \cite{kisin-modules}, to first construct a finite $\pqheight$ module over $\BrigL^+$.

Let $E(X) := \frac{(1+X)^p-1}{X}$ in $\ZZ_p\llbracket X \rrbracket$ denote the cyclotomic polynomial.
We equip $\ZZ_p\llbracket X \rrbracket$ with the cyclotomic Frobenius operator $\varphi$ given by identity on $\ZZ_p$ and setting $\varphi(X) = (1+X)^p-1$, and for $n \in \NN$ we set $E_n(X) := \varphi^n(E(X))$.
In particular, $\zeta_{p^{n+1}}-1$ is a simple zero of $E_n(X)$, where $\zeta_{p^{n+1}}$ is a primitive $p^{n+1}\textrm{-th}$ root of unity.
For $X = \mu$, we will write $E_n(X) = \xitilde_n$, for $n \in \NN$, and $E(\mu) = \varphi(\mu)/\mu = \xitilde = \xitilde_0 = [p]_q$.

\begin{rem}\label{rem:frob_coeff}
	Define $\phi_{L} : \BrigL^+ \rightarrow \BrigL^+$ to be the homomorphism given by the Frobenius $\varphi_L$ on $L$ and set $\phi_L(\mu) = \mu$, i.e.\ $\sum_{k \in \NN} \iota(a_k)\mu^k \mapsto \sum_{k \in \NN} \iota(\varphi_L(a_k))\mu^k$, where we used Lemma \ref{lem:brigl+_explicit} to represent an element of $\BrigL^+$.
	Then, $\BrigL^+$ is finite free of rank $p^d$ over $\BrigL^+$, via the map $\phi_L$, in particular, flat.
	Similarly, let $\phi_{\Lbreve} : \BrigLbreve^+ \rightarrow \BrigLbreve^+$ denote the homomorphism given by the Frobenius $\varphi_{\Lbreve}$ on $\Lbreve$ and set $\phi_{\Lbreve}(\mu) = \mu$.
	Moreover, note that we have $\varphi_{\Lbreve} : \Lbreve \isomorphic \Lbreve$, since the residue field of $\Lbreve$ is perfect, and therefore, we see that $\phi_{\Lbreve}$ is bijective on $\BrigLbreve^+$, with its inverse given as $\phi_{\Lbreve}^{-1} : \BrigLbreve^+ \rightarrow \BrigLbreve^+$, sending $\sum_{k \in \NN} \iota(a_k)\mu^k \mapsto \sum_{k \in \NN} \iota(\varphi_{\Lbreve}^{-1}(a_k))\mu^k$.
	Furthermore, from \S \ref{subsubsec:analytic_rings}, recall that we have an injective homomorphism $\BrigL^+ \rightarrow \BrigLbreve^+$, which is evidently compatible with $\phi_L$ on the left and $\phi_{\Lbreve}$ on the right.
\end{rem}

\begin{rem}\label{rem:zeros_tovermu}
	We have that $t/\mu$ is in $\BrigL^+ \hookrightarrow \BrigLbreve^+$ and we can write $t/\mu = \prod_{n \in \NN} (\xitilde_n/p)$ (see \cite[Exemple I.3.3]{berger-limites} and \cite[Remarque 4.12]{lazard}).
	The zeros of $t/\mu$ are given as $\zeta_{p^{n+1}} - 1$, for all $n \in \NN$.
	Moreover, we have that $\phi_{\Lbreve}^{-n}(t/\mu) = t/\mu$, therefore, the zeros of $\phi_{\Lbreve}^{-n}(t/\mu)$ are given by $\zeta_{p^{n+1}}-1$ as well.
\end{rem}

Now, let $\BLbrevenhat$ denote the completion of $\Lbreve(\zeta_{p^{n+1}}) \otimes_{\Lbreve} \BLbreve^+$ with respect to the maximal ideal generated by $\mu-(\zeta_{p^{n+1}}-1)$.
Since $\zeta_{p^{n+1}}-1$ is a simple root of $\xitilde_n$, therefore, we see that $(\mu-(\zeta_{p^{n+1}}-1)) = (\xitilde_n) \subset \BLbrevenhat$.
The local ring $\BLbrevenhat$ naturally admits an action of $\GammaLbreve$ induced by the diagonal action of $\GammaLbreve$ on the tensor product $\Lbreve(\zeta_{p^{n+1}}) \otimes_{\Lbreve} \BLbreve^+$.
We put a filtration on $\BLbrevenhat[1/\xitilde_n]$ by setting $\Fil^r \BLbrevenhat\big[1/\xitilde_n\big] := \xitilde_n^r\BLbrevenhat$, for $r \in \ZZ$.
We have inclusions $\BLbreve^+ \subset \BrigLbreve^+ \subset \BLbrevenhat[1/\xitilde_n\big]$.

Let $D_L := \ODcrysL(V)$ and $D_{\Lbreve} := \DcrysLbreve(V)$, and recall that using the $\varphi\equivariant$ injection $L \rightarrow \Lbreve$, we have an isomorphism of filtered $\varphi\modules$ $\Lbreve \otimes_L D_L \isomorphic D_{\Lbreve}$ from \eqref{eq:dcrys_llbreve_comp}.
Note that $D_L$ (resp.\ $D_{\Lbreve}$) is an effective filtered $\varphi\module$ over $L$ (resp.\ over $\Lbreve$), i.e.\ $\Fil^0 D_L = D_L$ (resp. $\Fil^0 D_{\Lbreve} = D_{\Lbreve}$), and we have a $\varphi\equivariant$ inclusion $D_L \subset D_{\Lbreve}$.
Now, consider a map,
\begin{equation}\label{eq:in_briglbreve}
	i_n : \BrigLbreve^+ \otimes_{\Lbreve} D_{\Lbreve} \xrightarrow[\sim]{\phi_{\Lbreve}^{-n} \otimes \varphi_{D_{\Lbreve}}^{-n}} \BrigLbreve^+ \otimes_{\Lbreve} D_{\Lbreve} \longrightarrow \BLbrevenhat \otimes_{\Lbreve} D_{\Lbreve},
\end{equation}
where $\phi_{\Lbreve}^{-1} : \BrigLbreve^+ \rightarrow \BrigLbreve^+$ is well defined by Remark \ref{rem:frob_coeff}, and $\varphi_{D_{\Lbreve}}$ is the (bijective) Frobenius-semilinear operator on $D_{\Lbreve}$.
The map $i_n$ is evidently well defined, and it extends to a map,
\begin{equation*}
	i_n : \BrigLbreve^+[\mu/t] \otimes_{\Lbreve} D_{\Lbreve} \longrightarrow \BLbrevenhat[\mu/t] \otimes_{\Lbreve} D_{\Lbreve}.
\end{equation*}
Define a $\BrigLbreve^+\module$ as follows:
\begin{equation*}
	\MLbreve(D_{\Lbreve}) := \big\{x \in \BrigLbreve^+[\mu/t] \otimes_{\Lbreve} D_{\Lbreve}, \textrm{ such that } \forall n \in \NN, i_n(x) \in \Fil^0(\BLbrevenhat[1/\xitilde_n] \otimes_{\Lbreve} D_{\Lbreve})\big\},
\end{equation*}
where $\BrigLbreve^+[\mu/t] \otimes_{\Lbreve} D_{\Lbreve}$ is equipped with the tensor product Frobenius and $\BLbrevenhat\big[1/\xitilde_n\big] \otimes_{\Lbreve} D_{\Lbreve}$ is equipped with the tensor product filtration.
By \cite[Lemma 1.2.2]{kisin-modules} and \cite[Lemma 2.2.1]{kisin-ren}, the $\BrigLbreve^+\module$ $\MLbreve(D_{\Lbreve})$ is finite free of rank $=\dim_{\Lbreve} D_{\Lbreve}$, stable under $\varphi$ and $\GammaLbreve$, and such that the cokernel of the injective map $1 \otimes \varphi : \varphi^*(\MLbreve(D_{\Lbreve})) \rightarrow \MLbreve(D_{\Lbreve})$ is killed by $\xitilde^s$ (where $s = \textrm{height of } T = \textrm{height of } V$), and the action of $\GammaLbreve$ is trivial modulo $\mu$.
Moreover, from \cite[Lemma 2.2.2]{kisin-ren}, there exists a unique $\Lbreve\textrm{-linear}$ section $\alpha : \MLbreve(D_{\Lbreve})/\mu \MLbreve(D_{\Lbreve}) \rightarrow \MLbreve(D_{\Lbreve})[\mu/t]$ such that the image $\alpha(\MLbreve(D_{\Lbreve})/\mu \MLbreve(D_{\Lbreve}))$ is $\GammaLbreve\textrm{-invariant}$.
Furthermore, the section $\alpha$ is $\varphi\equivariant$ and it induces an isomorphism,
\begin{equation}\label{eq:kisinren_crys_rigcomp}
	1 \otimes \alpha : \BrigLbreve^+[\mu/t] \otimes_{\Lbreve} (\MLbreve(D_{\Lbreve})/\mu \MLbreve(D_{\Lbreve})) \isomorphic \MLbreve(D_{\Lbreve})[\mu/t].
\end{equation}
Finally, from \cite[Proposition 2.2.6]{kisin-ren} we have a natural isomorphism $D_{\Lbreve} \isomorphic \MLbreve(D_{\Lbreve})/\mu \MLbreve(D_{\Lbreve})$ compatible with the respective Frobenii and filtrations, and under the isomorphism \eqref{eq:kisinren_crys_rigcomp}, the image of $D_{\Lbreve}$ coincides with $\alpha(\MLbreve(D_{\Lbreve})/\mu \MLbreve(D_{\Lbreve}))$.

Next, we note that the $\BrigLbrevedag\module$ $\BrigLbrevedag \otimes_{\BrigLbreve^+} \MLbreve(D_{\Lbreve})$ is pure of slope 0 using \cite[Theorem 1.3.8]{kisin-modules} and \cite[Proposition 2.3.3]{kisin-ren}.
Then, from \cite[Corollay 2.4.2]{kisin-ren} one obtains an $\ALbreve^+\module$ $N_{\Lbreve}$ finite free of rank $=\dim_{\Lbreve} D_{\Lbreve}$, equipped with a Frobenius-semilinear endomorphism $\varphi$ and semilinear and continuous action of $\GammaLbreve$, and such that cokernel of the injective map $1 \otimes \varphi : \varphi^*(N_{\Lbreve}) \rightarrow N_{\Lbreve}$ is killed by $\xitilde^s$, the action of $\GammaLbreve$ is trivial modulo $\mu$ and $\BrigLbreve^+ \otimes_{\ALbreve^+} N_{\Lbreve} \isomorphic \MLbreve(D_{\Lbreve})$ compatible with $(\varphi, \GammaLbreve)\action$.

\begin{lem}\label{lem:nlbreve_wach}
	There is a natural $\BrigLbreve^+\linear$ and $(\varphi, \GammaLbreve)\equivariant$ isomorphism $\beta: \MLbreve(D_{\Lbreve}) \isomorphic \NrigLbreve(V)$.
	Moreover, it restricts to a $\BLbreve^+\linear$ and $(\varphi, \GammaLbreve)\equivariant$ isomorphism $\beta: N_{\Lbreve}[1/p] \isomorphic \NLbreve(V)$.
\end{lem}
\begin{proof}
	Recall that by definition $\NrigLbreve(V) = \BrigLbreve^+ \otimes_{\BL^+} \NLbreve(V)$, and consider the following diagram:
	\begin{equation}\label{eq:alpha_beta_commute}
		\begin{tikzcd}[column sep=large]
			\BrigLbreve^+[\mu/t] \otimes_{\Lbreve} D_{\Lbreve} \arrow[r, "\sim"] \arrow[d, "\wr"'] & \NrigLbreve(V)[\mu/t] \\
			\BrigLbreve^+[\mu/t] \otimes_{\Lbreve} (\MLbreve(D_{\Lbreve})/\mu\MLbreve(D_{\Lbreve})) \arrow[r, "\sim", "1 \otimes \alpha"'] & \MLbreve(D_{\Lbreve})[\mu/t] \arrow[u, "\wr", "\beta"'],
		\end{tikzcd}
	\end{equation}
	where the top horizontal arrow is \eqref{eq:wach_crys_rigcomp}, the bottom horizontal arrow is \eqref{eq:kisinren_crys_rigcomp} and the left vertical arrow is the extension of scalars of the isomorphism $D_{\Lbreve} \isomorphic \MLbreve(D_{\Lbreve})/\mu\MLbreve(D_{\Lbreve})$ along the natural $(\varphi, \Gamma_{\Lbreve})\equivariant$ map $\Lbreve \rightarrow \BrigLbreve^+$.
	For the right vertical arrow $\beta$, we consider $\NrigLbreve(V)$ and $\MLbreve(D_{\Lbreve})$ as submodules of $\BrigLbreve^+[\mu/t] \otimes_{\Lbreve} \DcrysLbreve(V)$ and construct the map as follows:
	note that from the discussion after \eqref{eq:wach_crys_rigcomp}, we have a natural isomorphism $D_{\Lbreve} \isomorphic \NrigLbreve(V)/\mu\NrigLbreve(V)$ of filtered $\varphi\modules$ over $\Lbreve$.
	Moreover, from \cite[Lemma 2.1.2]{kisin-ren}, note that the action of $\GammaLbreve$ on $\NrigLbreve(V)$ is ``$\mathcal{O}\textrm{-analytic}$'' in the sense of \cite[\S 2.1.3]{kisin-ren}, where $\mathcal{O} = \ZZ_p$ in our case (this is true because in the language of op.\ cit., we see that the Lubin-Tate group law over $\OLbreve$ that we consider here is given by the Frobenius power series $(1+X)^p-1$ for the uniformiser $p$).
	Therefore, from the equivalence of categories in \cite[Proposition 2.2.6]{kisin-ren} and its proof, it follows that we have a natural isomorphism 
	\begin{equation*}
		\beta: \MLbreve(D_{\Lbreve}) \isomorphic \MLbreve(\NrigLbreve(V)/\mu\NrigLbreve(V)) \isomorphic \NrigLbreve(V),
	\end{equation*}
	as $\BrigLbreve^+\textrm{-submodules}$ of $\BrigLbreve^+[\mu/t] \otimes_{\Lbreve} D_{\Lbreve}$, compatible with $(\varphi, \GammaLbreve)\action$, and such that the reduction modulo $\mu$ of $\beta$ induces natural isomorphisms,
	\begin{equation*}
		\beta \textrm{ mod } \mu : \MLbreve(D_{\Lbreve})/\mu\MLbreve(D_{\Lbreve}) \isomorphic D_{\Lbreve} \isomorphic \NrigLbreve(V)/\mu\NrigLbreve(V),
	\end{equation*}
	of filtered $\varphi\modules$ over $\Lbreve$, where the latter isomorphism coincides with the one mentioned above (coming from the discussion after \eqref{eq:wach_crys_rigcomp}).
	Now, by composing the natural $\Lbreve\linear$ inclusion 
	\begin{equation*}
		(\MLbreve(D_{\Lbreve})/\mu\MLbreve(D_{\Lbreve})) \subset \BrigLbreve^+[\mu/t] \otimes_{\Lbreve} (\MLbreve(D_{\Lbreve})/\mu\MLbreve(D_{\Lbreve})),
	\end{equation*}
	with the inverse of the left vertical arrow, the top horizontal arrow and the inverse of the right vertical arrow of the diagram \eqref{eq:alpha_beta_commute}, provides an $\Lbreve\linear$ section $\MLbreve(D_{\Lbreve})/\mu\MLbreve(D_{\Lbreve}) \rightarrow \MLbreve(D_{\Lbreve})[\mu/t]$ satisfying the same properties as $\alpha$ (see the discussion before \eqref{eq:kisinren_crys_rigcomp}).
	Therefore, from the uniqueness of $\alpha$, it follows that the diagram commutes, thus proving our first claim.
	For the second claim, note that $\BrigLbrevedag \otimes_{\BrigLbreve^+} \MLbreve(D_{\Lbreve}) \isomorphic \BrigLbrevedag \otimes_{\BrigLbreve^+} \NrigLbreve(V)$ is pure of slope 0, so from \cite[Corollary 2.4.2]{kisin-ren} we conclude that the isomorphism $\beta$ induces an isomorphism $\beta: N_{\Lbreve}[1/p] \isomorphic \NLbreve(V)$ compatible with $(\varphi, \GammaLbreve)\action$.
	This allows us to conclude.
\end{proof}

Recall that from \eqref{eq:dcrys_llbreve_comp}, we have an isomorphism of filtered $\varphi\modules$ $\Lbreve \otimes_L D_L \isomorphic D_{\Lbreve}$.
\begin{defi}\label{defi:mldl}
	Define the following $\BrigL^+\module$:
	\begin{align*}
		\ML(D_L) &:= \big\{x \in \BrigL^+[\mu/t] \otimes_{L} D_L, \textrm{ such that } \forall n \in \NN, i_n(x) \in \Fil^0\big(\BLbrevenhat[1/\xitilde_n] \otimes_{\Lbreve} D_{\Lbreve}\big)\big\}\\
		&= \big(\BrigL^+[\mu/t] \otimes_{L} D_L\big) \cap \MLbreve(D_{\Lbreve}) \subset \BrigLbreve^+[\mu/t] \otimes_{\Lbreve} D_{\Lbreve}.
	\end{align*}
\end{defi}
From \S \ref{subsubsec:periodrings_Lbreve}, recall that we have a $\varphi\equivariant$ injective homomorphism $\BrigL^+ \rightarrow \BrigLbreve^+$, therefore, by definition $\ML(D_L)$ is stable under the induced tensor product Frobenius semilinear-operator $\varphi$ on $\BrigLbreve^+[\mu/t] \otimes_{\Lbreve} D_{\Lbreve}$.
Then, by using Lemma \ref{lem:nlbreve_wach} and the discussion preceding \eqref{eq:wach_crys_rigcomp}, we have $\varphi\equivariant$ inclusions 
\begin{equation*}
	\BrigLbreve^+ \otimes_{\Lbreve} D_{\Lbreve} \subset \MLbreve(D_{\Lbreve}) \subset (\mu/t)^s \BrigLbreve^+ \otimes_{\Lbreve} D_{\Lbreve}.
\end{equation*}
Moreover, from Lemma \ref{lem:brigl_lbreve_flat}, recall that $\BrigL^+ \rightarrow \BrigLbreve^+$ is flat, and from Lemma \ref{lem:brigl_lbreve_intersect} we have that $\BrigL^+ \cap (t/\mu)\BrigLbreve^+ = (t/\mu)\BrigLbreve^+$, or equivalently, $\BrigLbreve^+ \cap \BrigL^+ [\mu/t] = \BrigL^+$.
So, it follows that we have $\varphi\equivariant$ inclusions
\begin{equation}\label{eq:mldl_in_mutsdl}
	\BrigL^+ \otimes_L D_L \subset \ML(D_L) \subset (\mu/t)^s \BrigL^+ \otimes_L D_L.
\end{equation}
Therefore, similar to \eqref{eq:wach_crys_rigcomp}, we obtain a $\varphi\equivariant$ isomorphism,
\begin{equation}\label{eq:mldl_dcrys_comp}
	\ML(D_L)[\mu/t] \isomorphic \BrigL^+[\mu/t] \otimes_L D_L.
\end{equation}

\begin{lem}\label{rem:hat_flat}
	For each $n \in \NN$, the natural morphism $\BrigLbreve^+ \rightarrow \BLbrevenhat$ is flat, and therefore, the composition $\BrigL^+ \rightarrow \BrigLbreve^+ \rightarrow \BLbrevenhat$ is flat.
\end{lem}
\begin{proof}
	Note that we have a natural isomorphism $\Lbreve(\zeta_{p^n+1}) \isomorphic (\Lbreve(\zeta_{p^{n+1}}) \otimes_{\Lbreve} \BrigLbreve^+)/I$, where $I \subset \Lbreve(\zeta_{p^{n+1}}) \otimes_{\Lbreve} \BrigLbreve^+$ denotes the maximal ideal generated by $\mu-(\zeta_{p^{n+1}}-1)$, and let $(\Lbreve(\zeta_{p^{n+1}}) \otimes_{\Lbreve} \BrigLbreve^+)_{I}$ denote its localisation at $I$.
	Then, the natural map $(\Lbreve(\zeta_{p^{n+1}}) \otimes_{\Lbreve} \BrigLbreve^+)_{I} \rightarrow \BLbrevenhat$, realises the target as the $I\textrm{-adic}$ completion of the source which is a discrete valuation ring, in particular, the preceding map is flat.
	It is easy to see that the first map in the claim factors as the composition $\BrigLbreve^+ \rightarrow \Lbreve(\zeta_{p^{n+1}}) \otimes_{\Lbreve} \BrigLbreve^+ \rightarrow (\Lbreve(\zeta_{p^{n+1}}) \otimes_{\Lbreve} \BrigLbreve^+)_{I} \rightarrow \BLbrevenhat$, where each map is flat, hence, the composition is flat.
	Furthermore, recall that the natural map $\BrigL^+ \rightarrow \BrigLbreve^+$ is flat (see Lemma \ref{lem:brigl_lbreve_flat}), therefore, the composition $\BrigL^+ \rightarrow \BrigLbreve^+ \rightarrow \BLbrevenhat$ is flat as well.
\end{proof}

\begin{lem}\label{lem:descent_brigl+}
	Let us consider $\BLbrevenhat$ as a $\BrigL^+\algebra$ via the composition $i_{L,n} : \BrigL^+ \rightarrow \BrigLbreve^+ \xrightarrow[\sim]{\phi_{\Lbreve}^{-n}} \BrigLbreve^+ \rightarrow \BLbrevenhat$.
	\begin{enumarabicup}
	\item The homomorphism
		\begin{equation*}
			\BLbrevenhat \otimes_{i_{L,n}, \BrigL^+} (\BrigL^+ \otimes_L D_L) \longrightarrow \BLbrevenhat \otimes_{\Lbreve} D_{\Lbreve} \lisomorphic \BLbrevenhat \otimes_L D_L,
		\end{equation*}
		induced by $i_n$ in \eqref{eq:in_briglbreve}, is an isomorphism.

	\item The isomorphism in \textup{(1)} induces an isomorphism,
		\begin{equation*}
			\BLbrevenhat \otimes_{i_{L,n}, \BrigL^+} \ML(D_L) \isomorphic \sum_{i \in \NN} \xitilde_n^{-i}\BLbrevenhat \otimes_L \Fil^i D_L.
		\end{equation*}
		
	\item The $\varphi\equivariant$ homomorphism $\BrigLbreve^+ \otimes_{\BrigL^+} \ML(D_L) \rightarrow \MLbreve(D_{\Lbreve})$, obtained as the $\BrigLbreve^+\textrm{-linear}$ extension of the $\varphi\equivariant$ inclusion $\ML(D_L) \subset \MLbreve(D_{\Lbreve})$, is an isomorphism. 
		Moreover, $\ML(D_L)$ is a finite free $\BrigL^+\module$ of rank $=\dim_L D_L$.
	\end{enumarabicup}
\end{lem}
\begin{proof}
	The proof follows in a manner similar to \cite[Lemma 1.2.1]{kisin-modules}.
	Let us first note that the linearisation of $i_n$ along the morphism $i_{\Lbreve,n} : \BrigLbreve^+ \xrightarrow[\sim]{\phi_{\Lbreve}^{-n}} \BrigLbreve^+ \rightarrow \BLbrevenhat$, yields an isomorphism,
	\begin{equation*}
		\BLbrevenhat \otimes_{i_{\Lbreve,n}, \BrigLbreve^+} (\BrigLbreve^+ \otimes_{\Lbreve} D_{\Lbreve}) \isomorphic \BLbrevenhat \otimes_{\Lbreve} D_{\Lbreve}.
	\end{equation*}
	Moreover, from \eqref{eq:dcrys_llbreve_comp}, we have that $D_{\Lbreve} \isomorphic \Lbreve \otimes_L D_L$, so we can write,
	\begin{equation*}
		\BrigLbreve^+ \otimes_{\BrigL^+} (\BrigL^+ \otimes_L D_L) \isomorphic \BrigLbreve^+ \otimes_L D_L \isomorphic \BrigLbreve^+ \otimes_{\Lbreve} D_{\Lbreve}.
	\end{equation*}
	Then, by extending the composition above along $i_{\Lbreve,n}: \BrigLbreve^+ \rightarrow \BLbrevenhat$, we obtain that $\BLbrevenhat \otimes_{i_{L,n}, \BrigL^+} (\BrigL^+ \otimes_L D_L) \isomorphic \BLbrevenhat \otimes_{\Lbreve} D_{\Lbreve}$, as claimed in (1).

	To show (2), let us write for $k \in \NN$,
	\begin{equation*}
		\MLk(D_L) := \big\{x \in \BrigL^+[\mu/t] \otimes_L D_L \textrm{ such that } i_k(x) \in \Fil^0\big(\BLbrevekhat[1/\xitilde_k] \otimes_{\Lbreve} D_{\Lbreve}\big)\big\}.
	\end{equation*}
	Then, we have that $\ML(D_L) = \cap_{k \in \NN} \MLk(D_L) \subset \BrigL^+[\mu/t] \otimes_L D_L$.
	Moreover, using Lemma \ref{lem:brigl_lbreve_flat} and Lemma \ref{rem:hat_flat}, we see that the morphism $i_{L,n} : \BrigL^+ \rightarrow \BLbrevenhat$ is flat.
	So, we get that 
	\begin{equation*}
		\BLbrevenhat \otimes_{i_{L,n}, \BrigL^+} \ML(D_L) = \cap_{k \in \NN} \big(\BLbrevenhat \otimes_{i_{L,n}, \BrigL^+} \MLk(D_L)\big) \subset \BrigL^+[\mu/t] \otimes_L D_L.
	\end{equation*}
	To prove the claim, it suffices to show that the isomorphism in (1) induces the following two bijections:
	\begin{align*}
		\BLbrevenhat \otimes_{i_{L,n}, \BrigL^+} \MLn(D_L) &\isomorphic \sum_{r \in \NN} \xitilde_n^{-r}\BLbrevenhat \otimes_L \Fil^r D_L,\\
		\BLbrevenhat \otimes_{i_{L,n}, \BrigL^+} \MLk(D_L) &\isomorphic \BLbrevenhat[1/\xitilde_n] \otimes_L D_L, \hspace{2mm} \textrm{ for } k \neq n.
	\end{align*}
	For the first claim, note that by definition, we have a natural inclusion $\BLbrevenhat \otimes_{i_{L,n}, \BrigL^+} \MLn(D_L) \hookrightarrow \sum_{r \in \NN} \xitilde_n^{-r}\BLbrevenhat \otimes_L \Fil^r D_L$.
	For the converse, note that we have $\xitilde_n^{-1} = \frac{1}{p}\varphi^n(\mu/t)\varphi^{n+1}(t/\mu)$ in $\BrigL[\mu/t]$ and $\phi_{\Lbreve}^{-n}\big(\xitilde_n^{-1}\big) = \xitilde_n^{-1}$.
	So, for any $r \in \NN$ and $\xitilde_n^{-r} a \otimes d$ in $\xitilde_n^{-r} \BLbrevenhat \otimes_L \Fil^r D_L$, we have that $\xitilde_n^{-r} \otimes \varphi^n(d)$ is in $\MLn(D_L)$, since $i_n(\xitilde_n^{-r} \otimes \varphi^n(d)) = \xitilde_n^{-r} \otimes d$ is in $\Fil^0\big(\BLbrevenhat[1/\xitilde_n] \otimes_{\Lbreve} D_{\Lbreve}\big)$.
	Therefore, $\xitilde_n^{-r} a \otimes d = a \otimes i_n(\xitilde_n^{-r} \otimes \varphi^n(d))$ is in the image of $\BLbrevenhat \otimes_{i_{L,n}, \BrigL^+} \MLn(D_L)$.
	For the second claim, again note that by definition, we have a natural inclusion $\BLbrevenhat \otimes_{i_{L,n}, \BrigL^+} \MLk(D_L) \hookrightarrow \BLbrevenhat[1/\xitilde_n] \otimes_L D_L$.
	For the converse, take $\xitilde_n^{-r} a \otimes d$ in $\BLbrevenhat[1/\xitilde_n] \otimes_L D_L$, for some $r \in \NN$.
	Then, note that $\xitilde_n$ is a unit in $\BLbrevekhat$, since $\zeta_{p^{n+1}}-1$ is not a root of $\xitilde_k$ as $k \neq n$.
	So, we get that $i_k(\xitilde_n^{-r} \otimes \varphi^n(d)) = \xitilde_n^{-r} \otimes \varphi^{n-k}(d)$ is in $\BLbrevekhat \otimes_{\Lbreve} D_{\Lbreve} \subset \Fil^0\big(\BLbrevekhat[1/\xitilde_k] \otimes_{\Lbreve} D_{\Lbreve}\big)$, since $\Fil^0 D_L = D_L$.
	In particular, we have that $\xitilde_n^{-r} \otimes \varphi^n(d)$ is in $\MLk(D_L)$.
	Therefore, $\xitilde_n^{-r} a \otimes d = a \otimes i_n(\xitilde_n^{-r} \otimes \varphi^n(d))$ is in the image of $\BLbrevenhat \otimes_{i_{L,n}, \BrigL^+} \MLk(D_L)$.

	For (3), note that we have natural inclusions $\BrigLbreve^+ \otimes_L D_L \subset \BrigLbreve^+ \otimes_{\BrigL^+} \ML(D_L) \subset \MLbreve(D_{\Lbreve}) \subset (\mu/t)^s\BrigLbreve \otimes_L D_L$, where the first two inclusions follow since the map $\BrigL^+ \rightarrow \BrigLbreve^+$ is flat (see Lemma \ref{lem:brigl_lbreve_flat}) and $\ML(D_L) \subset (\mu/t)^s \BrigL^+ \otimes_L D_L$ from \eqref{eq:mldl_in_mutsdl}.
	So, we get that $(t/\mu)^s$ kills the cokernel of the map $\BrigLbreve^+ \otimes_{\BrigL^+} \ML(D_L) \rightarrow \MLbreve(D_{\Lbreve})$.
	Moreover, note that $\MLbreve(D_{\Lbreve}) \subset (\mu/t)^s \BrigLbreve^+ \otimes_{\Lbreve} D_{\Lbreve}$ is a closed submodule by \cite[Lemma 1.1.5, Lemma 1.2.2]{kisin-modules}.
	Now, since $\BrigL^+ \subset \BrigLbreve^+$ is a closed subring, therefore, we get that $\ML(D_L) \subset (\mu/t)^s \BrigL^+ \otimes_L D_L$ is closed and hence finite free over $\BrigL^+$ by Remark \ref{rem:closedsubmod_free}, and of rank $=\dim_L D_L$ by the isomorphism shown below.

	Let us write $\BrigL^+ = \lim_{\rho} \pazo(D(L, \rho))$ as the limit of rings of analytic functions on closed disks $D(L, \rho)$ of radius $0 < \rho < 1$ (see Remark \ref{rem:brigl_topology}); similarly write $\BrigLbreve^+ = \lim_{\rho} \pazo(D(\Lbreve, \rho))$.
	Since $\ML(D_L)$ and $\MLbreve(D_{\Lbreve})$ are free modules over their respective base rings, therefore, we have that $\ML(D_L) \isomorphic \lim_{\rho} (\pazo(D(L, \rho)) \otimes_{\BrigL^+} \ML(D_L))$ and $\MLbreve(D_{\Lbreve}) \isomorphic \lim_{\rho} (\pazo(D(\Lbreve, \rho)) \otimes_{\BrigLbreve^+} \MLbreve(D_{\Lbreve}))$.
	Then, to show our claim, it is enough to show that the map,
	\begin{equation}\label{eq:odlbreverho_map}
		\pazo(D(\Lbreve, \rho)) \otimes_{\BrigL^+} \ML(D_L) \longrightarrow \pazo(D(\Lbreve, \rho)) \otimes_{\BrigLbreve^+} \MLbreve(D_{\Lbreve}),
	\end{equation}
	is a bijection.
	Note that $\pazo(D(\Lbreve, \rho))$ is a domain, so injectivity of \eqref{eq:odlbreverho_map} can be checked after passing to the fraction field of $\pazo(D(\Lbreve, \rho))$.
	To check that \eqref{eq:odlbreverho_map} is surjective, let $Q$ denote the cokernel of \eqref{eq:odlbreverho_map} and we will show that $Q = 0$.
	Note that $Q$ is a finitely generated $S := \pazo(D(\Lbreve, \rho))\module$ killed by $(t/\mu)^s$ and $S$ is a principal ideal domain (see \cite[Chapter 2, Corollary 10]{bosch}).
	So, by the structure theorem of finitely generated modules over $S$, we can write $Q = \oplus S/\fraka_i$, where $\fraka_i = (a_i)$ for some nonzero primary elements $a_i \in S$ and such that $a_i$ divides $(t/\mu)^s$, for each $i$.
	Note that $\sqrt{\fraka_i}$ is a maximal ideal of $S$ and $Q_{\sqrt{\fraka_i}} = S/\fraka_i$, so to obtain that $Q = 0$, it is enough to show that $Q_{\sqrt{\fraka_i}} = 0$.
	From \cite[Chapter 2, Corollary 13]{bosch} note that each maximal ideal $\sqrt{\fraka_i}$ corresponds to a zero of $(t/\mu)^s$, in particular, we are reduced to showing that $Q$ vanishes at zeros of $t/\mu$.
	This follows from (2).
	Hence, we get that \eqref{eq:odlbreverho_map} is an isomorphism and passing to the limit over $\rho$ we obtain that $\BrigLbreve^+ \otimes_{\BrigL^+} \ML(D_L) \isomorphic \MLbreve(D_{\Lbreve})$.
\end{proof}

\begin{lem}\label{lem:mldl_props}
	We have following properties for the $\BrigL^+\module$ $\ML(D_L)$:
	\begin{enumarabicup}
	\item The cokernel of the injective map $1 \otimes \varphi : \varphi^*(\ML(D_L)) \rightarrow \ML(D_L)$ is killed by $[p]_q^s$.

	\item $\ML(D_L)$ is pure of slope 0, i.e.\ the $\BrigLdag\module$ $\BrigLdag \otimes_{\BrigL^+} \ML(D_L)$ is pure of slope 0 in the sense of \cite[\S 6.3]{kedlaya-monodromy}.
	\end{enumarabicup}
\end{lem}
\begin{proof}
	For (1), let us first note the following commutative diagram with exact rows:
	\begin{center}
		\begin{tikzcd}
			0 \arrow[r] & \ML(D_L) \arrow[r] \arrow[d] & \BrigL^+[\mu/t] \otimes_L D_L \arrow[r] \arrow[d] & Q \arrow[r] \arrow[d] & 0\\
			0 \arrow[r] & \MLbreve(D_{\Lbreve}) \arrow[r] & \BrigLbreve^+[\mu/t] \otimes_{\Lbreve} D_{\Lbreve}\arrow[r] & \BrigLbreve^+ \otimes_{\BrigL^+} Q \arrow[r] & 0,
		\end{tikzcd}
	\end{center}
	where $Q$ is the cokernel of the left horizontal arrow in the first row.
	All the maps above are $\varphi\equivariant$ and the vertical maps are injective (see \eqref{eq:mldl_in_mutsdl}, Lemma \ref{lem:nlbreve_wach}, Defnition \ref{defi:mldl} and Lemma \ref{lem:descent_brigl+} (3)).
	From Remark \ref{lem:brigl+_frob_ff}, Remark \ref{rem:briglbreve+}, Lemma \ref{lem:brigl_lbreve_intersect} and Lemma \ref{lem:brigl_lbreve_flat} recall that the maps $\varphi_L : \BrigL^+ \rightarrow \BrigL^+$ and $\varphi_{\Lbreve} : \BrigLbreve^+ \rightarrow \BrigLbreve^+$ are faithfully flat (we write $\varphi$ with subscripts to avoid confusion), the natural map $\BrigL^+ \rightarrow \BrigLbreve^+$ is flat and $\BrigLbreve^+ \cap \BrigL^+[\mu/t] = \BrigL^+$.
	Then, using Lemma \ref{lem:descent_brigl+} (3) and that $D_{\Lbreve} \isomorphic \Lbreve \otimes_L D_L$ from \eqref{eq:dcrys_llbreve_comp}, we get that $\varphi_{\Lbreve}^*(\MLbreve(D_{\Lbreve})) \isomorphic \BrigLbreve^+ \otimes_{\BrigL^+} \varphi_L^*(\ML(D_L))$ and $\varphi_L^*(\BrigL^+[\mu/t] \otimes_L D_L) \isomorphic \BrigL^+[\mu/t] \otimes_{\BrigL^+} \varphi_L^*(\BrigL^+ \otimes_L D_L) \subset \BrigLbreve^+[\mu/t] \otimes_{\BrigL^+[\mu/t]} \varphi_L^*(\BrigL^+[\mu/t] \otimes_L D_L) \isomorphic \varphi_{\Lbreve}^*(\BrigLbreve^+[\mu/t] \otimes_{\Lbreve} D_{\Lbreve})$.
	So, from the preceding discussion and the exactness of both rows in the diagram above, it follows that,
	\begin{align*}
		\varphi_L^*(\ML(D_L)) &= \big(\BrigL^+[\mu/t] \otimes_{\BrigL^+} \varphi_L^*(\ML(D_L))\big) \cap \big(\BrigLbreve^+ \otimes_{\BrigL^+} \varphi_L^*(\ML(D_L))\big)\\
		&\isomorphic \varphi_L^*\big(\BrigL^+[\mu/t] \otimes_L D_L\big) \cap \varphi_{\Lbreve}^*(\MLbreve(D_{\Lbreve})) \subset \varphi_{\Lbreve}^*(\BrigLbreve^+[\mu/t] \otimes_{\Lbreve} D_{\Lbreve}).
	\end{align*}
	Now, let $x$ in $\ML(D_L) \subset \MLbreve(D_{\Lbreve})$, then there exists some $y$ in $\varphi_{\Lbreve}^*(\MLbreve(D_{\Lbreve}))$ such that $(1 \otimes \varphi)y = \xitilde^s x$.
	Recall that $1 \otimes \varphi : \varphi_L^*(D_L) \isomorphic D_L$ and $\varphi(\mu/t) = (\xitilde\mu)/(pt)$, therefore, the cokernel of the induced map $1 \otimes \varphi : \varphi_L^*((\mu/t)^s \BrigL^+ \otimes_L D_L) \rightarrow (\mu/t)^s \BrigL^+ \otimes_L D_L$ is killed by $\xitilde^s$, in particular, $\xitilde^s x$ is in $(1 \otimes \varphi)\varphi_L^*((\mu/t)^s \BrigL^+ \otimes_L D_L)$.
	Since $1 \otimes \varphi$ is injective on $\varphi_{\Lbreve}^*((\mu/t)^s \BrigLbreve^+ \otimes_{\Lbreve} D_{\Lbreve})$, therefore, we get that $y$ is in $\varphi_L^*((\mu/t)^s \BrigL^+ \otimes_L D_L) \cap \varphi_{\Lbreve}^*(\MLbreve(D_{\Lbreve})) = \varphi_L^*(\ML(D_L))$.
	In particular, the cokernel of the natural map $1 \otimes \varphi : \varphi_L^*(\ML(D_L)) \rightarrow \ML(D_L)$ is killed by $\xitilde^s$.

	For (2), note that from Lemma \ref{lem:descent_brigl+} (3), we have that $\BrigLbreve^+ \otimes_{\BrigL^+} \ML(D_L) \isomorphic \MLbreve(D_{\Lbreve})$.
	Moreover, from \cite[Theorem 6.10]{kedlaya-monodromy} we obtain a slope filtration on $\BrigLdag \otimes_{\BrigL^+} \ML(D_L)$ such that base changing this slope filtration along $\BrigLdag \rightarrow \BrigLbrevedag$ gives a slope filtration on $\BrigLbrevedag \otimes_{\BrigLbreve^+} \MLbreve(D_{\Lbreve})$.
	However, from \cite[Theorem 1.3.8]{kisin-modules} and \cite[Proposition 2.3.3]{kisin-ren}, we know that $\BrigLbrevedag \otimes_{\BrigLbreve^+} \MLbreve(D_{\Lbreve})$ is pure of slope 0.
	Therefore, we must have that $\ML(D_L)$ is pure of slope 0.
\end{proof}

\subsection{Stability under Galois action}\label{subsec:nriglv_galois_action}

In this subsection, we will define and study a finite free $(\varphi, \Gamma_L)\module$ $\NrigL(V)$ over $\BrigL^+$, of slope 0, and obtained from the $\BrigL^+\module$ in Definition \ref{defi:mldl}.
From \S \ref{subsubsec:analytic_rings}, recall that we have identifications $\BrigLtilde^+ = (\Brigtilde^+)^{H_L} = \cap_{n \in \NN} \varphi^n(\Bcrys^+(\OLinfty))$, where the last equality follows since $\Bcrys^+(\OLinfty) = \Bcrys^+(\OLbar)^{H_L}$ (see \S \ref{subsubsec:crystalline_rings}).
Moreover, using the isomorphism in Lemma \ref{lem:faltings_isomorphism_dcrys} and Remark \ref{rem:b+dcrys_hlinv}, we see that $\Bcrys(\OLinfty) \otimes_{L} \ODcrysL(V)$ is equipped with an action of $\Gamma_L$.
We have that $\BrigLtilde^+ \otimes_L \ODcrysL(V) \subset \Bcrys(\OLinfty) \otimes_{L} \ODcrysL(V)$ and we claim the following:
\begin{lem}\label{lem:brigtilde_odcrys_glact}
	The $\BrigLtilde^+\module$ $\BrigLtilde^+ \otimes_L \ODcrysL(V)$ is stable under the action of $\Gamma_L$.
	For $a \otimes x$ in $\BrigLtilde^+ \otimes_L \ODcrysL(V)$, this action can be explicitly described by the foloowing formula:
	\begin{equation*}
		g(a \otimes x) = g(a) \otimes \textstyle \sum_{\smbfk \in \NN^d} \prod_{i=1}^d \partial_i^{k_i}(x) \prod_{i=1}^d (g([X_i^{\flat}]) - [X_i^{\flat}])^{[k_i]}, \hspace{2mm} \textrm{for} \hspace{1mm} g \in \Gamma_L.
	\end{equation*}
\end{lem}
\begin{proof}
	The non-canonical $(\varphi, G_{\Lbreve})\equivariant$ $L\algebra$ structure on $\OBcrys^+(\OLinfty)$ from \S \ref{subsubsec:crystalline_rings}, extends to a $(\varphi, G_{\Lbreve})\equivariant$ $\Lbreve\algebra$ structure, and it provides $(\varphi, G_{\Lbreve})\equivariant$ $L\algebra$ and $\Lbreve\algebra$ structures on $\Bcrys^+(\OLinfty)$, via the composition $L \rightarrow \Lbreve \rightarrow \OBcrys^+(\OLinfty) \twoheadrightarrow \Bcrys^+(\OLinfty)$, where the last map is the projection map described before Lemma \ref{lem:faltings_isomorphism_dcrys}.
	Moreover, recall that we have $L \otimes_{\varphi^n, L} \ODcrysL(V) \isomorphic \ODcrysL(V)$, for all $n \in \NN$.
	So, we can write,
	\begin{align*}
		\Bcrys^+(\OLinfty) \otimes_{L} \ODcrysL(V) &\isomorphic \Bcrys^+(\OLinfty) \otimes_{\Lbreve} \DcrysLbreve(V)\\
		&\isomorphic \Bcrys^+(\OLinfty) \otimes_{\varphi_{\Lbreve}^{-n}, \Lbreve} (\Lbreve \otimes_{\varphi_{\Lbreve}^n, \Lbreve} \DcrysLbreve(V)).
	\end{align*}
	Applying $\varphi^n$ to the isomorphism above gives that $\varphi^n(\Bcrys^+(\OLinfty) \otimes_{L} \ODcrysL(V)) \isomorphic \varphi^n(\Bcrys^+(\OLinfty)) \otimes_{L} \ODcrysL(V)$.
	Note that the Frobenius endomorphism $\varphi$ on $\Bcrys^+(\OLinfty) \otimes_{L} \ODcrys(V)$ commutes with the action of $\Gamma_L$ described in Remark \ref{rem:b+dcrys_hlinv}.
	Therefore, the following is stable under the $\Gamma_L\action$:
	\begin{align*}
		\cap_{n \in \NN} \varphi^n(\Bcrys^+(\OLinfty) \otimes_L \ODcrysL(V)) &\isomorphic (\cap_{n \in \NN} \varphi^n(\Bcrys^+(\OLinfty))) \otimes_L \ODcrysL(V)\\
			&= \BrigLtilde^+ \otimes_L \ODcrysL(V).
	\end{align*}
	The second claim follows from Lemma \ref{lem:odcrys_gammal_action}.
\end{proof}

Extending the isomorphism in \eqref{eq:wach_crys_rigcomp} along the map $\BrigLbreve^+[\mu/t] \rightarrow \BrigLtilde^+[\mu/t]$ (see \S \ref{subsubsec:periodrings_Lbreve}), we obtain a $\varphi\equivariant$ isomorphism $\BrigLtilde^+[\mu/t] \otimes_{\Lbreve} \DcrysLbreve(V) \isomorphic \BrigLtilde^+[\mu/t] \otimes_{\BLbreve^+} \NLbreve(V)$.
Recall that for any $g \in \Gamma_L$, we have that $g(t) = \chi(g) t$ and $g(\mu) = (1+\mu)^{\chi(g)}-1$, where $\chi$ is the $\padic$ cyclotomic character.
Now, using that $\Lbreve \otimes_L \ODcrysL(V) \isomorphic \DcrysLbreve(V)$, we get $\varphi\equivariant$ isomorphisms $\BrigLtilde^+[\mu/t] \otimes_L \ODcrysL(V) \isomorphic \BrigLtilde^+[\mu/t] \otimes_{\Lbreve} \DcrysLbreve(V) \isomorphic \BrigLtilde^+[\mu/t] \otimes_{\BLbreve^+} \NLbreve(V)$, and we equip the last term with a $\Gamma_L\action$ by transport of structure via this isomorphism.
In particular, the preceding discussion induces an action of $\Gamma_L$ over $\BrigLtilde^+[\mu/t] \otimes_{\BrigLbreve^+} \NrigLbreve(V) = \BrigLtilde^+[\mu/t] \otimes_{\BLbreve^+} \NLbreve(V)$.
Our first objective is to show that $\BrigLtilde^+ \otimes_{\BrigL^+} \NrigLbreve(V) \subset \BrigLtilde^+[\mu/t] \otimes_{\BrigL^+} \NrigLbreve(V)$ is stable under the action of $\Gamma_L$ on the latter.
We will do this by embedding everything into $\Bcrys(\OLbar) \otimes_{\QQ_p} V$.

Let us fix some elements in $\Acrys(\OLinfty)$.
For $n \in \NN$, let $n = (p-1)f(n) + r(n)$ with $r(n), f(n) \in \NN$ and $0 \leqslant r(n) < p-1$.
Set $t^{\{n\}} := \frac{t^n}{f(n)!p^{f(n)}}$ and $\Lambda := \big\{\sum_{n \in \NN} a_n t^{\{n\}}, \hspace{1mm} \textrm{with} \hspace{1mm} a_n \in O_F \hspace{1mm} \textrm{such that} \hspace{1mm} a_n \rightarrow 0 \hspace{1mm} \textrm{as} \hspace{1mm} n \rightarrow +\infty\big\} = O_F[t, (t^{p-1}/p)^{[k]}, k \in \NN\big]^{\wedge}$, where ${}^{\wedge}$ denotes the $\padic$ completion.
Then, we have an isomorphism of rings,
\begin{equation*}
	O_F[\mu, (\mu^{p-1}/p)^{[k]}, k \in \NN]^{\wedge} \isomorphic \Lambda,
\end{equation*}
via the map $\mu \mapsto \exp(t)-1$ with the inverse map given as $t \mapsto \log(1+\mu)$ (see \cite[Lemme 6.2.13]{brinon-relatif}).
Furthermore, for $r \in \NN$ and $A := \Ainf(\OLinfty)$, $\Ainf(\OLbar)$, $\Acrys(\OLinfty)$ or $\Acrys(\OLbar)$, set
\begin{equation}\label{eq:ira}
	I^{(r)} A := \{a \in A \textrm{ such that } \varphi^n(a) \in \Fil^r A \textrm{ for all } n \in \NN\}.
\end{equation}
\begin{lem}\label{lem:ira_gen}
	We note the following facts:
	\begin{enumarabicup}
	\item $t^{p-1} \in p\Acrys(\OLinfty)$, $t^{\{n\}} \in \Acrys(\OLinfty)$ and $t/\mu$ is a unit in $\Lambda \subset \Acrys(\OLinfty)$.

	\item For any $r \in \NN$, we have that $I^{(r)}\Ainf(\OLinfty) = \mu^r\Ainf(\OLinfty)$ and $I^{(p-1)}\Ainf(\OLinfty) = \mu^{p-1}\Ainf(\OLinfty)$.

	\item Let $S = O_F\llbracket \mu \rrbracket$, then the natural map $\Ainf(\OLinfty) \widehat{\otimes}_S \Lambda \rightarrow \Acrys(\OLinfty)$, defined via $\sum_{k \in \NN} a_k \otimes (\mu^{p-1}/p)^{[k]} \mapsto \sum_{k \in \NN} a_k (\mu^{p-1}/p)^{[k]}$, is continuous for the $\padic$ topology and an isomorphism of $\Ainf(\OLinfty)\textrm{-algebras}$.

	\item The ideal $I^{(r)} \Acrys(\OLinfty)$ is topologically generated by $t^{\{s\}}$, for $s \geqslant r$.

	\item The natural map $\Ainf(\OLinfty)/I^{(r)}\Ainf(\OLinfty) \rightarrow \Acrys(\OLinfty)/I^{(r)}\Acrys(\OLinfty)$ is injective and the cokernel is killed by $m!p^{m}$, where $m = \lfloor\frac{r}{p-1}\rfloor$.
	\end{enumarabicup}
	Similar statements are true for $\Ainf(\OLbar)$ and $\Acrys(\OLbar)$.
\end{lem}
\begin{proof}
	All claims except (3) follow from \cite[\S 5.2]{fontaine-corps-periodes} and \cite[\S A3]{tsuji-cst}.
	The proof of the claim in (3) follows in a manner similar to the proof of \cite[Proposition 6.2.14]{brinon-relatif}.
\end{proof}

\begin{rem}\label{rem:binf_bcrys_modis}
	Note that the $\QQ_p\textrm{-algebras}$ $\Binf(\OLinfty) := \Ainf(\OLinfty)[1/p]$, $\Bcrys^+(\OLinfty) := \Acrys(\OLinfty)[1/p]$ and $\BrigLtilde^+$ naturally embed into the $\QQ_p\textrm{-algebra}$ $\Bcrys(\OLinfty)$, and we equip the former rings with a filtration induced from the natural filtration on $\Bcrys(\OLinfty)$ (see \S \ref{subsubsec:crystalline_rings}).
	Then, one can define ideals similar to \eqref{eq:ira} for these rings and from Lemma \ref{lem:ira_gen} (5), we obtain isomorphisms $\Binf(\OLinfty)/I^{(r)}\Binf(\OLinfty) \isomorphic \Bcrys^+(\OLinfty)/I^{(r)}\Bcrys^+(\OLinfty)$ and $\Binf(\OLbar)/I^{(r)}\Binf(\OLbar) \isomorphic \Bcrys^+(\OLbar)/I^{(r)}\Bcrys^+(\OLbar)$.
\end{rem}

\begin{prop}\label{prop:nlinf_gammastab}
	The $\Binf(\OLinfty)\module$ 
	\begin{equation*}
		\NLbreveinfty(V) := \Binf(\OLinfty) \otimes_{\BLbreve^+} \NLbreve(V) \subset (\Binf(\OLbar) \otimes_{\QQ_p} V)^{H_L} = \Dtilde_L^+(V),
	\end{equation*}
	is stable under the residual action of $\Gamma_L$ on $\Dtilde_L^+(V)$ and we equip $\NLbreveinfty(V)$ with this action.
	Then, we have a natural $\Gamma_L\equivariant$ embedding $\NLbreveinfty(V) \subset \Bcrys(\OLinfty) \otimes_L \ODcrysL(V)$ (see Remark \ref{rem:glaction_odcrys} and \eqref{eq:b+dcrys_hlinv} for the $\Gamma_L\action$ on the latter).
\end{prop}
\begin{proof}
	From Lemma \ref{lem:wach_props} (2), consider the following exact sequence:
	\begin{equation}\label{eq:nlbreveinf_mus}
		0 \longrightarrow \mu^s \Dtilde^+_L(V) \longrightarrow \NLbreveinfty(V) \longrightarrow \NLbreveinfty(V)/\mu^s\Dtilde_L^+(V) \longrightarrow 0,
	\end{equation}
	where we know that $\mu^s \Dtilde_L^+(V) \subset \Dtilde_L^+(V)$ is stable under the action of $\Gamma_L$.
	Therefore, to show that the middle term above is stable under the action of $\Gamma_L$, it is enough to show that for the inclusion,
	\begin{align*}
		\NLbreveinfty(V)/\mu^s\Dtilde_L^+(V) \subset \Dtilde_L^+(V)/\mu^s \Dtilde_L^+(V) &\subset (\Binf(\OLbar)/\mu^s\Binf(\OLbar) \otimes_{\QQ_p} V)^{H_L}\\
			&\subset \Binf(\OLbar)/\mu^s\Binf(\OLbar) \otimes_{\QQ_p} V,
	\end{align*}
	the image of the first term in the last term is stable under the action of $G_L$.

	Note that from Lemma \ref{lem:wac_crys_rep_comp}, we have a natural $\Bcrys(\OLbar)\linear$ and $(\varphi, G_{\Lbreve})\equivariant$ isomorphism $\Bcrys(\OLbar) \otimes_{\BLbreve^+} \NLbreve(V) \isomorphic \Bcrys(\OLbar) \otimes_{\QQ_p} V$.
	In view of Remark \ref{rem:binf_bcrys_modis}, let us set,
	\begin{equation*}
		M := (I^{(s)} \Bcrys^+(\OLbar) \otimes_{\QQ_p} V) \cap (\Bcrys^+(\OLinfty) \otimes_{\BLbreve^+} \NLbreve(V)) \subset \Bcrys(\OLbar) \otimes_{\QQ_p} V.
	\end{equation*}
	Then we obtain a diagram with exact rows
	\begin{center}
		\begin{tikzcd}
			0 \arrow[r] & \mu^s \Dtilde_L^+(V) \arrow[r] \arrow[d] & \NLbreveinfty(V) \arrow[r] \arrow[d] & \NLbreveinfty(V)/\mu^s\Dtilde_L^+(V) \arrow[r] \arrow[d] & 0\\
			0 \arrow[r] & M \arrow[r] & \Bcrys^+(\OLinfty) \otimes_{\BLbreve^+} \NLbreve(V) \arrow[r] & (\Bcrys^+(\OLinfty) \otimes_{\BLbreve^+} \NLbreve(V))/M \arrow[r] & 0.
		\end{tikzcd}
	\end{center}
	The left vertical arrow is injective by Lemma \ref{lem:wach_props} (2) and the middle arrow is obviously injective.
	
\begin{lem}\label{lem:nlinf_mod_mus}
	The inclusion $\NLbreveinfty(V) \subset \Bcrys^+(\OLinfty) \otimes_{\BLbreve^+} \NLbreve(V)$ induces a $\GammaLbreve\textrm{-equivariant}$ isomorphism of $\Binf(\OLinfty)\modules$ $\NLbreveinfty(V)/\mu^s\Dtilde_L^+(V) \isomorphic (\Bcrys^+(\OLinfty) \otimes_{\BLbreve^+} \NLbreve(V))/M$.
\end{lem}
\begin{proof}
	First, we observe that by Lemma \ref{lem:wach_props} (2) we have that,
	\begin{align*}
		M \cap \NLbreveinfty(V) &= (I^{(s)} \Bcrys^+(\OLbar) \otimes_{\QQ_p} V) \cap \NLbreveinfty(V) \\
		&\subset (I^{(s)} \Bcrys^+(\OLbar) \otimes_{\QQ_p} V) \cap \Dtilde_L^+(V) \subset \mu^s \Dtilde_L^+(V).
	\end{align*}
	Therefore, we get that the rightmost vertical map in the diagram above is injective.
	Next, we need to show that $\NLbreveinfty(V) + M = \Bcrys^+(\OLinfty) \otimes_{\BLbreve^+} \NLbreve(V)$.
	It is clear that the left expression is contained in the right.
	To show the converse, let $x$ in $\Bcrys^+(\OLinfty) \otimes_{\BLbreve^+} \NLbreve(V)$.
	Then, for $m \in \NN$ large enough, we have that $p^m x$ is in $\Acrys(\OLinfty) \otimes_{\ALbreve^+} \NLbreve(T)$.
	By the isomorphism in Lemma \ref{lem:ira_gen} (3), for $r = \lceil \frac{s}{p-1}\rceil$, $k \in \NN$ and $x_k$ in $\NLbreve(T)$ such that $x_k \rightarrow 0$ as $k \rightarrow +\infty$, we can write
	\begin{equation*}
		p^m x = \sum_{k \in \NN} x_k (\mu^{p-1}/p)^{[k]} = \sum_{0 \leqslant k \leqslant r-1} x_k (\mu^{p-1}/p)^{[k]} + \sum_{k \geqslant r} x_k (\mu/p)^{[k]}.
	\end{equation*}
	Clearly, the first summation in the rightmost expression is in $\NLbreveinfty(V)$.
	Moreover, from Lemma \ref{lem:ira_gen} (1) there exists some $v \in \Lambda^{\times}$, such that $\mu^{p-1}/p = v t^{p-1}/p$.
	Therefore, we obtain that the second summation is in $(I^{(s)}\Acrys(\OLbar) \otimes_{\ZZ_p} T) \cap (\Acrys(\OLinfty) \otimes_{\ALbreve^+} \NLbreve(T)) \subset M$.
	Hence, $x$ is in $\NLbreveinfty(V) + M$.
\end{proof}

	Let us now consider the diagram \eqref{eq:galois_action} below with the following description:
	in \eqref{eq:galois_action} the bottom horizontal arrow is a $(\varphi, G_L)\equivariant$ isomorphism since $V$ is a crystalline representation of $G_L$.
	The left vertical arrow from the fourth to the third row is induced by the projection $\OBcrys(\OLbar) \twoheadrightarrow \Bcrys(\OLbar)$, via $X_i \mapsto [X_i^{\flat}]$, it admits a section as in \eqref{eq:bcryslin_section}, it is evidently $\varphi\equivariant$ and it is $G_L\equivariant$ since the codomain is equipped with a $G_L\action$ by transport of structure from the domain (see Remark \ref{rem:glaction_odcrys}).
	The right vertical arrow from the fourth to the third row is also induced by the projection $\OBcrys(\OLbar) \twoheadrightarrow \Bcrys(\OLbar)$, it admits a natural section $\Bcrys(\OLbar) \otimes_{\QQ_p} V \rightarrow (\OBcrys(\OLbar) \otimes V)^{\partial=0}$ and it is naturally $(\varphi, G_L)\equivariant$.
	The horizontal arrow in the third row is the inverse of the isomorphism in Lemma \ref{lem:faltings_isomorphism_dcrys}, which is given as the composition of the inverse of the bottom left vertical arrow, the bottom horizontal arrow and the bottom right vertical arrow, and it is $(\varphi, G_L)\equivariant$ by the preceding discussion and Remark \ref{rem:glaction_odcrys}.
	In particular, we get that the lower square is commutative and $(\varphi, G_L)\equivariant$.
	Next, the left vertical arrow from the third to the second row is an isomorphism since $\Lbreve \otimes_L \ODcrysL(V) \isomorphic \DcrysLbreve(V)$ by \eqref{eq:dcrys_llbreve_comp} and its $(\varphi, G_{\Lbreve})\textrm{-equivariance}$ can either be checked by the explicit formula in Remark \ref{rem:glaction_odcrys} or by observing that the non-canonical map $L \rightarrow \Lbreve \rightarrow \Bcrys(\OLbar)$ is $(\varphi, G_{\Lbreve})\equivariant$ (see the proof of Lemma \ref{lem:brigtilde_odcrys_glact}).
	The horizontal arrow in the second row is a $(\varphi, G_{\Lbreve})\equivariant$ isomorphism since $V$ is a crystalline representation of $G_{\Lbreve}$.
	Commutativity of the middle square follows since the outer square between the second and the fourth row as well as the lower square are commutative.
	Commutativity and $(\varphi, G_{\Lbreve})\textrm{-equivariance}$ of the top square follows from Lemma \ref{lem:wac_crys_rep_comp}.
	\begin{equation}\label{eq:galois_action}
		\begin{tikzcd}
			\Bcrys^+(\OLinfty) \otimes_{\Lbreve} \DcrysLbreve(V) \arrow[r, "\sim"] \arrow[d, hookrightarrow] & \Bcrys^+(\OLinfty) \otimes_{\BLbreve^+} \NLbreve(V) \arrow[d, hookrightarrow] \\
			\Bcrys(\OLbar) \otimes_{\Lbreve} \DcrysLbreve(V) \arrow[r, "\sim"] & \Bcrys(\OLbar) \otimes_{\QQ_p} V \arrow[d, equal] \\
			\Bcrys(\OLbar) \otimes_{L} \ODcrysL(V) \arrow[r, "\sim"] \arrow[u, "\wr"] & \Bcrys(\OLbar) \otimes_{\QQ_p} V \\
			(\OBcrys(\OLbar) \otimes_L \ODcrysL(V))^{\partial=0} \arrow[r, "\sim"] \arrow[u, "\wr"] & (\OBcrys(\OLbar) \otimes_{\QQ_p} V)^{\partial=0} \arrow[u, "\wr"].
		\end{tikzcd}
	\end{equation}

	Furthermore, in the diagram \eqref{eq:galois_action}, the image of composition of the top two left vertical maps inside $\Bcrys(\OLbar) \otimes_L \ODcrysL(V)$ is stable under the action of $G_L$ by Remark \ref{rem:glaction_odcrys}.
	So the image of composition of the top two right vertical maps inside $\Bcrys(\OLbar) \otimes_{\QQ_p} V$ is stable under the action of $G_L$, and it follows that its image $(\Bcrys^+(\OLinfty) \otimes_{\BLbreve^+} \NLbreve(V))/M \subset \Bcrys^+(\OLbar)/I^{(s)}\Bcrys^+(\OLbar) \otimes_{\QQ_p} V \isomorphic \Binf(\OLbar)/\mu^s\Binf(\OLbar) \otimes_{\QQ_p} V$, is stable under the action of $G_L$.
	Therefore, from Lemma \ref{lem:nlinf_mod_mus}, we obtain that the image of $\NLbreveinfty(V)/\mu^s\Dtilde_L^+(V) \subset \Binf(\OLbar)/\mu^s \otimes_{\QQ_p} V$ is stable under the action of $G_L$.
	Hence, from \eqref{eq:nlbreveinf_mus} we conclude that $\NLbreveinfty(V)$ is stable under the action of $\Gamma_L$ and the following natural composition is $\Gamma_L\equivariant$:
	\begin{equation}\label{eq:ninf_ncris_compatible}
		\Binf(\OLinfty) \otimes_{\BLbreve^+} \NLbreve(V) \subset \Bcrys^+(\OLinfty) \otimes_{\BLbreve^+} \NLbreve(V) \isomorphic \Bcrys^+(\OLinfty) \otimes_{L} \ODcrysL(V).
	\end{equation}
	This concludes our proof.
\end{proof}

Recall that $\NrigLbreve(V) = \BrigLbreve^+ \otimes_{\BLbreve^+} \NLbreve(V)$ and we note the following:
\begin{cor}\label{cor:nlrig_gammastab}
	Extending $\BrigLtilde^+\textrm{-linearly}$ the embedding $\NLbreveinfty(V) \subset \Bcrys(\OLinfty) \otimes_L \ODcrysL(V)$ from Proposition \ref{prop:nlinf_gammastab}, gives an identification of $\BrigLtilde^+\textrm{-submodules}$ of $\Bcrys(\OLinfty) \otimes_L \ODcrysL(V)$ as,
	\begin{equation*}
		\BrigLtilde^+ \otimes_{\Binf(\OLinfty)} \NLbreveinfty(V) = \BrigLtilde^+ \otimes_{\BLbreve^+} \NLbreve(V) = \BrigLtilde^+ \otimes_{\BrigLbreve^+} \NrigLbreve(V),
	\end{equation*}
	stable under the $\Gamma_L\action$ on $\Bcrys(\OLinfty) \otimes_L \ODcrysL(V)$.
\end{cor}
\begin{proof}
	The equality in the claim follows from definitions and the compatibility of $\Gamma_L\textrm{-actions}$ follows from \eqref{eq:ninf_ncris_compatible}.
	Then, by using \eqref{eq:wach_crys_rigcomp}, we have that $\BrigLtilde^+[\mu/t] \otimes_{\BLbreve^+} \NLbreve(V) \isomorphic \BrigLtilde^+[\mu/t] \otimes_{\Lbreve} \DcrysLbreve(V) \subset \Bcrys(\OLinfty) \otimes_{\Lbreve} \DcrysLbreve(V) = \Bcrys(\OLinfty) \otimes_L \ODcrysL(V)$.
	Therefore, we can view $\BrigLtilde^+ \otimes_{\BLbreve^+} \NLbreve(V)$ as a $\BrigLtilde^+\submodule$ of $\Bcrys(\OLinfty) \otimes_L \ODcrysL(V)$.
	Now, the stability of $\BrigLtilde^+ \otimes_{\BLbreve^+} \NLbreve(V)$ under the $\Gamma_L\action$ follows from Proposition \ref{prop:nlinf_gammastab}.
\end{proof}

Recall that from Definition \ref{defi:mldl}, we have a $\BrigL^+\submodule$ $\ML(\ODcrysL(V)) \subset \MLbreve(\DcrysLbreve(V))$ stable under the action of $(\varphi, \GammaLbreve)$, and from Lemma \ref{lem:nlbreve_wach}, we have a $\BrigLbreve^+\linear$ and $(\varphi, \GammaLbreve)\equivariant$ isomorphism $\beta: \MLbreve(\DcrysLbreve(V)) \isomorphic \NrigLbreve(V)$.
Define a $\BrigL^+\submodule$ of $\NrigLbreve(V)$ as,
\begin{equation}\label{eq:nriglv}
	\NrigL(V) := \beta(\ML(\ODcrysL(V))) \subset \NrigLbreve(V).
\end{equation}
Since the map $\BrigL^+ \rightarrow \BrigLbreve^+$ constructed in \S \ref{subsubsec:periodrings_Lbreve} is $(\varphi, \GammaLbreve)\equivariant$, therefore, from \eqref{eq:nriglv} we obtain a $\BrigL^+\linear$ and $(\varphi, \GammaLbreve)\equivariant$ isomorphism $\beta: \ML(\ODcrysL(V)) \isomorphic \NrigL(V)$.
In particular, from Lemma \ref{lem:descent_brigl+} (3), we get that $\NrigL(V)$ is a finite free $\BrigL^+\module$ of rank $=\dim_{\QQ_p} V$, and the natural map $\BrigLbreve^+ \otimes_{\BrigL^+} \NrigL(V) \rightarrow \NrigLbreve(V)$ is a $(\varphi, \GammaLbreve)\equivariant$ isomorphism, since $\beta$ is $(\varphi, \GammaLbreve)\equivariant$.
Moreover, from Lemma \ref{lem:mldl_props}, it follows that $\NrigL(V)$ is of finite $\pqheight$ and pure of slope 0.
Now, consider the following diagram:
\begin{equation}\label{eq:nrigl_dcrys_comp}
	\begin{tikzcd}
		\ML(\ODcrysL(V))[\mu/t] \arrow[r, "\sim"] \arrow[rr, bend left=12, "\sim"', "\beta"] \arrow[d] & \BrigL^+[\mu/t] \otimes_L \ODcrysL(V) \arrow[d] & \NrigL(V)[\mu/t] \arrow[l, "\sim"'] \arrow[d]\\
		\MLbreve(\DcrysLbreve(V))[\mu/t] \arrow[r, "\sim"] \arrow[rr, bend right=12, "\sim", "\beta"'] & \BrigLbreve^+[\mu/t] \otimes_{\Lbreve} \DcrysLbreve(V) & \NrigLbreve(V)[\mu/t] \arrow[l, "\sim"'].
	\end{tikzcd}
\end{equation}
In the diagram \eqref{eq:nrigl_dcrys_comp}, all vertical arrows are natural inclusions.
In the bottom row, the left to the right horizontal arrow is the inverse of the composition of the lower horizontal and the left vertical arrow of diagram \eqref{eq:alpha_beta_commute}, the right to the left horizontal arrow is the inverse of \eqref{eq:wach_crys_rigcomp}, the curved arrow is the map $\beta$ in Lemma \ref{lem:nlbreve_wach} and the resulting triangle commutes by diagram \eqref{eq:alpha_beta_commute}.
In the top row, the left to the right horizontal arrow is the isomorphism in \eqref{eq:mldl_dcrys_comp}, the curved arrow is from \eqref{eq:nriglv}, the right to the left horizontal arrow is the composition of the inverse of $\beta$ with the inverse of \eqref{eq:mldl_dcrys_comp} and the resulting triangle commutes by definition.
Moreover, the two inner squares commute by definition and all maps are $(\varphi, \GammaLbreve)\equivariant$.

Using the diagram \eqref{eq:nrigl_dcrys_comp} and Defintion \ref{defi:mldl}, we can write 
\begin{equation*}
	\NrigL(V) = (\BrigL^+[\mu/t] \otimes_L \ODcrysL(V)) \cap \NrigLbreve(V) \subset \BrigLbreve^+[\mu/t] \otimes_{\Lbreve} \DcrysLbreve(V),
\end{equation*}
in particular, we will now consider $\NrigL(V)$ as a $\BrigL^+\submodule$ of $\BrigL^+[\mu/t] \otimes_L \ODcrysL(V)$.
Furthermore, from Lemma \ref{lem:odcrys_gammal_action}, recall that $\BrigL^+ \otimes_L \ODcrysL(V) \subset \Bcrys^+(\OLinfty) \otimes_L \ODcrysL(V)$ is stable under the action of $\Gamma_L$ and we equip the former with a $\Gamma_L\action$ induced from the latter.
Since we have that $g(t) = \chi(g) t$ and $g(\mu) = (1+\mu)^{\chi(g)} - 1$, for any $g \in \Gamma_L$ and $\chi$ the $\padic$ cyclotomic character, therefore, the preceding $\Gamma_L\action$ extends to $\BrigL^+[\mu/t] \otimes_L \ODcrysL(V)$.
\begin{prop}\label{prop:nrigl_gammastab}
	The $\BrigL^+\textrm{-submodule}$ $\NrigL(V) \subset \BrigL^+[\mu/t] \otimes_L \ODcrysL(V)$ is stable under the action of $\Gamma_L$.
	Moreover, the preceding inclusion extends to a $\BrigL^+[\mu/t]\linear$ and $(\varphi, \Gamma_L)\textrm{-compatible}$ isomorphism,
	\begin{equation}\label{eq:brigl_dcrys_wach_comp}
		\BrigL^+[\mu/t] \otimes_{\BrigL^+} \NrigL(V) \isomorphic \BrigL^+[\mu/t] \otimes_L \ODcrysL(V).
	\end{equation}
\end{prop}
\begin{proof}
	From Corollary \ref{cor:nlrig_gammastab} and the discussion after \eqref{eq:nriglv}, we have that $\BrigLtilde^+ \otimes_{\BrigLbreve^+} \NrigLbreve(V) \isomorphic \BrigLtilde^+ \otimes_{\BrigL^+} \NrigL(V)$, is stable under the action of $\Gamma_L$ on $\Bcrys(\OLinfty) \otimes_L \ODcrysL(V)$.
	Moreover, using Lemma \ref{lem:odcrys_gammal_action} and the discussion after \eqref{eq:nrigl_dcrys_comp}, we have a $\Gamma_L\equivariant$ embedding $\BrigL^+[\mu/t] \otimes_L \ODcrysL(V) \subset \Bcrys(\OLinfty) \otimes_L \ODcrysL(V)$.
	Therefore, inside $\Bcrys(\OLinfty) \otimes_L \ODcrysL(V)$, the following intersection is stable under the $\Gamma_L\action$:
	\begin{align*}
		\big(\BrigLtilde^+ \otimes_{\BrigL^+} \NrigL(V)\big) &\cap \big(\BrigL^+[\mu/t] \otimes_L \ODcrysL(V)\big)\\
		&= \big(\BrigLtilde^+ \otimes_{\BrigL^+} \NrigL(V)\big) \cap \big(\BrigL^+[\mu/t] \otimes_{\BrigL^+} \NrigL(V)\big)\\
		&= (\BrigLtilde^+ \cap \BrigL^+[\mu/t]) \otimes_{\BrigL^+} \NrigL(V) = \NrigL(V).
	\end{align*}
	The first equality follows from \eqref{eq:nrigl_dcrys_comp} and the second equality follows from Lemma \ref{lem:brigl_tilde_t} and the fact that $\NrigL(V)$ is finite free over $\BrigL^+$.
	This proves the first claim.
	For the second claim, note that by definition, the $\BrigL^+[\mu/t]\linear$ extension of the $(\varphi, \Gamma_L)\equivariant$ inclusion $\NrigL(V) \subset \BrigL^+[\mu/t] \otimes_L \ODcrysL(V)$, coincides with the top right horizontal arrow of the diagram \eqref{eq:nrigl_dcrys_comp}.
	Hence, the isomorphism in \eqref{eq:brigl_dcrys_wach_comp} follows.
\end{proof}

\begin{cor}\label{cor:nrigl_gamma_modmu}
	The action of $\Gamma_L$ on $\NrigL(V)$ is trivial modulo $\mu$.
\end{cor}
\begin{proof}
	Note that we have $g(\mu) = (1+\mu)^{\chi(g)}-1$ and $g(t) = \chi(g) t$, for any $g \in \Gamma_L$ and $\chi$ the $\padic$ cyclotomic character, in particular, $(g-1)(\mu/t) = \mu u_g(\mu/t)$, for some $u_g \in \BL^+$.
	Therefore, using Lemma \ref{lem:brigodcris_modmu} it follows that the action of $\Gamma_L$ is trivial modulo $\mu$ on $\BrigL^+[\mu/t] \otimes_L \ODcrysL(V) \lisomorphic \BrigL^+[\mu/t] \otimes_{\BrigL^+} \NrigL(V)$ (see \eqref{eq:brigl_dcrys_wach_comp}).
	Next, from Proposition \ref{prop:nrigl_gammastab}, note that we have a $(\varphi, \Gamma_L)\equivariant$ inclusion $\NrigL(V) \subset \BrigL^+[\mu/t] \otimes_{\BrigL^+} \NrigL(V) \isomorphic \BrigL^+[\mu/t] \otimes_L \ODcrysL(V)$.
	Let $x$ be in $\NrigL(V)$, then for any $g \in \Gamma_L$, we have that $(g-1)x$ is in $\NrigL(V) \subset \NrigLbreve(V)$ and $(g-1)x$ is also in $\mu \BrigL^+[\mu/t] \otimes_{\BrigL^+} \NrigL(V)$.
	Now, inside $\NrigLbreve(V)[\mu/t]$, we have that,
	\begin{align*}
		\NrigLbreve(V) &\cap \big(\mu \BrigL^+[\mu/t] \otimes_{\BrigL^+} \NrigL(V)\big) \\
		&= \big(\BrigLbreve^+ \otimes_{\BrigL^+} \NrigL(V)\big) \cap \big(\mu \BrigL^+[\mu/t] \otimes_{\BrigL^+} \NrigL(V)\big) \\
		&= (\BrigLbreve^+ \cap \mu \BrigL^+[\mu/t]) \otimes_{\BrigL^+} \NrigL(V) = \mu \NrigL(V),
	\end{align*}
	where the first equality follows from the isomorphism $\BrigLbreve^+ \otimes_{\BrigL^+} \NrigL(V) \isomorphic \NrigLbreve(V)$ (see the discussion after \eqref{eq:nriglv}), the second equality follows since $\NrigL(V)$ is free over $\BrigL^+$ and the last equality follows from Lemma \ref{lem:brigl_lbreve_intersect}.
	Hence, we conclude that $(g-1)x$ is in $\mu \NrigL(V)$, for any $x$ in $\NrigL(V)$ and $g \in \Gamma_L$.
\end{proof}

\subsection{Compatibility with \texorpdfstring{$(\varphi, \Gamma_L)\modules$}{-}}\label{subsec:compatibility_phigammmod}

From \S \ref{subsec:phigamma_mod_imperfect}, recall that $\DrigLdag(V)$ is a pure of slope 0 finite free $(\varphi, \Gamma_L)\module$ over $\BrigLdag$, functorially attached to $V$.
The following result is a generalisation of \cite[Proposition 3.5 \& Th\'eor\`eme 3.6]{berger-differentielles} from the perfect residue field case to $L$:
\begin{prop}\label{prop:dcrys_brigtilde_comp}
	There are natural $(\varphi, G_L)\equivariant$ isomorphisms,
	\begin{enumarabicup}
	\item $\Brigtilde^+[1/t] \otimes_L \ODcrysL(V) \isomorphic \Brigtilde^+[1/t] \otimes_{\QQ_p} V$.

	\item $\Brigtildedag[1/t] \otimes_L \ODcrysL(V) \isomorphic \Brigtildedag[1/t] \otimes_{\BrigLdag} \DrigLdag(V)$.
	\end{enumarabicup}
\end{prop}
\begin{proof}
	For (1), recall that from Lemma \ref{lem:brigtilde_odcrys_glact}, there is a $\Brigtilde^+\textrm{-linear}$ and $(\varphi, G_L)\equivariant$ map,
	\begin{equation*}
		\Brigtilde^+ \otimes_L \ODcrysL(V) \longrightarrow \Bcrys(\OLbar) \otimes_L \ODcrysL(V) \isomorphic \Bcrys(\OLbar) \otimes_{\QQ_p} V,
	\end{equation*}
	where the isomorphism is from Lemma \ref{lem:faltings_isomorphism_dcrys}.
	Extending the isomorphism in \eqref{eq:wach_crys_rigcomp} along $\BrigLtilde^+[\mu/t] \rightarrow \Brigtilde^+[1/t]$ and using \eqref{eq:dcrys_llbreve_comp}, we obtain a $\varphi\equivariant$ isomorphism,
	\begin{equation*}
		\Brigtilde^+[1/t] \otimes_L \ODcrysL(V) \isomorphic \Brigtilde^+[1/t] \otimes_{\Lbreve} \DcrysLbreve(V) \isomorphic \Brigtilde^+[1/t] \otimes_{\BLbreve^+} \NLbreve(V).
	\end{equation*}
	The preceding isomorphism fits into a commutative diagram compatibly with \eqref{eq:galois_action},
	\begin{equation}\label{eq:odcrys_wach_rep_comp}
		\begin{tikzcd}
			\Brigtilde^+[1/t] \otimes_L \ODcrysL(V) \arrow[d, "\wr"] \arrow[r] & \Bcrys(\OLbar) \otimes_L \ODcrysL(V) \arrow[rd, "\sim"]\\
			\Brigtilde^+[1/t] \otimes_{\BLbreve^+} \NLbreve(V) \arrow[r, "\sim"] & \Brigtilde^+[1/t] \otimes_{\QQ_p} V \arrow[r] & \Bcrys(\OLbar) \otimes_{\QQ_p} V,
		\end{tikzcd}
	\end{equation}
	where the left horizontal arrow in the bottom row is induced from the isomorphism $\Ainf(\OLbar)[1/\mu] \otimes_{\ALbreve^+} \NLbreve(T) \isomorphic \Ainf(\OLbar)[1/\mu] \otimes_{\ZZ_p} T$ (see Lemma \ref{lem:wach_props} (2)), the slanted isomorphism is the isomorphism in the third row of \eqref{eq:galois_action} and the rest are natural injective maps.
	Since the slanted isomorphism is $(\varphi, G_L)\equivariant$, therefore, we obtain that the isomorphism $\Brigtilde^+[1/t] \otimes_L \ODcrysL(V) \isomorphic \Brigtilde^+[1/t] \otimes_{\QQ_p} V$ is $(\varphi, G_L)\equivariant$, showing (1).
	For (2), by extending the isomorphism in (1) along $\Brigtilde^+[1/t] \rightarrow \Brigtildedag[1/t]$ and using \eqref{eq:brigbatdag_iso}, we obtain $(\varphi, G_L)\equivariant$ isomorphisms,
	\begin{equation*}
		\Brigtildedag[1/t] \otimes_L \ODcrysL(V) \isomorphic\Brigtildedag[1/t] \otimes_{\QQ_p} V \isomorphic \Brigtildedag[1/t] \otimes_{\BrigLdag} \DrigLdag(V).\qedhere
	\end{equation*}
\end{proof}

From the discussion after \eqref{eq:nriglv} and Proposition \ref{prop:nrigl_gammastab}, we have that $\BrigLdag \otimes_{\BrigL^+} \NrigL(V)$ is a pure of slope 0 finite free $(\varphi, \Gamma_L)\module$ over $\BrigLdag$ of rank $=\dim_{\QQ_p} V$.
Therefore, by the equivalence of categories in \cite[Lemma 4.5.7]{ohkubo-differential}, there exists a unique finite free \'etale $(\varphi, \Gamma_L)\module$ $\dldag$ over $\BLdag$ of rank $= \dim_{\QQ_p} V$ such that $\BrigLdag \otimes_{\BrigL^+} \NrigL(V) \isomorphic \BrigLdag \otimes_{\BLdag} \dldag$ compatible with $(\varphi, \Gamma_L)\action$.
\begin{cor}\label{cor:nrig_dag_comp}
	There exists a natural $(\varphi, G_L)\equivariant$ isomorphism $\Brigtildedag \otimes_{\BrigL^+} \NrigL(V) \isomorphic \Brigtildedag \otimes_{\BLdag} V$ inducing natural $(\varphi, \Gamma_L)\equivariant$ isomorphisms $\dldag \isomorphic \DLdag(V)$ and $\BrigLdag \otimes_{\BrigL^+} \NrigL(V) \isomorphic \BrigLdag \otimes_{\BLdag} \DLdag(V)$.
\end{cor}
\begin{proof}
	Consider the following diagram:
	\begin{center}
		\begin{tikzcd}
			\Brigtildedag \otimes_{\BrigL^+} \NrigL(V) \arrow[r, "\sim"] \arrow[d] & \Brigtildedag \otimes_{\BrigLbreve^+} \NrigLbreve(V) \arrow[r, "\sim"] \arrow[d] & \Brigtildedag \otimes_{\QQ_p} V \arrow[d]\\
			\Brigtildedag[1/t] \otimes_{\BrigL^+} \NrigL(V) \arrow[r, "\sim"] & \Brigtildedag[1/t] \otimes_{L} \ODcrysL(V) \arrow[r, "\sim"] & \Brigtildedag[1/t] \otimes_{\QQ_p} V.
		\end{tikzcd}
	\end{center}
	In the top row, the left horizontal arrow is induced by the isomorphism $\BrigLbreve^+ \otimes \NrigL(V) \isomorphic \NrigLbreve(V)$ (see the discussion after \eqref{eq:nriglv}) and the right horizontal arrow is induced by the isomorphism $\Ainf(\OLbar)[1/\mu] \otimes_{\ALbreve^+} \NLbreve(T) \isomorphic \Ainf(\OLbar)[1/\mu] \otimes_{\ZZ_p} T$ (see Lemma \ref{lem:wach_props} (2)).
	In the bottom row, the left horizontal arrow is induced by the $(\varphi, \Gamma_L)\equivariant$ isomorphism $\BrigL^+[\mu/t] \otimes_{\BrigL^+} \NrigL(V) \isomorphic \BrigL^+[\mu/t] \otimes_L \ODcrysL(V)$ (see \eqref{eq:brigl_dcrys_wach_comp} in Proposition \ref{prop:nrigl_gammastab}) and the right horizontal arrow is induced from Propositon \ref{prop:dcrys_brigtilde_comp} (1).
	The left and the right vertical arrows are natural maps and the middle vertical arrow is induced from \eqref{eq:wach_crys_rigcomp} and \eqref{eq:dcrys_llbreve_comp}.
	Commutativity of the left square follows from \eqref{eq:nrigl_dcrys_comp} and commutativity of the right square follows from \eqref{eq:odcrys_wach_rep_comp}.
	This shows the first claim.

	For the second claim, set $V' := (\Btildedag \otimes_{\BLdag} \dldag)^{\varphi=1}$, and note that it is a $\padic$ representation of $G_L$ with $\dim_{\QQ_p} V' = \dim_{\QQ_p} V$ (see \cite[Th\'eor\`eme 4.35]{andreatta-brinon}).
	Moreover, we have that $V' \subset (\Brigtildedag \otimes_{\BLdag} \dldag)^{\varphi=1} \isomorphic (\Brigtildedag \otimes_{\BrigL^+} \NrigL(V))^{\varphi=1} \isomorphic (\Brigtildedag \otimes_{\QQ_p} V)^{\varphi=1} = V$, where the first isomorphism follows from the discussion before the statement of the claim above, the second isomorphism follows from the first claim proven in the previous papragraph, and the last equality follows from Lemma \ref{lem:brigtilde_frobfixed}.
	Therefore, $V' \isomorphic V$ as $G_L\textrm{-representations}$ and it implies that $\dldag = \DLdag(V') \isomorphic \DLdag(V)$ as \'etale $(\varphi, \Gamma_L)\modules$ over $\BLdag$.
	It is straightforward to verify that this isomorphism is compatible with the commutative diagram above.
	This concludes our proof.
\end{proof}

\begin{rem}\label{rem:brigldag_odcrys_dldag_comp}
	As indicated before Corollary \ref{cor:nrig_dag_comp}, for a $\padic$ crystalline representation of $V$, combining the $(\varphi, \Gamma_L)\equivariant$ isomorphism $\BrigLdag \otimes_{\BrigL^+} \NrigL(V) \isomorphic \BrigLdag \otimes_{\BLdag} \DLdag(V)$ together with the inverse of the isomorphism \eqref{eq:brigl_dcrys_wach_comp}, gives a $\BrigLdag\linear$ $(\varphi, \Gamma_L)\equivariant$ isomorphism,
	\begin{equation}\label{eq:brigldag_odcrys_dldag_comp}
		\BrigLdag \otimes_{L} \ODcrysL(V) \isomorphic \BrigLdag \otimes_{\BLdag} \DLdag(V).
	\end{equation}
	The isomorphism \eqref{eq:brigldag_odcrys_dldag_comp} generalises \cite[Proposition 3.7]{berger-differentielles} from the perfect residue field case to $L$.
\end{rem}

\subsection{Obtaining Wach module}\label{subsec:obtaining_wachmod}

The finite free $\BrigL^+\module$ $\NrigL(V)$ is of finite $\pqheight$ $s$ and pure of slope 0 (see Lemma \eqref{lem:mldl_props}), therefore, from Lemma \ref{lem:finiteheightslope0_equiv} (2) there exists a unique finite free $\BL^+\module$ of rank $= \dim_{\QQ_p} V$ and finite $\pqheight$ $s$, whose extension of scalars along $\BL^+ \rightarrow \BrigL^+$ gives $\NrigL(V)$.
In particular, from the proof of Lemma \ref{lem:finiteheightslope0_equiv} we note the following:
\begin{defi}\label{defi:nlv}
	Define $\NL(V) := \NrigL(V) \cap \DLdag(V) \subset \DrigLdag(V)$.
\end{defi}
The $\BL^+\module$ $\NL(V)$ is finite free of rank $=\dim_{\QQ_p} V$ and it is equipped with an induced Frobenius-semilinear endomorphism $\varphi$ such that the cokernel of the injective map $(1 \otimes \varphi) : \varphi^*(\NL(V)) \rightarrow \NL(V)$ is killed by $[p]_q^s$, since $\NrigL(V)$ is of finite $\pqheight$ $s$ and $1 \otimes \varphi : \varphi^*(\DLdag(V)) \isomorphic \DLdag(V)$.
Moreover, we have that $\NL(V) \subset \DL^+(V)$ because inside $\DrigLdag(V)$ we have,
\begin{align*}
	\NL(V) = \NrigL(V) \cap \DLdag(V) &\subset (\Brigtilde^+ \otimes_{\QQ_p} V)^{H_L} \cap (\Bdag \otimes_{\QQ_p} V)^{H_L}\\
	&\subset ((\Brigtilde^+ \otimes_{\QQ_p} V) \cap (\Bdag \otimes_{\QQ_p} V))^{H_L}\\
	&\subset ((\Brigtilde^+ \cap \Bdag) \otimes_{\QQ_p} V)^{H_L} = (\mbfb^+ \otimes_{\QQ_p} V)^{H_L} = \DL^+(V).
\end{align*}
Furthermore, since $\NrigL(V)$ and $\DLdag(V)$ are stable under the compatible action of $\Gamma_L$ (see Proposition \ref{prop:nrigl_gammastab} and Corollary \ref{cor:nrig_dag_comp}), we conclude that $\NL(V)$ is stable under the induced $\Gamma_L\action$.
In particular, from the preceding discussion and Lemma \ref{lem:finiteheightslope0_equiv}, we obtain $(\varphi, \Gamma_L)\equivariant$ isomorphisms,
\begin{equation}\label{eq:nlv_large}
	\BrigL^+ \otimes_{\BL^+} \NL(V) \isomorphic \NrigL(V) \hspace{2mm} \textrm{and} \hspace{2mm} \BLdag \otimes_{\BL^+} \NL(V) \isomorphic \DLdag(V).
\end{equation}
\begin{lem}\label{lem:nlv_gamma_modmu}
	The action of $\Gamma_L$ on $\NL(V)$ is trivial modulo $\mu$.
\end{lem}
\begin{proof}
	Let $g \in \Gamma_L$ and $x \in \NL(V)$.
	Then, $(g-1)x$ is in $\NL(V) \subset \DLdag(V)$.
	Moreover, from Corollary \ref{cor:nrigl_gamma_modmu}, we have that $(g-1)x$ is in $\mu\NrigL(V)$.
	Therefore, inside $\DrigLdag(V)$, by using \eqref{eq:nlv_large} we get that,
	\begin{equation*}
		(g-1)x \in \DLdag(V) \cap \mu\NrigL(V) = (\BLdag \cap \mu\BrigL^+) \otimes_{\BL^+} \NL(V) = \mu\NL(V),
	\end{equation*}
	as claimed.
\end{proof}

\begin{defi}\label{defi:nlt}
	Define the Wach module over $\AL^+ = \BL^+ \cap \AL \subset \BL$ as,
	\begin{equation*}
		\NL(T) := \NL(V) \cap \DL(T) \subset \DL(V).
	\end{equation*}
\end{defi}

\begin{proof}[Proof of Theorem \ref{thm:crys_fh_imperfect}]
	We will show that $\NL(T)$ from Definition \ref{defi:nlt} satisfies all axioms of Definition \ref{defi:finite_pqheight_imperfect}.
	From the definition, note that $\NL(T)$ is a finitely generated torsion-free $\AL^+\module$ and an elementary computation shows that $\NL(T) \cap \mu^n \NL(V) = \mu^n \NL(T)$, for all $n \in \NN$, in particular, $\NL(T)/\mu\NL(T)$ is $p\textrm{-torsion}$ free.
	Moreover, we have that $\NL(T)[1/p] = \NL(V)$, and a simple diagram chase shows that $(\NL(T)/p\NL(T))[\mu] = (\NL(T)/\mu\NL(T))[p] = 0$ and $(\AL \otimes_{\AL^+} \NL(T))/p(\AL \otimes_{\AL^+} \NL(T)) = (\NL(T)/p\NL(T))[1/\mu]$.
	So, we have that $\NL(T)/p^n\NL(T) \subset (\NL(T)/p^n\NL(T))[1/\mu] = (\AL \otimes_{\AL^+} \NL(T))/p^n(\AL \otimes_{\AL^+} \NL(T))$, for all $n \in \NN$, and therefore, $\NL(T) \cap p^n (\AL \otimes_{\AL^+} \NL(T)) = p^n \NL(T)$, in particular, $\NL(V) \cap (\AL \otimes_{\AL^+} \NL(T)) = \NL(T)$.
	Now, by using Remark \ref{rem:intersect_finitefree}, it follows that $\NL(T)$ is a finite free $\AL^+\module$ of rank $= \rank_{\BL^+} \NL(V) = \dim_{\QQ_p} V$.
	Alternatively, to get the preceding statement, one can also use \cite[Lemme II.1.3]{berger-limites} (the proof of loc.\ cit.\ does not require the residue field of discrete valuation base field, $L$ in our case, to be perfect).

	From the definition, it also follows that $\NL(T) \cap p^n\DL(T) = p^n \NL(T)$, in particular, we have that $\NL(T)/p^n\NL(T) \subset \DL(T)/p^n\DL(T)$, and therefore, $(\NL(T)/p^n\NL(T))[1/\mu] \subset \DL(T)/p^n\DL(T)$.
	So, we get that $(\AL \otimes_{\AL^+} \NL(T))/p^n(\AL \otimes_{\AL^+} \NL(T)) \subset \DL(T)/p^n\DL(T)$, or equivalently, $(\AL \otimes_{\AL^+} \NL(T)) \cap p^n \DL(T) = p^n(\AL \otimes_{\AL^+} \NL(T))$.
	Note that we have $(\AL \otimes_{\AL^+} \NL(T))[1/p] = \BL \otimes_{\BL^+} \NL(V) \isomorphic \DL(V)$, where the last isomorphism follows from \eqref{eq:nlv_large}.
	Therefore, we get that $\AL \otimes_{\AL^+} \NL(T) = \DL(T) \cap (\AL \otimes_{\AL^+} \NL(T))[1/p] \isomorphic \DL(T) \cap \DL(V) = \DL(T)$.

	Next, $\NL(T)$ is equipped with an induced Frobenius-semilinear endomorphism $\varphi$.
	We have that $\varphi : \AL^+ \rightarrow \AL^+$ is finite and faithfully flat of degree $p^{d+1}$ and $\varphi^*(\AL) \isomorphic \AL^+ \otimes_{\varphi, \AL^+} \AL$ and similarly $\varphi^*(\BL^+) \isomorphic \AL^+ \otimes_{\varphi, \AL^+} \BL^+$ (see \S \ref{subsubsec:phigammamod_rings_imperfect}).
	Therefore, we get that $\varphi^*(\NL(V)) = \BL^+ \otimes_{\varphi, \BL^+} \NL(V) \isomorphic \AL^+ \otimes_{\varphi, \AL^+} \NL(V)$ and $\varphi^*(\DL(T)) = \AL \otimes_{\varphi, \AL} \DL(T) \isomorphic \AL^+ \otimes_{\varphi, \AL^+} \DL(T)$.
	Then, it easily follows that $\varphi^*(\NL(T)) = \varphi^*(\NL(V)) \cap \varphi^*(\DL(T)) \subset \varphi^*(\DL(V))$.
	Now, since $1 \otimes \varphi$ is injective on $\varphi^*(\DL(V))$, $1 \otimes \varphi : \varphi^*(\DL(T)) \isomorphic \DL(T)$ and the cokernel of $1 \otimes \varphi : \varphi^*(\NL(V)) \rightarrow \NL(V)$ is killed by $[p]_q^s$, therefore, we get that the cokernel of the injective map $1 \otimes \varphi : \varphi^*(\NL(T)) \rightarrow \NL(T)$ is killed by $[p]_q^s$.
	Finally, note that $\NL(T)$ is equipped with an induced $\Gamma_L\action$ such that $\Gamma_L$ acts trivially on $\NL(T)/\mu \NL(T)$ (follows easily from Lemma \ref{lem:nlv_gamma_modmu}), and we have that $\AL \otimes_{\AL^+} \NL(T) \isomorphic \DL(T)$.
	Hence, we conclude that $T$ is of finite $\pqheight$.
\end{proof}

\begin{cor}\label{cor:wachmod_scalarext}
	There exists a natural isomorphism of \'etale $(\varphi, \GammaLbreve)\modules$ $\ALbreve \otimes_{\AL} \DL(T) \isomorphic \DLbreve(T)$ and a natural isomorphism of Wach modules $\ALbreve^+ \otimes_{\AL^+} \NL(T) \isomorphic \NLbreve(T)$.
\end{cor}
\begin{proof}
	Note that we have an injection of \'etale $(\varphi, \GammaLbreve)\modules$ $\ALbreve \otimes_{\AL} \DL(T) \subset \DLbreve(T)$ and isomorphisms $(W(\CC_L^{\flat}) \otimes_{\AL} \DL(T))^{\varphi=1} \isomorphic T \lisomorphic (W(\CC_L^{\flat}) \otimes_{\ALbreve} \DLbreve(T))^{\varphi=1}$ as $G_{\Lbreve}\textrm{-representations}$.
	So, we get that $\ALbreve \otimes_{\AL} \DL(T) \isomorphic \DLbreve(T)$.
	Furthermore, we have a $(\varphi, \GammaLbreve)\equivariant$ injection of Wach modules $\ALbreve^+ \otimes_{\AL^+} \NL(T) \subset \NLbreve(T)$.
	So, by the unniquess of a Wach module attached to $T$ (see Lemma \ref{lem:finiteheight_props_imperfect}), it follows that $\ALbreve^+ \otimes_{\AL^+} \NL(T) \isomorphic \NLbreve(T)$.
\end{proof}

\begin{proof}[Proof of Corollary \ref{cor:crystalline_wach_rat_equivalence_imperfect}]
	The equivalence of $\otimes\textrm{-categories}$ follows from Theorem \ref{thm:crys_fh_imperfect} and we are left to show the exactness of the functor $\NL$ since the exactness of the quasi-inverse functor follows from Proposition \ref{prop:wach_etale_ff_imperfect} and the exact equivalence in \eqref{eq:rep_phigamma_imperfect}.
	From \S \ref{subsubsec:periodrings_Lbreve}, recall that $\AL^+ \rightarrow \ALbreve^+$ is faithfully flat, therefore, $\BL^+ \rightarrow \BLbreve^+$ is faithfully flat.
	Moreover, for a $\padic$ crystalline representation $V$ of $G_L$, from Corollary \ref{cor:wachmod_scalarext}, note that we have $\BLbreve^+ \otimes_{\BL^+} \NL(V) \isomorphic \NLbreve(V)$.
	So, given an exact sequence,
	\begin{equation}\label{eq:crysrep_exactseq}
		0 \rightarrow V_1 \rightarrow V_2 \rightarrow V_3 \rightarrow 0, 
	\end{equation}
	 of $\padic$ crystalline representations of $G_L$, it is enough to show that the following sequence is exact:
	\begin{equation}\label{eq:nlbrevev_exact}
		0 \rightarrow \NLbreve(V_1) \rightarrow \NLbreve(V_2) \rightarrow \NLbreve(V_3) \rightarrow 0.
	\end{equation}
	Furthermore, note that \eqref{eq:crysrep_exactseq} is exact if and only if it is exact after tensoring with $\QQ_p(r)$, for any $r \in \ZZ$.
	Similarly, \eqref{eq:nlbrevev_exact} is exact if and only if it is exact after tensoring with $\mu^{-r} \BLbreve^+(r)$.
	So we may assume that \eqref{eq:crysrep_exactseq} is an exact sequence of positive crystalline representations, i.e.\ the Wach modules in \eqref{eq:nlbrevev_exact} are effective.
	Moreover, the map $\BLbreve^+ \rightarrow \BrigLbreve^+$ is faithfully flat (by an argument similar to Lemma \ref{lem:bl_brigl_ff}), so it is enough to show that the following sequence is exact:
	\begin{equation*}
		0 \rightarrow \NrigLbreve(V_1) \rightarrow \NrigLbreve(V_2) \rightarrow \NrigLbreve(V_3) \rightarrow 0.
	\end{equation*}
	Exactness of the preceding sequence follows from Lemma \ref{lem:nlbreve_wach}, \cite[Theorem 1.2.15]{kisin-modules}, \cite[Proposition 2.2.6]{kisin-ren} and the exactness of the functor $\DcrysLbreve$.
	This allows us to conclude.
\end{proof}

%% file: references.bib
@article{abhinandan-relative-wach-i,
     author = {Abhinandan},
     title = {Crystalline representations and {Wach} modules in the relative case},
     journal = {Annales de l'Institut Fourier},
     pages = {379--474},
     publisher = {Association des Annales de l{\textquoteright}institut Fourier},
     volume = {75},
     number = {1},
     year = {2025},
     doi = {10.5802/aif.3670},
     language = {en},
     url = {https://aif.centre-mersenne.org/articles/10.5802/aif.3670/}
}

@article{abhinandan-syntomic,
       author = {{Abhinandan}},
        title = "{Syntomic complex and $p$-adic nearby cycles}",
         year = 2023,
        month = aug,
          eid = {arXiv:2308.10736},
}

@article{abhinandan-relative-wach-ii,
       author = {{Abhinandan}},
        title = "{Crystalline representations and Wach modules in the relative case II}",
         year = 2023,
        month = sep,
          eid = {arXiv:2309.16446},
}

@article{abhinandan-prismatic-wach,
       author = {{Abhinandan}},
        title = "{Prismatic $F$-crystals and Wach modules}",
         year = 2024,
        month = may,
          eid = {arXiv:2405.18245},
}

@article {andreatta-phigamma,
    AUTHOR = {Andreatta, Fabrizio},
     TITLE = {Generalized ring of norms and generalized
              {$(\phi,\Gamma)$}-modules},
   JOURNAL = {Ann. Sci. \'{E}cole Norm. Sup. (4)},
  FJOURNAL = {Annales Scientifiques de l'\'{E}cole Normale Sup\'{e}rieure. Quatri\`eme
              S\'{e}rie},
    VOLUME = {39},
      YEAR = {2006},
    NUMBER = {4},
     PAGES = {599--647},
      ISSN = {0012-9593},
   MRCLASS = {12J10 (13F30)},
  MRNUMBER = {2290139},
MRREVIEWER = {Alan Koch},
       DOI = {10.1016/j.ansens.2006.07.003},
       URL = {https://doi.org/10.1016/j.ansens.2006.07.003},
}

@incollection {andreatta-brinon,
    AUTHOR = {Andreatta, Fabrizio and Brinon, Olivier},
     TITLE = {Surconvergence des repr\'{e}sentations {$p$}-adiques: le cas
              relatif},
      NOTE = {Repr\'{e}sentations $p$-adiques de groupes $p$-adiques. I.
              Repr\'{e}sentations galoisiennes et $(\phi,\Gamma)$-modules},
   JOURNAL = {Ast\'{e}risque},
  FJOURNAL = {Ast\'{e}risque},
    NUMBER = {319},
      YEAR = {2008},
     PAGES = {39--116},
      ISSN = {0303-1179},
      ISBN = {978-2-85629-256-3},
   MRCLASS = {11S15 (11F80 11S20)},
  MRNUMBER = {2493216},
MRREVIEWER = {Alan Koch},
}

@article {berger-limites,
    AUTHOR = {Berger, Laurent},
     TITLE = {Limites de repr\'{e}sentations cristallines},
   JOURNAL = {Compos. Math.},
  FJOURNAL = {Compositio Mathematica},
    VOLUME = {140},
      YEAR = {2004},
    NUMBER = {6},
     PAGES = {1473--1498},
      ISSN = {0010-437X},
   MRCLASS = {11S25 (11F80 11R23 13K05 14F30)},
  MRNUMBER = {2098398},
       DOI = {10.1112/S0010437X04000879},
       URL = {https://doi.org/10.1112/S0010437X04000879},
}

@article {berger-differentielles,
    AUTHOR = {Berger, Laurent},
     TITLE = {Repr\'{e}sentations {$p$}-adiques et \'{e}quations diff\'{e}rentielles},
   JOURNAL = {Invent. Math.},
  FJOURNAL = {Inventiones Mathematicae},
    VOLUME = {148},
      YEAR = {2002},
    NUMBER = {2},
     PAGES = {219--284},
      ISSN = {0020-9910},
   MRCLASS = {14F30 (11S20 12H25 14F40 14G20)},
  MRNUMBER = {1906150},
MRREVIEWER = {Adolfo Quir\'{o}s},
       DOI = {10.1007/s002220100202},
       URL = {https://doi.org/10.1007/s002220100202},
}

@article {bhatt-scholze-prisms,
    AUTHOR = {Bhatt, Bhargav and Scholze, Peter},
     TITLE = {Prisms and prismatic cohomology},
   JOURNAL = {Ann. of Math. (2)},
  FJOURNAL = {Annals of Mathematics. Second Series},
    VOLUME = {196},
      YEAR = {2022},
    NUMBER = {3},
     PAGES = {1135--1275},
      ISSN = {0003-486X,1939-8980},
   MRCLASS = {14F30 (14F20 14F40)},
  MRNUMBER = {4502597},
       DOI = {10.4007/annals.2022.196.3.5},
       URL = {https://doi.org/10.4007/annals.2022.196.3.5},
}

@article{bhatt-scholze-crystals,
    AUTHOR = {Bhatt, Bhargav and Scholze, Peter},
     TITLE = {Prismatic {$F$}-crystals and crystalline {G}alois
              representations},
   JOURNAL = {Camb. J. Math.},
  FJOURNAL = {Cambridge Journal of Mathematics},
    VOLUME = {11},
      YEAR = {2023},
    NUMBER = {2},
     PAGES = {507--562},
      ISSN = {2168-0930,2168-0949},
   MRCLASS = {14 (11F80 12 13 16)},
  MRNUMBER = {4600546},
}

@article {bhatt-morrow-scholze-1,
    AUTHOR = {Bhatt, Bhargav and Morrow, Matthew and Scholze, Peter},
     TITLE = {Integral {$p$}-adic {H}odge theory},
   JOURNAL = {Publ. Math. Inst. Hautes \'{E}tudes Sci.},
  FJOURNAL = {Publications Math\'{e}matiques. Institut de Hautes \'{E}tudes
              Scientifiques},
    VOLUME = {128},
      YEAR = {2018},
     PAGES = {219--397},
      ISSN = {0073-8301},
   MRCLASS = {14F30},
  MRNUMBER = {3905467},
MRREVIEWER = {Daniel Robert Gulotta},
       DOI = {10.1007/s10240-019-00102-z},
       URL = {https://doi.org/10.1007/s10240-019-00102-z},
}

@article{bhatt-morrow-scholze-2,
    AUTHOR = {Bhatt, Bhargav and Morrow, Matthew and Scholze, Peter},
     TITLE = {Topological {H}ochschild homology and integral {$p$}-adic
              {H}odge theory},
   JOURNAL = {Publ. Math. Inst. Hautes \'{E}tudes Sci.},
  FJOURNAL = {Publications Math\'{e}matiques. Institut de Hautes \'{E}tudes
              Scientifiques},
    VOLUME = {129},
      YEAR = {2019},
     PAGES = {199--310},
      ISSN = {0073-8301},
   MRCLASS = {14F30 (13A35)},
  MRNUMBER = {3949030},
MRREVIEWER = {Lance Edward Miller},
       DOI = {10.1007/s10240-019-00106-9},
       URL = {https://doi.org/10.1007/s10240-019-00106-9},
}

@book {bosch,
    AUTHOR = {Bosch, Siegfried},
     TITLE = {Lectures on formal and rigid geometry},
    SERIES = {Lecture Notes in Mathematics},
    VOLUME = {2105},
 PUBLISHER = {Springer, Cham},
      YEAR = {2014},
     PAGES = {viii+254},
      ISBN = {978-3-319-04416-3; 978-3-319-04417-0},
   MRCLASS = {14G22 (11S85 14-02 14D15)},
  MRNUMBER = {3309387},
MRREVIEWER = {Alessandra Bertapelle},
       DOI = {10.1007/978-3-319-04417-0},
       URL = {https://doi.org/10.1007/978-3-319-04417-0},
}

@article {brinon-imparfait,
    AUTHOR = {Brinon, Olivier},
     TITLE = {Repr\'{e}sentations cristallines dans le cas d'un corps r\'{e}siduel
              imparfait},
   JOURNAL = {Ann. Inst. Fourier (Grenoble)},
  FJOURNAL = {Universit\'{e} de Grenoble. Annales de l'Institut Fourier},
    VOLUME = {56},
      YEAR = {2006},
    NUMBER = {4},
     PAGES = {919--999},
      ISSN = {0373-0956},
   MRCLASS = {11S20 (11F80 11F85 11S15 11S25 14F30)},
  MRNUMBER = {2266883},
MRREVIEWER = {Laurent N. Berger},
       URL = {http://aif.cedram.org/item?id=AIF_2006__56_4_919_0},
}

@article {brinon-trihan,
    AUTHOR = {Brinon, Olivier and Trihan, Fabien},
     TITLE = {Repr\'{e}sentations cristallines et {$F$}-cristaux: le cas d'un
              corps r\'{e}siduel imparfait},
   JOURNAL = {Rend. Semin. Mat. Univ. Padova},
  FJOURNAL = {Rendiconti del Seminario Matematico della Universit\`a di
              Padova. Mathematical Journal of the University of Padua},
    VOLUME = {119},
      YEAR = {2008},
     PAGES = {141--171},
      ISSN = {0041-8994},
      ISBN = {978-88-7784-291-6},
   MRCLASS = {14F30 (11S20)},
  MRNUMBER = {2431507},
MRREVIEWER = {Adrian Vasiu},
       DOI = {10.4171/RSMUP/119-4},
       URL = {https://doi.org/10.4171/RSMUP/119-4},
}

@article{brinon-relatif,
    AUTHOR = {Brinon, Olivier},
     TITLE = {Repr\'{e}sentations {$p$}-adiques cristallines et de de {R}ham dans le cas relatif},
   JOURNAL = {M\'{e}m. Soc. Math. Fr. (N.S.)},
  FJOURNAL = {M\'{e}moires de la Soci\'{e}t\'{e} Math\'{e}matique de France. Nouvelle S\'{e}rie},
    NUMBER = {112},
      YEAR = {2008},
     PAGES = {vi+159},
      ISSN = {0249-633X},
      ISBN = {978-2-85629-250-1},
   MRCLASS = {14F30 (11S25 11S80)},
  MRNUMBER = {2484979},
MRREVIEWER = {Laurent N. Berger},
       DOI = {10.24033/msmf.424},
       URL = {https://doi.org/10.24033/msmf.424},
}

@book {cartan-eilenberg,
    AUTHOR = {Cartan, Henri and Eilenberg, Samuel},
     TITLE = {Homological algebra},
    SERIES = {Princeton Landmarks in Mathematics},
      NOTE = {With an appendix by David A. Buchsbaum,
              Reprint of the 1956 original},
 PUBLISHER = {Princeton University Press, Princeton, NJ},
      YEAR = {1999},
     PAGES = {xvi+390},
      ISBN = {0-691-04991-2},
   MRCLASS = {18Gxx (01A75)},
  MRNUMBER = {1731415},
}

@article {chang-diamond,
    AUTHOR = {Chang, Seunghwan and Diamond, Fred},
     TITLE = {Extensions of rank one {$(\phi,\Gamma)$}-modules and
              crystalline representations},
   JOURNAL = {Compos. Math.},
  FJOURNAL = {Compositio Mathematica},
    VOLUME = {147},
      YEAR = {2011},
    NUMBER = {2},
     PAGES = {375--427},
      ISSN = {0010-437X},
   MRCLASS = {11S25 (11F85)},
  MRNUMBER = {2776609},
MRREVIEWER = {Gabor Wiese},
       DOI = {10.1112/S0010437X1000504X},
       URL = {https://doi.org/10.1112/S0010437X1000504X},
}

@article {cherbonnier-colmez,
    AUTHOR = {Cherbonnier, F. and Colmez, P.},
     TITLE = {Repr\'{e}sentations {$p$}-adiques surconvergentes},
   JOURNAL = {Invent. Math.},
  FJOURNAL = {Inventiones Mathematicae},
    VOLUME = {133},
      YEAR = {1998},
    NUMBER = {3},
     PAGES = {581--611},
      ISSN = {0020-9910},
   MRCLASS = {11S15 (11S25)},
  MRNUMBER = {1645070},
MRREVIEWER = {Jean-Pierre Wintenberger},
       DOI = {10.1007/s002220050255},
       URL = {https://doi.org/10.1007/s002220050255},
}

@article {colmez-hauteur,
    AUTHOR = {Colmez, Pierre},
     TITLE = {Repr\'{e}sentations cristallines et repr\'{e}sentations de hauteur
              finie},
   JOURNAL = {J. Reine Angew. Math.},
  FJOURNAL = {Journal f\"{u}r die Reine und Angewandte Mathematik. [Crelle's
              Journal]},
    VOLUME = {514},
      YEAR = {1999},
     PAGES = {119--143},
      ISSN = {0075-4102},
   MRCLASS = {11S20 (11S25 14F30)},
  MRNUMBER = {1711279},
MRREVIEWER = {Abdellah Mokrane},
       DOI = {10.1515/crll.1999.068},
       URL = {https://doi.org/10.1515/crll.1999.068},
}

@article {colmez-banach,
    AUTHOR = {Colmez, Pierre},
     TITLE = {Espaces de {B}anach de dimension finie},
   JOURNAL = {J. Inst. Math. Jussieu},
  FJOURNAL = {Journal of the Institute of Mathematics of Jussieu. JIMJ.
              Journal de l'Institut de Math\'{e}matiques de Jussieu},
    VOLUME = {1},
      YEAR = {2002},
    NUMBER = {3},
     PAGES = {331--439},
      ISSN = {1474-7480},
   MRCLASS = {11S85 (46B99 46S10)},
  MRNUMBER = {1956055},
MRREVIEWER = {Mark Kisin},
       DOI = {10.1017/S1474748002000099},
       URL = {https://doi.org/10.1017/S1474748002000099},
}

@article{colmez-niziol,
    AUTHOR = {Colmez, Pierre and Nizio{\l} , Wies{\l}awa},
     TITLE = {Syntomic complexes and {$p$}-adic nearby cycles},
   JOURNAL = {Invent. Math.},
  FJOURNAL = {Inventiones Mathematicae},
    VOLUME = {208},
      YEAR = {2017},
    NUMBER = {1},
     PAGES = {1--108},
      ISSN = {0020-9910},
   MRCLASS = {14F30 (11S25 14F20 14F40 14G20 14G22)},
  MRNUMBER = {3621832},
MRREVIEWER = {Adolfo Quir\'{o}s},
       DOI = {10.1007/s00222-016-0683-3},
       URL = {https://doi.org/10.1007/s00222-016-0683-3},
}

@article{du-liu-moon-shimizu,
    AUTHOR = {Du, Heng and Liu, Tong and Moon, Yong Suk and Shimizu, Koji},
     TITLE = {Completed prismatic {$F$}-crystals and crystalline
              {$Z_p$}-local systems},
   JOURNAL = {Compos. Math.},
  FJOURNAL = {Compositio Mathematica},
    VOLUME = {160},
      YEAR = {2024},
    NUMBER = {5},
     PAGES = {1101--1166},
      ISSN = {0010-437X,1570-5846},
   MRCLASS = {14F30 (11F80 14G45)},
  MRNUMBER = {4733770},
       DOI = {10.1112/S0010437X24007097},
       URL = {https://doi.org/10.1112/S0010437X24007097},
}

@incollection {faltings-crystalline,
    AUTHOR = {Faltings, Gerd},
     TITLE = {Crystalline cohomology and {$p$}-adic
              {G}alois-representations},
 BOOKTITLE = {Algebraic analysis, geometry, and number theory ({B}altimore,
              {MD}, 1988)},
     PAGES = {25--80},
 PUBLISHER = {Johns Hopkins Univ. Press, Baltimore, MD},
      YEAR = {1989},
      ISBN = {0-8018-3841-X},
   MRCLASS = {14F30 (14F40)},
  MRNUMBER = {1463696},
MRREVIEWER = {Abdellah\ Mokrane},
}

@book {fontaine-pdivisibles,
    AUTHOR = {Fontaine, Jean-Marc},
     TITLE = {Groupes {$p$}-divisibles sur les corps locaux.},
    SERIES = {},
 PUBLISHER = {Soci\'{e}t\'{e} Math\'{e}matique de France, Paris,, },
      YEAR = {1977},
     PAGES = {i+262},
   MRCLASS = {14L05},
  MRNUMBER = {498610},
MRREVIEWER = {Loren\ D.\ Olson},
}

@article{fontaine-annals,
	TITLE = {Sur certains types de repr\'{e}sentations {$p$}-adiques du groupe de {G}alois d'un corps local; construction d'un anneau de {B}arsotti-{T}ate},
	AUTHOR = {Fontaine, Jean-Marc},
	JOURNAL = {Ann. of Math. (2)},
	FJOURNAL = {Annals of Mathematics. Second Series},
	VOLUME = {115},
	NUMBER = {3},
	PAGES = {529--577},
	YEAR = {1982},
	ISSN = {0003-486X},
	MRCLASS = {14F30 (12B20 14G20 14L05)},
	MRNUMBER = {657238},
	MRREVIEWER = {F. Baldassarri},
	DOI = {10.2307/2007012},
	URL = {https://doi.org/10.2307/2007012},
}

@incollection {fontaine-phigamma,
    AUTHOR = {Fontaine, Jean-Marc},
     TITLE = {Repr\'{e}sentations {$p$}-adiques des corps locaux. {I}},
 BOOKTITLE = {The {G}rothendieck {F}estschrift, {V}ol. {II}},
    SERIES = {Progr. Math.},
    VOLUME = {87},
     PAGES = {249--309},
 PUBLISHER = {Birkh\"{a}user Boston, Boston, MA},
      YEAR = {1990},
   MRCLASS = {11S23 (14F30 14L05)},
  MRNUMBER = {1106901},
MRREVIEWER = {Rutger Noot},
}

@article{fontaine-corps-periodes,
	AUTHOR = {Fontaine, Jean-Marc},
	TITLE = {Le corps des p\'{e}riodes {$p$}-adiques},
	NOTE = {With an appendix by Pierre Colmez, P\'{e}riodes $p$-adiques (Bures-sur-Yvette, 1988)},
	JOURNAL = {Ast\'{e}risque},
	FJOURNAL = {Ast\'{e}risque},
	NUMBER = {223},
	YEAR = {1994},
	PAGES = {59--111},
	ISSN = {0303-1179},
	MRCLASS = {11G25 (14F20 14F30 14F40)},
	MRNUMBER = {1293971},
	MRREVIEWER = {Adolfo Quir\'{o}s},
}

@article{gao,
       author = {{Gao}, Hui},
        title = "{Integral $p$-adic Hodge theory in the imperfect residue field case}",
         year = 2020,
        month = jul,
          eid = {arXiv:2007.06879},
}

@article {guo-reinecke,
    AUTHOR = {Guo, Haoyang and Reinecke, Emanuel},
     TITLE = {A prismatic approach to crystalline local systems},
   JOURNAL = {Invent. Math.},
  FJOURNAL = {Inventiones Mathematicae},
    VOLUME = {236},
      YEAR = {2024},
    NUMBER = {1},
     PAGES = {17--164},
      ISSN = {0020-9910,1432-1297},
   MRCLASS = {14F30},
  MRNUMBER = {4712864},
       DOI = {10.1007/s00222-024-01238-4},
       URL = {https://doi.org/10.1007/s00222-024-01238-4},
}

@article {helmer,
    AUTHOR = {Helmer, Olaf},
     TITLE = {The elementary divisor theorem for certain rings without chain
              condition},
   JOURNAL = {Bull. Amer. Math. Soc.},
  FJOURNAL = {Bulletin of the American Mathematical Society},
    VOLUME = {49},
      YEAR = {1943},
     PAGES = {225--236},
      ISSN = {0002-9904},
   MRCLASS = {09.1X},
  MRNUMBER = {7744},
MRREVIEWER = {I. Kaplansky},
       DOI = {10.1090/S0002-9904-1943-07886-X},
       URL = {https://doi.org/10.1090/S0002-9904-1943-07886-X},
}

@article{hyodo,
    AUTHOR = {Hyodo, Osamu},
     TITLE = {On the {H}odge-{T}ate decomposition in the imperfect residue
              field case},
   JOURNAL = {J. Reine Angew. Math.},
  FJOURNAL = {Journal f\"{u}r die Reine und Angewandte Mathematik. [Crelle's
              Journal]},
    VOLUME = {365},
      YEAR = {1986},
     PAGES = {97--113},
      ISSN = {0075-4102},
   MRCLASS = {14L05 (11G10 11G25 11S25)},
  MRNUMBER = {826154},
MRREVIEWER = {Kazuya Kato},
       DOI = {10.1515/crll.1986.365.97},
       URL = {https://doi.org/10.1515/crll.1986.365.97},
}

@article {kedlaya-monodromy,
    AUTHOR = {Kedlaya, Kiran S.},
     TITLE = {A {$p$}-adic local monodromy theorem},
   JOURNAL = {Ann. of Math. (2)},
  FJOURNAL = {Annals of Mathematics. Second Series},
    VOLUME = {160},
      YEAR = {2004},
    NUMBER = {1},
     PAGES = {93--184},
      ISSN = {0003-486X},
   MRCLASS = {14F30 (11G25 12H25)},
  MRNUMBER = {2119719},
MRREVIEWER = {Laurent N. Berger},
       DOI = {10.4007/annals.2004.160.93},
       URL = {https://doi.org/10.4007/annals.2004.160.93},
}

@article {kedlaya-slope,
    AUTHOR = {Kedlaya, Kiran S.},
     TITLE = {Slope filtrations revisited},
   JOURNAL = {Doc. Math.},
  FJOURNAL = {Documenta Mathematica},
    VOLUME = {10},
      YEAR = {2005},
     PAGES = {447--525},
      ISSN = {1431-0635},
   MRCLASS = {14F30 (11G25 12H25 14G22)},
  MRNUMBER = {2184462},
MRREVIEWER = {Adolfo Quir\'{o}s},
}

@incollection {kisin-modules,
    AUTHOR = {Kisin, Mark},
     TITLE = {Crystalline representations and {$F$}-crystals},
 BOOKTITLE = {Algebraic geometry and number theory},
    SERIES = {Progr. Math.},
    VOLUME = {253},
     PAGES = {459--496},
 PUBLISHER = {Birkh\"{a}user Boston, Boston, MA},
      YEAR = {2006},
   MRCLASS = {11S20 (14F30)},
  MRNUMBER = {2263197},
MRREVIEWER = {Martin C. Olsson},
       DOI = {10.1007/978-0-8176-4532-8\_7},
       URL = {https://doi.org/10.1007/978-0-8176-4532-8_7},
}

@article {kisin-ren,
    AUTHOR = {Kisin, Mark and Ren, Wei},
     TITLE = {Galois representations and {L}ubin-{T}ate groups},
   JOURNAL = {Doc. Math.},
  FJOURNAL = {Documenta Mathematica},
    VOLUME = {14},
      YEAR = {2009},
     PAGES = {441--461},
      ISSN = {1431-0635},
   MRCLASS = {11F80 (11S20 14F30)},
  MRNUMBER = {2565906},
MRREVIEWER = {Michael M. Schein},
}

@book {lang,
    AUTHOR = {Lang, Serge},
     TITLE = {Cyclotomic fields {I} and {II}},
    SERIES = {Graduate Texts in Mathematics},
    VOLUME = {121},
   EDITION = {second},
      NOTE = {With an appendix by Karl Rubin},
 PUBLISHER = {Springer-Verlag, New York},
      YEAR = {1990},
     PAGES = {xviii+433},
      ISBN = {0-387-96671-4},
   MRCLASS = {11-02 (11R18 11R20 11S40 11T22)},
  MRNUMBER = {1029028},
MRREVIEWER = {T. Mets\"{a}nkyl\"{a}},
       DOI = {10.1007/978-1-4612-0987-4},
       URL = {https://doi.org/10.1007/978-1-4612-0987-4},
}

@article{lazard,
    AUTHOR = {Lazard, Michel},
     TITLE = {Les z\'{e}ros des fonctions analytiques d'une variable sur un
              corps valu\'{e} complet},
   JOURNAL = {Inst. Hautes \'{E}tudes Sci. Publ. Math.},
  FJOURNAL = {Institut des Hautes \'{E}tudes Scientifiques. Publications
              Math\'{e}matiques},
    NUMBER = {14},
      YEAR = {1962},
     PAGES = {47--75},
      ISSN = {0073-8301},
   MRCLASS = {12.70},
  MRNUMBER = {152519},
MRREVIEWER = {H. R\"{h}rl},
       URL = {http://www.numdam.org/item?id=PMIHES_1962__14__47_0},
}

@ARTICLE{morrow-tsuji,
       author = {{Morrow}, Matthew and {Tsuji}, Takeshi},
        title = "{Generalised representations as q-connections in integral $p$-adic Hodge theory}",
         year = 2020,
        month = oct,
          eid = {arXiv:2010.04059},
}

@article{ohkubo-differential,
    AUTHOR = {Ohkubo, Shun},
     TITLE = {On differential modules associated to de {R}ham
              representations in the imperfect residue field case},
   JOURNAL = {Algebra Number Theory},
  FJOURNAL = {Algebra \& Number Theory},
    VOLUME = {9},
      YEAR = {2015},
    NUMBER = {8},
     PAGES = {1881--1954},
      ISSN = {1937-0652},
   MRCLASS = {11S15},
  MRNUMBER = {3418746},
MRREVIEWER = {Nigel Byott},
       DOI = {10.2140/ant.2015.9.1881},
       URL = {https://doi.org/10.2140/ant.2015.9.1881},
}

@article {ohkubo-monodromy,
    AUTHOR = {Ohkubo, Shun},
     TITLE = {The {$p$}-adic monodromy theorem in the imperfect residue
              field case},
   JOURNAL = {Algebra Number Theory},
  FJOURNAL = {Algebra \& Number Theory},
    VOLUME = {7},
      YEAR = {2013},
    NUMBER = {8},
     PAGES = {1977--2037},
      ISSN = {1937-0652},
   MRCLASS = {11F80 (11F85 11S15 11S20)},
  MRNUMBER = {3134041},
MRREVIEWER = {Claus Mazanti Sorensen},
       DOI = {10.2140/ant.2013.7.1977},
       URL = {https://doi.org/10.2140/ant.2013.7.1977},
}

@article{scholze-q-deformations,
    AUTHOR = {Scholze, Peter},
     TITLE = {Canonical {$q$}-deformations in arithmetic geometry},
   JOURNAL = {Ann. Fac. Sci. Toulouse Math. (6)},
  FJOURNAL = {Annales de la Facult\'{e} des Sciences de Toulouse. Math\'{e}matiques. S\'{e}rie 6},
    VOLUME = {26},
      YEAR = {2017},
    NUMBER = {5},
     PAGES = {1163--1192},
      ISSN = {0240-2963},
   MRCLASS = {14F30 (11G25 12H10 14F05 14F40)},
  MRNUMBER = {3746625},
MRREVIEWER = {Nobuo Tsuzuki},
       DOI = {10.5802/afst.1563},
       URL = {https://doi.org/10.5802/afst.1563},
}

@misc{stacks-project,
  author       = {The {Stacks project authors}},
  title        = {The Stacks project},
  howpublished = {\url{https://stacks.math.columbia.edu}},
  year         = {2023},
}

@article {tsuji-cst,
    AUTHOR = {Tsuji, Takeshi},
     TITLE = {{$p$}-adic \'{e}tale cohomology and crystalline cohomology in the
              semi-stable reduction case},
   JOURNAL = {Invent. Math.},
  FJOURNAL = {Inventiones Mathematicae},
    VOLUME = {137},
      YEAR = {1999},
    NUMBER = {2},
     PAGES = {233--411},
      ISSN = {0020-9910},
   MRCLASS = {14F30 (14F20)},
  MRNUMBER = {1705837},
MRREVIEWER = {Abdellah Mokrane},
       DOI = {10.1007/s002220050330},
       URL = {https://doi.org/10.1007/s002220050330},
}

@article{tsuji-crystalline,
	author = {Tsuji, Takeshi},
	title = {Crystalline sheaves and filtered convergent $F$-isocrystals on log schemes},
}

@article {wach-free,
    AUTHOR = {Wach, Nathalie},
     TITLE = {Repr\'{e}sentations {$p$}-adiques potentiellement cristallines},
   JOURNAL = {Bull. Soc. Math. France},
  FJOURNAL = {Bulletin de la Soci\'{e}t\'{e} Math\'{e}matique de France},
    VOLUME = {124},
      YEAR = {1996},
    NUMBER = {3},
     PAGES = {375--400},
      ISSN = {0037-9484},
   MRCLASS = {11S23 (14F30 14L05)},
  MRNUMBER = {1415732},
MRREVIEWER = {Abdellah Mokrane},
       URL = {http://www.numdam.org/item?id=BSMF_1996__124_3_375_0},
}

@article {wach-torsion,
    AUTHOR = {Wach, Nathalie},
     TITLE = {Repr\'{e}sentations cristallines de torsion},
   JOURNAL = {Compositio Math.},
  FJOURNAL = {Compositio Mathematica},
    VOLUME = {108},
      YEAR = {1997},
    NUMBER = {2},
     PAGES = {185--240},
      ISSN = {0010-437X},
   MRCLASS = {11S23 (14F30 14L05)},
  MRNUMBER = {1468834},
MRREVIEWER = {Abdellah Mokrane},
       DOI = {10.1023/A:1000108818774},
       URL = {https://doi.org/10.1023/A:1000108818774},
}
